\documentclass[10pt,reqno]{amsart}
\usepackage{amssymb,mathrsfs,color,mathtools,empheq, verbatim, epstopdf}
\usepackage{pinlabel}
\mathtoolsset{showonlyrefs}
\usepackage{hyperref} 

\usepackage{enumerate}
\usepackage{tensor}

\usepackage{graphicx}
\usepackage{xcolor} % A package to add color.
\usepackage{tensor}
\usepackage{slashed}
\usepackage{cite}

\usepackage[shortlabels]{enumitem}

\usepackage{geometry}

\def\bS {\mathbb{S}}

\newcommand{\tx}[1]{\mathrm{#1}}

\newcommand{\bs}[1]{\boldsymbol{#1}}
\newcommand{\vd}{\mathrm{d}}
\newcommand{\udr}{\,r\vd r}

\newcommand{\uln}[1]{{\underline{ #1 }}}

\definecolor{deepgreen}{cmyk}{1,0,1,0.5}

\newcommand{\A}{\mathcal{A}}

\newcommand{\E}{\mathcal{E}}

\newcommand{\HH}{\mathcal{H}}

\newcommand{\LL}{\mathcal{L}}

\newcommand{\M}{\mathcal{M}}

\newcommand{\N}{\mathbb{N}}
\newcommand{\R}{\mathbb{R}}
\newcommand{\Sp}{\mathbb{S}}
\newcommand{\Z}{\mathbb{Z}}

\newcommand{\al}{\alpha}
\newcommand{\be}{\beta}
\newcommand{\ga}{\gamma}

\newcommand{\om}{\omega}
\newcommand{\la}{\lambda}
\newcommand{\lam}{\lambda} 
\newcommand{\te}{\theta}

\newcommand{\s}{\sigma}

\newcommand{\si}{\varsigma}

\newcommand{\De}{\Delta}

\newcommand{\La}{\Lambda}
\newcommand{\Lam}{\Lambda}

\newcommand{\p}{\partial}
\newcommand{\na}{\nabla}

\makeatletter

\newcommand{\Rmnum}[1]{\expandafter\@slowromancap\romannumeral #1@}
\makeatother

\newcommand{\ti}{\widetilde}
\newcommand{\ba}{\overline}

\newcommand{\U}{\underline}

\newcommand{\ang}[1]{\left\langle{#1}\right\rangle}
\newcommand{\abs}[1]{\left\lvert{#1}\right\rvert}

\newcommand{\ant}[1]{\begin{align*}\begin{split} #1 \end{split}\end{align*}}
\newcommand{\EQ}[1]{\begin{equation}\begin{split} #1 \end{split}\end{equation}}
\newcommand{\pmat}[1]{\begin{pmatrix} #1 \end{pmatrix}}

\newcommand{\Del}[1]{}

\numberwithin{equation}{section}

\newtheorem{thm}{Theorem}[section]
\newtheorem{cor}[thm]{Corollary}
\newtheorem{lem}[thm]{Lemma}
\newtheorem{prop}[thm]{Proposition}

\theoremstyle{remark}
\newtheorem{claim}[thm]{Claim}
\newtheorem{rem}[thm]{Remark}
\newtheorem{defn}[thm]{Definition}

\newcommand{\mand}{{\ \ \text{and} \ \  }}

\newcommand{\mas}{{\ \ \text{as} \ \ }}

\newcommand{\uD}{\operatorname{D}}

\newcommand{\rest}{\!\!\restriction}

\definecolor{green}{rgb}{0,0.8,0} % Redefines the color green.
\newcommand{\ud}{\mathrm{d}}

\newcommand{\eps}{\epsilon}

\newcommand{\bfd}{{\bf d}}

\newcommand{\hbfd}{\widehat{\bfd}}

\newcommand{\calA}{\mathcal A}

\newcommand{\ula}{\underline{\lambda}}

\begin{document}

\title[Uniqueness of two-bubble wave maps]{Uniqueness of two-bubble wave maps}%The unique two-bubble threshold wave map}
\author{Jacek Jendrej}
\author{Andrew Lawrie}

\begin{abstract}
This is the second part of a two-paper series that establishes the uniqueness and regularity of a threshold energy wave map that does not scatter in both time directions. 

Consider the $\Sp^2$-valued equivariant energy critical wave maps equation on $\R^{1+2}$, with equivariance class $k \ge 4$.  It is known that every topologically trivial  wave map with energy less than twice that of the unique $k$-equivariant harmonic map $\bs Q_k$ scatters in both time directions. 
We study maps with precisely the threshold energy $\E =  2 \E(\bs Q_k)$. 

In the first part of the series we gave a refined construction of a threshold wave map that asymptotically decouples into a superposition of two harmonic maps (bubbles), one of which is concentrating in scale.  In this paper, we show that this solution is the \emph{unique} (up to the natural invariances of the equation) two-bubble wave map. Combined with our earlier work~\cite{JL1} we can now give an exact description of every threshold wave map. 

%We also establish additional regularity and dynamical properties of this solution, which forms a $2$-bubble in one time direction and scatters freely in the other. 
\end{abstract}

\thanks{J.Jendrej was supported by  ANR-18-CE40-0028 project ESSED.  A. Lawrie was supported by NSF grant DMS-1700127 and a Sloan Research Fellowship}

\maketitle

\section{Introduction}

This paper concerns wave maps  from the Minkowski space $\R^{1+2}_{t, x}$ into the two-sphere $\bS^2$, with k-equivariant symmetry. These are formal critical points of the Lagrangian action, 
\EQ{
\calA( \Psi)  = \frac{1}{2} \int_{\R^{1+2}_{t, x}} \Big( {-} \abs{\p_t \Psi(t, x)}^2 + \abs{\na \Psi(t, x)}^2 \Big) \, \ud x  \ud t, 
}
restricted to the class of maps  $\Psi: \R^{1+2}_{t, x} \to \Sp^2 \subset \R^3$ that take the form, 
\ant{
\Psi(t, r, \theta) = (u(t, r), k \te ) \hookrightarrow ( \sin u (t, r)\cos k\theta , \sin u(t, r) \sin  k \theta, \cos u(t, r)) \in \Sp^2 \subset \R^3,
}
for some fixed $k \in \N$. Here $u$ is the colatitude measured from the north pole of the sphere and the metric on $\Sp^2$ is given by $ds^2 = d u^2+ \sin^2 u\,  d \om^2$.  
We note that $(r, \te)$ are polar coordinates on $\R^2$, and $u(t, r)$ is  radially symmetric.

Wave maps are known  as nonlinear $\s$-models in high energy physics literature, see for example,~\cite{MS, GeGr17}. They satisfy a canonical example of a geometric wave equation -- it  simultaneously generalizes the free scalar wave equation to manifold valued maps and the classical harmonic maps equation to Lorentzian domains. The $2d$ case considered here is of particular interest, as the static solutions given by finite energy harmonic maps are amongst the simplest examples of topological solitons;  other examples include kinks in scalar field equations,  vortices in Ginzburg-Landau equations, magnetic monopoles, Skyrmions, and Yang-Mills instantons; see~\cite{MS}. 
Wave maps under $k$-equivariant symmetry possess intriguing features from the point of view of nonlinear dynamics, for example, bubbling harmonic maps, multi-soliton solutions, etc.,  in the relatively simple setting of a geometrically natural scalar semilinear wave equation. 
%balances  features 
%
%gives a natural setting in which to study dynamic such a bubbling 
%it admits several interesting features, such as nontrivial stationary harmonic maps which lead to interesting dynamics such as bubbling, which is the focus of this paper. 
For a more thorough presentation of the physical or geometric content of wave maps, see e.g.,~\cite{MS, ShSt00, GeGr17}.
%Here we will simply regard \eqref{eq:wmk} as a scalar semilinear problem.

%We study maps $\Psi$ satisfying $\Psi \circ \rho^k = \rho^k \circ \Psi$ for all rotations $\rho \in SO(2)$.
%We consider a subclass of such maps known as $k$-equivariant, corresponding  to those that take the  form
%\ant{
%\Psi(t, r, \theta) = (u(t, r), k \te ) \hookrightarrow ( \sin u \cos k\theta , \sin u \sin  k \theta, \cos u) \in \Sp^2 \subset \R^3,
%}
%where $u$ is the colatitude measured from the north pole of the sphere and the metric on $\Sp^2$ is given by $ds^2 = d u^2+ \sin^2 u\,  d \om^2$.  

The Cauchy problem for $k$-equivariant wave maps is given by
\EQ{ \label{eq:wmk}
\p_{t}^2 u -   \p_{r}^2 u  - \frac{1}{r}  \p_r u + k^2  \frac{\sin 2 u}{2r^2} &= 0, \\ 
(u(t_0), \partial_t u(t_0)) &= (u_0, \dot u_0), \quad t_0 \in \R.
}
The conserved energy is 
\EQ{ \label{eq:energy} 
\E( \bs u(t)) := 2 \pi \int_0^\infty \frac{1}{2}  \left((\p_t u)^2  + (\p_r u)^2 + k^2 \frac{\sin^2 u}{r^2} \right) \, r \, \ud r,
}
where we have used bold font to denote the vector 
$
\bs u (t) := (u(t), \p_t u(t)).
$
We will write vectors with two components as  $\bs v = (v, \dot v)$, noting that the notation $\dot v$ will not, in general, refer to a time derivative of $v$ but rather just to the second component of $\bs v$. With this notation~\eqref{eq:wmk} can be rephrased as the Hamiltonian system 
\EQ{ \label{eq:uham} 
\frac{\ud}{\ud t} \bs u(t) =  J \circ \uD \E( \bs u(t)),
}
where 
\EQ{ \label{eq:DE}
J = \pmat{ 0 &1 \\ -1 &0}, \quad \uD \E( \bs u(t))  =  \pmat{ - \De u(t)  + r^{-2} f(u(t)) \\ \p_t u(t) }.
}
Note that above we have introduced the notation, 
\EQ{ \label{eq:f-def} 
f( u):= k^2 \sin (2 u).
}
We remark that both~\eqref{eq:uham}  and~\eqref{eq:energy} are invariant under the scaling
\EQ{ \label{eq:uscale} 
\bs u(t,  \cdot) \mapsto \bs u(t/ \la,  \cdot)_{\la}  = (u(t/ \la,  \cdot/ \la), \la^{-1} \p_t u( t/ \la, \cdot/ \la)), \qquad \lam >0.
}
which makes this problem energy critical.

 It follows from~\eqref{eq:energy} that any regular $k$-equivariant initial data $\bs u_0$  of finite energy  must satisfy $\lim_{r \to 0}u_0( r) = m\pi$ and  $\lim_{r \to \infty}u_0(r)=n\pi$  for some $m,n \in \Z$. Since the smooth wave map flow depends continuously on the initial data these integers are fixed over any time interval $t\in I$ on which the solution is defined. This splits the energy space into disjoint classes indexed by the pair $(m, n)$ and it is natural to consider the Cauchy  problem~\eqref{eq:wmk} within a fixed class. These classes are related to the topological degree of the full map $\Psi(t): \R^2 \to  \Sp^2$. In particular, $k$-equivariant wave maps with $(m, n) = (0, 0)$ correspond to topologically trivial maps $\Psi$, whereas those with $(m, n) = (0, 1)$  are degree $= k$ maps. 
 
 The unique (up to scaling) $k$-equivariant harmonic map is given explicitly by 
\EQ{ \label{eq:Q-def} 
Q(r) := 2 \arctan (r^k),
}
 and we write, $\bs Q := (Q, 0)$. We note that $Q(r)$ has degree $=k$ and  it is a standard fact that $\bs Q$  minimizes the energy amongst all degree $k$ maps (see, e.g.,~\cite{JL1}) and in particular amongst $k$-equivariant maps with $(m, n) = (0, 1)$. It is not hard to show that
$
 \E( \bs Q ) = 4  \pi k .
 $

 In this paper we  consider topologically trivial $k$-equivariant wave maps, i.e., those with data $\bs u_0$ that satisfies $\lim_{r \to 0} u_0(r)  =\lim_{r \to \infty} u_0(r) = 0$. The natural function space in which to consider such solutions in the energy space, which comes with the norm, 
 \EQ{
  \| \bs u_0 \|_{\HH}^2 := \|u_0 \|_{H}^2 + \| \dot u_0 \|_{L^2}^2 := \int_0^\infty  \Big((\p_r u_0(r))^2 +  k^2 \frac{u_0(r)^2}{r^2}  \Big) \, r \, \ud r +  \int_0^\infty \dot u_0(r)^2 \, r \, \ud r .
 }
 Denoting by $\LL_0 := -\De + k^2 r^{-2}$ we remark that the $H$ norm of a smooth function $u_0$ can also expressed as 
$
  \| u_0  \|_H^2 = \ang{ \LL_0 u_0 \mid u_0},
 $
 where $\ang{f \mid g} := (2\pi)^{-1} \ang{ f \mid g}_{L^2( \R^2)}$ is the $L^2$ inner product. We use $\LL_0$ to define spaces of higher regularity, and we let $\HH^2$ denote the norm 
  %\Red{Check what is good def of $\HH^2$ based on coercivity statement}
\EQ{
\| \bs u_0 \|_{\HH^2}^2:= \| u_0 \|_{H^2}^2 +   \| \dot u_0 \|_H^2  := \ang{ \LL_0 u_0 \mid \LL_0 u_0}  + \ang{\LL_0 \dot u_0 \mid \dot u_0}.
}
We also require the following weighted norm, 
\EQ{
\| \bs u_0 \|_{\bs \Lam^{-1} \HH} := \| (r \p_r u_0, ( r \p_r + 1) \dot u_0) \|_{\HH} .
}
While $\bs Q \not \in \HH$, this solution to~\eqref{eq:wmk} still plays a significant role in the dynamics of solutions in $\HH$;  for example,  superpositions of two bubbles, i.e.,  $Q(r/ \lam) - Q(r/ \mu)$ for $ \lam \neq \mu$, are elements of $H$.

 \subsection{Sub-threshold theorems and bubbling}
 The regularity theory for energy critical wave maps has been extensively studied; \cite{CTZcpam, CTZduke, STZ92, STZ94, KlaMac93, KlaMac94, KlaMac95, KlaMac97, KlaSel97, KlaSel02, Tat98, Tao1, Tao2, Tat01, Kri04}. Recently, the focus has been on the nonlinear dynamics of solutions with large energy. A remarkable~\emph{sub-threshold theorem} was established in~\cite{ST1, ST2, KS, Tao37}: every wave map with energy less than that of the first nontrivial harmonic map is globally regular on $\R^{1+2}$ and scatters to a constant map. The role of the minimal harmonic map in the formulation of the sub-threshold theoem was first clarified by fundamental work of Struwe~\cite{Struwe}, who showed that the smooth equivariant wave map flow can only develop a singularity by concentrating energy at the tip of a light cone via the bubbling off of at least one non-trivial finite energy harmonic map.  Bubbling wave maps were first constructed in a series of influential works by  Krieger, Schlag, Tataru~\cite{KST}, Rodnianski, Sterbenz~\cite{RS}, and Rapha\"el, Rodnianski~\cite{RR}, with the latter work yielding a stable blow-up regime;  see also the recent work~\cite{KrMiao-Duke}  for stability properties of the solutions from~\cite{KST}, as well as~\cite{JLR1} for a classification of blowup solutions with a given radiation profile, and~\cite{Pil-19} for a construction of a new class of singular solutions that blow up in infinite time. In particular, all of these works demonstrate that blow up by bubbling can occur for maps with energy slightly above the ground-state harmonic map, which shows the sharpness of the sub-threshold theorem. 
 
The sub-threshold theorem can be refined by taking into account the topological degree of the map. Only topologically trivial maps can scatter to a constant map and it was shown in~\cite{CKLS1, LO1} that the correct threshold that ensures scattering is $\E < 2\E( \bs Q)$ (rather than $\E(\bs Q)$). The reasoning behind the number $2 \E(\bs Q)$ is as follows. The topological degree counts (with orientation) the number of times a map `wraps around' $\Sp^2$. If a harmonic map of degree $k$ bubbles off from a wave map $\Psi(t)$,  then, in order for $\Psi(t)$  to be degree zero,  it must also `unwrap' $k$ times away from the bubble. The minimum energy required for each wrapping  is $4 \pi k = \E(\bs Q)$. Thus the energy required for a degree zero map to form a bubble is  $\E \ge  8 \pi k   =  2\E(\bs Q)$. 

%Restricting our attention to $k$-equivariant maps, we observe that Struwe's bubbling theorem implies that  any singular wave map $\bs u(t)$  converges locally (in a sense) to $\bs Q$, which means $\E(\bs Q) = 4\pi k$ is required to form a singularity. Since bubbling means  If the map is topologically trivial 

%  In fact, with the additional observation that only topologically trivial maps can scatter to a constant map, the threshold theorem for scattering was refined in~\cite{CKLS1, LO1} as follows: every 
 
 \subsection{Main result: uniqueness of two-bubble wave maps} 
 
We consider topologically trivial $k$-equivariant maps with precisely the threshold energy $\E =  2 \E(\bs Q)$. Building on the work~\cite{JJ-AJM} of the first author and our work~\cite{JL1}, we can now give an exact description of every such map.  We show that for equivariance classes $k \ge 4$,  there is a \emph{unique} (up to the natural invariances up the equation) threshold wave map that does not scatter in both time directions. 
 
 Let $\bs u(t) : [T_0, \infty) \to \HH$ be a solution to~\eqref{eq:wmk} with $\E(\bs u) = 2 \E(\bs Q)$. We say $\bs u(t)$ is a \emph{two-bubble in
 forward time} if there exist $\iota \in \{+1, -1\}$
 and continuous functions $\lambda(t), \mu(t) > 0$ such that
 \EQ{ \label{eq:2-bub-def} 
\lim_{t \to \infty} \| ( u(t) - \iota( Q_{\la(t)} -  Q_{\mu(t)}), \p_t u(t))\|_{\HH} = 0, \quad \lambda(t) \ll \mu(t)\text{ as }t \to \infty.
}
A \emph{two-bubble in the backward time direction} is defined similarly. Here $Q_{\nu}$ denotes the scaling 
$
Q_{\nu}(r) := Q( r/ \nu). 
$
 In~\cite{JJ-AJM} the first author constructed a two-bubble in forward time. 
%Precisely, for each $k \ge 3$, there exists a solution $\bs u: [T_0, \infty) \to \HH$ of \eqref{eq:wmk}
%such that
%  \begin{equation}
%    \label{eq:2bub}
%    \lim_{t\to \infty}\big\|\bs u(t) - \big({-}\bs Q + \bs Q_{q_k |t|^{-\frac{2}{k-2}}}\big)\big\|_{\HH} = 0,
%  \end{equation}
%  where $q_k > 0$ is an explicit constant depending on $k$. 
   In~\cite{JL1} we showed  that the solution from~\cite{JJ-AJM} must be global and scattering in backwards time.
    In the companion paper~\cite{JL2-regularity} we gave a refined construction of a two-bubble in forward time, showing that it possesses additional regularity and decay, i.e., it lies in the space $ \HH \cap \HH^2 \cap \bs \Lam^{-1} \HH$.
   In this paper we show that \emph{there is only one $2$-bubble wave map} in each equivariance classes $k \ge 4$. 
   
   %A by product of our proof is that this solution possesses additional regularity and decay, i.e., it lies in the space $ \HH \cap \HH^2 \cap \bs \Lam \HH$. 
 
\begin{thm}[Uniqueness of $2$-bubble wave maps]  \label{t:main}  Let $k \ge 4$. There  exists a global-in-time solution $\bs u_c:  \R \to \HH \cap \HH^2 \cap \bs \Lam^{-1} \HH$ of \eqref{eq:wmk}
such that
  \begin{equation}
    \label{eq:2bub}
   \big\|\bs u_c(t) - \big({-}\bs Q + \bs Q_{q_k |t|^{-\frac{2}{k-2}}}\big)\big\|_{\HH}   \to 0   \mas t \to \infty, 
  \end{equation}
  where $q_k > 0$ is an explicit constant depending on $k$ (see~\eqref{eq:q-rho-ga}). 
%  Moreover, $\bs u_c(t)$ lies in the space,  
%  \EQ{
%  \bs u_c(t)  \in \HH \cap \HH^2 \cap \bs \Lam \HH.
%  } 
  
  Moreover, if $\bs u(t) \in \HH$ is any other $2$-bubble in forward time, then there exists $(t_0, \mu_0) \in \R \times (0, \infty)$ such that, 
  \EQ{
  \bs u(t) = \bs u_{c, t_0, \mu_0, \pm}(t) :=  \pm \Big(u_c( t- t_0, r/ \mu_0), \frac{1}{\mu_0} \p_t u_c(t- t_0, r/ \mu_0) \Big), 
  }
  i.e., $\bs u_c(t)$ is \emph{unique} up to sign, time translation, and scale. 

\end{thm} 

\begin{rem}
We note that~\cite[Theorem 1.6]{JL1} ensures that $\bs u_c(t)$ is global  and scatters freely in backwards time. 
\end{rem} 

 \begin{rem} 
 The solution $\bs u_c(t)$ from Theorem~\ref{t:main} was constructed in~\cite{JJ-AJM}. However, the proof of uniqueness given in Section~\ref{s:unique} requires more detailed information about $\bs u_c(t)$ than what is obtained via the methods in~\cite{JJ-AJM}. This refined construction of $\bs u_c(t)$ is carried out in the companion paper~\cite{JL2-regularity}, and is summarized in Theorem~\ref{t:refined} below. Of course only after Theorem~\ref{t:main}  is proved can we be sure that $\bs u_c(t)$ is the same solution found in~\cite{JJ-AJM}.  We note that the companion paper~\cite{JL2-regularity} contains the proof that $\bs u_c(t) \in \HH\cap \HH^2 \cap \bs \Lam^{-1} \HH$, as well as an expansion of the solution into profiles that decay up to the rate $t^{-\frac{3k-2}{k-2}}$, along with a precise dynamical characterization of the modulation parameters associated to each bubble; see the beginning of Section~\ref{s:refined} for a detailed statement. 
 \end{rem} 
 
% \begin{rem} 
%We highlight that $\bs u_c(t)$ lies in $ \HH^2$, i.e., it is smoother than a general finite energy solution. We explain in Section~\ref{s:unique} that the trajectory of $\bs u_c(t)$ naturally yields a $2$-dimensional invariant submanifold of $\HH$ for~\eqref{eq:wmk}. While we expect that this manifold is smooth, we do not pursue this here.  
%\end{rem} 

This result can be combined with the main theorem in~\cite{JL1} to obtain the following complete classification.
% Let us introduce the notation, 
%\EQ{
%\bs u_{c, t_0, \mu_0, \pm}(t) := \pm \Big(u_c( t- t_0, r/ \mu_0), \frac{1}{\mu_0} \p_t u_c(t- t_0, r/ \mu_0) \Big).
%}

\begin{thm}[Classification of $\E = 2 \E(\bs Q)$ wave maps] \label{t:class}  Fix any equivariance class $k \ge 4$.  Let $\bs u :(T_-, T_+) \to \HH$ be a solution to~\eqref{eq:wmk} such that
\EQ{
\E(\bs u) = 2 \E(\bs Q) = 8\pi k.
}
Then $T_- = -\infty$, $T_+ = +\infty,$ and one the following alternatives holds:
 \begin{itemize}[leftmargin=0.5cm]
\item $\bs u (t)$ scatters freely in both time directions 
\item  $\bs u(t)  = \bs u_{c, t_0, \mu_0, \pm}(t)$, for some $(t_0, \mu_0) \in \R \times (0, \infty)$. This solution  is a two-bubble in forward time and freely scattering in backwards time 
\item  $\bs u(t)  = ( u_{c, t_0, \mu_0, \pm}(-t), - \p_t u_{c, t_0, \mu_0, \pm}(-t))$, for some $(t_0, \mu_0) \in \R \times (0, \infty)$. This solution is a two-bubble in backwards time and  freely scattering  in forwards time and is given by time-reversing the solution from Theorem~\ref{t:main} 
 \end{itemize} 
 \end{thm}  
 
\begin{rem} Several of the conclusions in the statement of Theorem~\ref{t:class} were proved in~\cite[Theorem 1.6]{JL1}.  In that work we showed that any threshold solution that does not scatter  in some direction  \emph{must} be a two-bubble in that direction as in~\eqref{eq:2-bub-def}, and with rates $\lam(t), \mu(t)$ that are to leading order the same as the rates of $\bs u_c(t)$. Additionally, in~\cite{JL1} we solved the so-called \emph{collision problem} for this equation. We showed that any two bubble in forward time must scatter freely in backwards time, i.e., when scales of the bubbles become comparable, this `collision' completely annihilates the $2$-bubble structure and the entire solution becomes free radiation; see also~\cite{MM11, MM11-2, MM18}. Viewing the evolution of $\bs u_c(t)$ in forward time, this means that the $2$-bubble emerges from pure radiation, and constitutes an orbit connecting two different dynamical behaviors. 
%This type of behavior is expected to hold in general outside of the completely integrable setting. 

%Compare with the weaker classification prove in~\cite{JL1} -- that one of the main points proved there was that the only two options in one direction were scattering or a two bubble (i.e., soliton resolution). Also say explicitly that the scattering in the opposite direction was proved in~\cite{JL1} and that this solve the 'collision problem' for this equation . Also, time reversibility  of the equation and uniqueness give third bullet point. 
\end{rem} 

\begin{rem} Theorem~\ref{t:main} fits in a broader program  to classify solutions to nonlinear wave equations via their linear radiation. Note that the solution $\bs u_c(t)$ emits zero linear radiation as $t \to \infty$. For~\eqref{eq:wmk} the conjecture (soliton resolution) is that the only solutions with this property are the trivial solution $\bs u(t)  \equiv 0$ and pure multi-bubbles such as $\bs u_c(t)$. Theorem~\ref{t:main} says that  $\bs u_c(t)$ is the only solution with two bubbles that emits zero radiation forward-in-time. More generally, one can fix a linear forward radiation profile $\bs u_L(t)$ and ask if there are solutions $\bs u(t)$ to~\eqref{eq:wmk} that asymptotically decouple into a sum of bubbles plus the radiative wave $\bs u_L(t)$, and more ambitiously, for which $\bs u_L(t)$ can these be classified? 

This same type of perspective can be taken in the context of solutions that develop a singularity in finite time via bubbling, see~\cite{JLR1} where a classification is given in terms of a given finite time radiation profile $\bs u^*$, which is a weak limit of the solution as $t \to T_+< \infty$. 
\end{rem}

 \begin{rem} %\Red{actually proof in uniquness paper works fine for $k=2, 3$ -- alter this remark}. 
 
 We expect identical theorems to hold for the equivariance classes $k=2,3$. In fact, the argument used to prove uniqueness in this paper adapts easily to these cases. However, we only carry out the refined construction in~\cite{JL2-regularity} for $k \ge 4$. The proof given in that paper can be readily adapted to cover the $k=3$ case, but we avoided this due to a technical inconvenience to keep the exposition as simple as possible, see~\cite[Remark 1.4]{JL2-regularity}. The $k=2$ case is more delicate due to the failure $r \p_r Q(r)  \not \in H^*$. This introduces the need for cut-offs in the modulation analysis. This issue was confronted in~\cite{JL1}, but we also avoided it in~\cite{JL2-regularity} to keep the analysis as straightforward  as possible. Finally, the dynamics of non-scattering threshold solutions in the case $k=1$ is different -- there is blow up in finite time; see the recent paper by Rodriguez~\cite{R19}. However, we still expect an analogous uniqueness statement to hold in that setting; see e.g.,~\cite[Conjecture 1.9]{JLR1}. 
 \end{rem}

\begin{rem} 
One can compare/contrast Theorem~\ref{t:class} with the classification of $E = E(W)$ threshold solutions of the focusing energy critical wave equation by Duyckaerts and Merle~\cite{DM08} with data $(u_0, u_1)$  in the subset $(\dot H^1 \cap L^2) \times L^2$ of the energy space; see also~\cite{DM09} for the corresponding theorem for NLS. There $W$ is the ground state Aubin-Talenti solution and it is shown that every threshold solution either scatters in both time directions, exhibits ODE blow up in both directions, is equal to $W$,  or is one of two solutions  $W^{\pm}$; $W^-$ scatters freely in one direction and scatters to $W$ in the other, and $W^+$ exhibits finite time ODE blow up  in one direction and scatters to $W$ in the other.  One main difference here is that the non-scattering threshold solution $\bs u_c(t)$  contains  $2$ bubbles, one of which is concentrating, which significantly complicates the analysis.   
\end{rem} 

\begin{rem}[Strong vs. weak soliton interactions] 
The first uniqueness result for multi-solitons is due to Martel~\cite{Martel05} who constructed and proved uniqueness of $N$-soliton solutions to g-KdV with distinct, nontrivial velocities. We refer to the multi-solitons in that work as \emph{weakly interacting} since the leading order dynamics are given/determined by the internal motion of each individual soliton.   We emphasize here a distinction with Theorem~\ref{t:main}: the bubbles in  $\bs u_c(t)$ are  \emph{strongly interacting} in the sense that the dynamics are driven by nonlinear interactions between the two bubbles. %whereas in the weakly interacting regime the soliton interactions are negligible to main order.  

%One can also compare/contrast  Theorem~\ref{t:main} with other multi-soliton results such as the landmark works of Merle~\cite{Merle90}, Martel~\cite{Martel05}, and Martel, Merle~\cite{MM06}, which  constructed $N$-soliton solutions to g-KdV and NLS with distinct, nontrivial velocities;   see also~\cite{MaMeTs, CMM11, CM}.  Only the paper of Martel~\cite{Martel05} proves uniqueness.
%%To the authors' knowledge, only one of these included a uniqueness statement, namely~\cite{Martel05}, where Martel proved uniqueness of weakly interacting  $N$-solitons for g-KdV for each given set of distinct velocities. 
% We refer to the multi-solitons in all of those works as \emph{weakly interacting} since the leading order dynamics are given/determined by the internal motion of each individual soliton.   We emphasize here a distinction with Theorem~\ref{t:main}: the bubbles in  $\bs u_c(t)$ are  \emph{strongly interacting} in the sense that the dynamics are driven by nonlinear interactions between the two bubbles, whereas in the weakly interacting regime the soliton interactions are negligible to main order.  
% %These regimes (strong and weak interactions) are simply different, and we do not claim any increase in difficulty in accessing the strongly interacting one considered here.  
% There have been several recent $2$-soliton constructions in the strongly interacting regime; see e.g., ~\cite{JJ-APDE, moi-gkdv, Ngu-17, JM-19, JKL1}. 
% %We highlight the connection with our recent work~\cite{JKL1} below, which also includes a uniqueness result. 
\end{rem} 

\begin{rem}[Unique strongly interacting topological solitons] 
One may compare Theorem~\ref{t:main} with the authors' recent work with Kowalczyk, ~\cite{JKL1}, which establishes the existence and \emph{uniqueness} of strongly interacting  kink-antikink solutions to scalar field equations on the line (e.g., sine-Gordon and $\phi^4$-model). While the Theorem~\ref{t:main} and the main result in~\cite{JKL1}  are quite similar in nature (albeit for different equations), we develop a completely different technique in this paper to establish uniqueness. We explain the difference in more detail in Remark~\ref{r:kinks} below. 
\end{rem} 

\begin{rem}[Uniqueness theorems in the blow up setting] Finally, we mention two other uniqueness results for solutions with non-trivial dynamics in the blow-up setting, namely the pioneering work of Merle~\cite{Merle93} which proved the existence and uniqueness  (up to phase) of minimal mass blow up for the mass critical NLS, and the remarkable paper by Rapha\"el and Szeftel~\cite{RaSz11} which proved an analogous result for same equation, but with an  inhomogeneous nonlinearity (which precludes the use of the psuedo-conformal symmetry in the proof).  Several techniques used in this series of papers were inspired by~\cite{RaSz11}, although we emphasize the method we use to prove uniqueness is novel. 
\end{rem} 

\subsection{The existence and regularity of  two bubble wave maps}  \label{s:refined} 
%The purpose of this section is to derive more precise information about the $2$-bubble solution that was constructed in~\cite{JJ-AJM}. In fact, we will redo the construction from scratch and only after uniqueness is proved (Theorem~\ref{t:main}) will it be clear that the solution constructed here is the same one found in~\cite{JJ-AJM}. 
The starting point for the proof of Theorem~\ref{t:main} is the existence of the two-bubble wave map $\bs u_c(t)$ from Theorem~\ref{t:main} with a precise description of its dynamics and regularity. This is the content of the companion paper~\cite{JL2-regularity}. For the reader's convenience we review the main conclusions here. 

%We next give a precise description of the solution $\bs u_c(t)$ from Theorem~\ref{t:main} 

We begin by introducing some notation needed to state Theorem~\ref{t:refined} below.  We define, 
\EQ{ \label{eq:q-rho-ga} 
\rho_k &:= \Big( \frac{8k}{\pi}  \sin( \pi/k) \Big)^{\frac{1}{2}} , \quad 
 \gamma_k := \frac{k}{2} \rho_k^2, \quad q_k:=  \Big( \frac{k-2}{2} \rho_k\Big)^{-\frac{2}{k-2}}
 }
 We remark that $\rho_k^2 = 16 k \| \Lam Q \|_{L^2}^{-2}$.  
 Given a radial function $w: \R^2  \to \R$ we denote the $H$ and $L^2$  re-scalings as follows 
\EQ{ \label{eq:scale} 
w_\la(r) := w(r/ \la), \quad
w_{\ula}(r)  := \frac{1}{\la} w (r/ \la)
}
The corresponding infinitesimal generators  are given by 
\EQ{ \label{eq:LaLa0} 
&\La w:= -\frac{\partial}{\partial \lambda}\bigg|_{\lambda = 1} w_\la = r \p_r w  \quad (H  \,  \textrm{scaling}) \\
& \La_0 w:= -\frac{\partial}{\partial \lambda}\bigg|_{\lambda = 1} w_{\ula} = (1 + r \p_r ) w  \quad (L^2  \textrm{scaling})
}
Next, we define $C^{\infty}(0, \infty)$ functions $A, B, \ti B$ as the unique solutions to the equations, 
 \EQ{  \label{eq:ABB-def} 
 \LL A &= - \La_0 \La Q,  \quad 
 0=  \ang{A \mid \La Q} \\ 
 \LL B &= \gamma
 _k  \La Q - 4 r^{k-2} [\La Q]^2,  \quad 
0 =  \ang{B \mid \La Q},  \\
  \LL \ti B &= - \gamma_k \La Q + 4r^{-k-2} [\La Q]^2,  \quad 
  0 =  \ang{ \ti B \mid \La Q} 
  } 
  where here $\LL := -\De + r^{-2} f'(Q)$ is the operator obtained via linearization about $Q$. 
  These are constructed in \cite[Lemma 3.3]{JL2-regularity}, and here we note that 
  \EQ{
  A(r), B(r) &= O( r^k) \mas r \to 0, \quad \ti B(r) = O( r^k \abs{\log r}) \mas r \to 0 \\
  A(r), B(r), \ti B(r) &= O( r^{-k +2}) \mas r \to \infty
  }
  Next, given a time interval $J \subset \R$ and  a quadruplet of $C^1$ functions $(\mu(t), \lam(t), a(t), b(t))$ on $J$ we define  the $2$-bubble ansatz, $$\bs \Phi(\mu(t), \lam(t), a(t), b(t), r) = (\Phi(\mu(t), \lam(t), a(t), b(t), r), \dot \Phi(\mu(t), \lam(t), a(t), b(t), r))$$ by 
  \EQ{ \label{eq:Phi-def1} 
 \Phi(\mu, \lam, a, b)&:= (Q_{\la} + b^2 A_\la + \nu^k B_\la ) - ( Q_\mu + a^2 A_\mu + \nu^k \ti B_\mu) \\
\dot \Phi(\mu, \lam, a, b)&:= b \La Q_{\U \la}  + b^3 \La A_{\U \la}  - 2 \gamma_k b \nu^k A_{\U \la}  + b \nu^k \La B_{\U \la} - k b \nu^k B_{\U \la} - k a \nu^{k+1} B_{\U \lam}  \\
& \quad  +  a \La Q_{\U \mu}+ a^3 \La A_{\U \mu}  + 2 \ti \ga_k a \nu^k A_{ \U\mu} + a \nu^k \La \ti B_{\U \mu} + k b \nu^{k-1} \ti B_{ \U\mu}  + k a \nu^k \ti B_{\U \mu} 
  }
  where we have introduced the notation, 
$
  \nu:= \lam/ \mu.
  $
  To ensure that $\dot \Phi \in L^2$, we now restrict to the setting $k \ge 4$. See~\cite[Remark 1.4]{JL2-regularity} for a discussion of the cases $k =2, 3$. The main result from~\cite{JL2-regularity} is the following theorem.

\begin{thm}[A refined two-bubble construction] \label{t:refined} \emph{\cite{JL2-regularity}} 
Fix any equivariance class $k \ge 4$. There exists a global-in-time solution $\bs u_c(t) \in \HH$ to~\eqref{eq:wmk} that is a \emph{two-bubble in forward time} with the following additional properties: 
\begin{itemize} 
\item The solution $\bs u_c(t)$ lies in the space $\HH \cap \HH^2 \cap \bs \Lam^{-1} \HH$, and scatters freely in negative time. 
\item There exists $T_0>0$, a quadruplet of  $C^1( [T_0, \infty)$ functions $(\mu_c(t), \lam_c(t), a_c(t), b_c(t))$, and $\bs w_c(t)  \in \HH \cap \HH^2 \cap \bs \Lam^{-1} \HH$ so that on the time interval $[T_0, \infty)$ the solution $\bs u_c(t)$ admits a decomposition, 
\EQ{ \label{eq:u_c-def} 
\bs u_c(t)  = \bs \Phi(\mu_c(t), \lam_c(t), a_c(t), b_c(t)) + \bs w_c(t) 
}
where $\bs \Phi$ is defined in~\eqref{eq:Phi-def1} and the functions $(\mu_c(t), \lam_c(t), a_c(t), b_c(t))$ satisfy, 
\EQ{ \label{eq:mod-est} 
\lam_c(t) &= q_k t^{-\frac{2}{k-2}} (1 + O( t^{-\frac{4}{k-2} + \eps}) ) \mas t \to \infty, \\
\mu_c(t) &= 1 -\frac{k}{2(k+2)} q_k^2 t^{-\frac{4}{k-2}}  + O(t^{-\frac{6}{k-2}+ \eps}) \mas t \to \infty , \\
b_c(t) & = q_k  \frac{2}{k-2} t^{- \frac{k}{k-2}}( 1+  O( t^{-\frac{4}{k-2} + \eps}) ) , \mas t \to \infty ,\\  
a_c(t) & =  \frac{2k}{(k-2)(k+2)} q_k^2 t^{-\frac{k+2}{k-2}} ( 1 + O( t^{-\frac{4}{k-2} + \eps}) )  \mas t \to \infty ,
% \frac{k}{k+2} \rho_k\ti  \lam_{\app}^{\frac{k}{2} + 1}
% \abs{a_c(t)}& = o(1) t^{-\frac{k}{k-2}} \mas t \to \infty , 
}
where $\eps>0$ is any fixed small constant. We also have, 
\EQ{ \label{eq:mod'} 
\abs{ \lam_c'(t) + b_c(t)}  &\lesssim  t^{-\frac{2}{k-2}(2k-1)}   \mas t \to \infty,  \\
 \abs{\mu_c'(t)  - a_c(t)} & \lesssim  t^{-\frac{2}{k-2}(2k-1)}   \mas t \to \infty .
}
Finally,  $\bs w_c(t)$ satisfies, 
\EQ{ \label{eq:w-est} 
 \| \bs w_c(t) \|_{\HH}^2 & \lesssim \lam_c(t)^{3k-2}, \\
 \| \bs w_c(t) \|_{\HH^2}^2 & \lesssim \lam_c(t)^{3k-4}, \\
 \| \bs \Lam \bs w_c(t)  \|_{\HH}^2 & \lesssim \lam_c(t)^{2k-2} , 
}
uniformly in $t \ge T_0$. 
\end{itemize} 
\end{thm}

\subsection{An outline of the proof of uniqueness: method of refined modulation parameters} 
The goal of this paper is to prove that $\bs u_c(t)$ from~\eqref{eq:u_c-def} is the unique $2$-bubble in forward time up to a change of sign,  a fixed time translation, and rescaling. We introduce a dynamical method to accomplish this. We will  highlight below where the need for the refined construction in Theorem~\ref{t:refined} appears in the proof. 

%Let us begin by assuming Theorem~\ref{t:refined} has been proved; i.e., that we have constructed a $2$-bubble solution $\bs u_c(t)$ (with the sign $ \iota = +1$ in~\eqref{eq:2-bub-def}). To prove Theorem~\ref{t:main} we need  to show that $\bs u_c(t)$ is the unique $2$-bubble up to a fixed time translation and  rescaling. 

%Over the course of this discussion we will highlight the need for the refined construction obtained in Theorem~\ref{t:refined} and proved in Sections~\ref{s:reg},~\ref{s:construction}. 

The first observation is that $\bs u_c(t)$ yields an invariant $2$-dimensional sub-manifold $\M$ of the energy space $\HH$ via time translation and scaling. For large times, it is natural to endow this manifold with coordinates related to the $2$-bubble structure of $\bs u_c(t)$, i.e., for all $t \ge T_0$ we use the $C^1$ functions $\lam_c(t), \mu_c(t)$ given by Theorem~\ref{t:refined} such that 
\EQ{
\bs u_c(t) = \bs Q_{\lam_c(t)} + \bs Q_{\mu_c(t)} + o_{\HH}(1)
}
where $\lam_c(t) = q_k t^{-\frac{2}{k-2}}(1 + o(1))$ and $\mu_c = 1 + o(1)$ as $t \to \infty$. Because $\lam_c(t)$ is monotonic in time, it is natural to reparamaterize time via the inverse function, i.e., $t = \lam_{c}^{-1}(\s)$ for $\s \in (0, \s_0]$ where $\s_0 = \lam_c(T_0)$. We define, 
\EQ{
\bs U(\mu, \s) := \bs u_c( \lam_c^{-1}(\s))_{\mu} 
}
i.e., $(\mu, \s)$ give coordinates on $\M$ in the large time regime. 

Now let $\bs u(t) \in \HH $ be \emph{any other} $2$-bubble solution in forward time (with the  sign $\iota = +1$ in~\eqref{eq:2-bub-def}). The idea is to modulate about $\bs U(\mu, \s)$. Via a standard argument, we show that there exist $C^1$-functions $\mu(t), \s(t)$ and $\bs g(t) \in \HH$ such that for large enough times $t$ we have 
\EQ{ \label{eq:ortho-intro} 
\bs u(t) &=  \bs U(\mu(t), \s(t)) + \bs g(t) ,\\
0& = \ang{ \Lam Q_{\mu(t) \s(t)} \mid  g(t)}  = \ang{ \Lam Q_{\mu(t)} \mid  g(t)}
}
and that $ \| \bs g(t) \|_{\HH}, \s(t) \to 0 \mas t \to \infty$ (for the simple reason that $\M$ also asymptotically approaches the set of two bubble configurations $\{ \bs Q_{\lam} - \bs Q_\mu: (\lam, \mu) \in   (0, \infty) \times (0, \infty), \lam/ \mu \ll 1\}$). Note that the desired uniqueness would follow from showing that $\bs g(t) = 0$ for some time $t \ge T_0$.  

We now make use of the fact that $\bs u(t)$ and $\bs U(\mu, \s)$ both have energy $ \E = 2 \E(\bs Q)$. For each time $t \ge T_0$ we consider a Taylor expansion of the energy, 
\EQ{
&\E(\bs U(\mu(t), \s(t))) = \E(  \bs u(t)) = \E( \bs U(\mu(t), \s(t)) +\bs g(t))  \\
& \quad =  \E( \bs U(\mu(t), \s(t))) + \ang{ D \E(\bs U(\mu(t), \s(t))) \mid \bs g(t)}    +  \ang{ D^2 \E(\bs U(\mu(t), \s(t)))  \bs g(t) \mid \bs g(t)}  + o( \| \bs g \|_{\HH}^2) %\\
%& = 2 \E(\bs Q)  +  \ang{ D \E(\bs U(\mu(t), \s(t))) \mid \bs g(t)}    +  \ang{ D^2 \E(\bs U(\mu(t), \s(t)))  \bs g(t) \mid \bs g(t)}  + o( \| \bs g \|_{\HH}^2)
}
Subtracting $\E(\bs U(\mu(t), \s(t)))$ from both sides,  establishing a coercivity estimate for the quadratic term (which is a consequence of  the orthogonality conditions~\eqref{eq:ortho-intro}), and making the ``little oh" term above smaller than half the coercivity constant $c_1>0$ (which is possible by taking $T_0>0$ large enough)  we arrive at the inequality 
\EQ{ \label{eq:enexp-intro} 
 0 \ge  \ang{ D \E(\bs U(\mu(t), \s(t))) \mid \bs g(t)}  + \frac{1}{2} c_1 \| \bs g(t) \|_{\HH}^2
}
We next turn to studying the dynamics of the term $\ang{ D \E(\bs U(\mu(t), \s(t))) \mid \bs g(t)}$ with the objective of finding a contradiction above in the case that $\bs g(t) \not \equiv 0$.  This is not an unnatural object to study, as one can observe that $D \E(\bs U(\mu, \s)) = - \lam_c'(\lam_c^{-1}(\s)) \mu^{-1}  J \circ \p_\s \bs U(\mu, \s)$, i.e., it is a renormalized $90$-degree rotation of the tangent vector $\p_\s \bs U(\mu, \s)$ and moreover to leading order we have, 
\EQ{
\ang{ D \E(\bs U(\mu(t), \s(t))) \mid \bs g(t)}  \approx  - \lam_c'(\lam_c^{-1}(\s)) (\mu \s)^{-1} \ang{ \Lam Q_{\mu \s} \mid \dot g} 
}
In other words, $\ang{ D \E(\bs U(\mu(t), \s(t))) \mid \bs g(t)}$ is deeply related to the dynamics of the modulation parameters as can be seen from differentiating the terms in the second line in~\eqref{eq:ortho-intro}. However, a naive second  differentiation of the orthogonality conditions~\eqref{eq:ortho-intro} does not directly reveal useful information on the dynamics since terms of critical size but indeterminate sign arise. Here we use a technique similar to the one developed in~\cite{JJ-APDE, JJ-AJM, JL1} -- we perform an ad hoc correction \emph{to the modulation parameters themselves} using a localized virial functional. 
%The derivative of $\ang{ D \E(\bs U(\mu(t), \s(t))) \mid \bs g(t)}$ is thus related to second derivatives of the modulation parameters, 
After proving that $- \lam_c'(\lam_c^{-1}(\s)) = \rho_k \s^{\frac{k}{2}} (1 + o(1))$ in Theorem~\ref{t:refined}, we define, 
\EQ{ \label{eq:b-intro} 
b(t) :=  \frac{1}{\rho_k \s^{\frac{k}{2}}}  \ang{ D \E(\bs U(\mu(t), \s(t))) \mid \bs g(t)} + \ang{ \calA_0(\mu(t) \s(t)) g(t) \mid \dot g(t)}
}
where $\A_0(\mu \s)$ is the same localized and rescaled version of the virial operator used in the companion paper~\cite{JL2-regularity}, i.e., $\A_0(\mu \s) \approx (\mu \s)^{-1} \Lam_0$ up to scale $\mu \s$.  While the correction is small (order $ \| \bs g \|_{\HH}^2$) as compared to the first term, its derivative is large and designed to cancel terms with critical size but indeterminate sign. 

%We will use a similar localized virial correction to prove energy-type estimates needed to prove Theorem~\ref{t:refined} and we describe this technique, and how it was inspired by a related device in work of Raphael and Szeftel~\cite{RaSz11}, in more detail in the next subsection below.   

The heart of the argument is an almost monotonicity formula for $b(t)$, proved in Proposition~\ref{p:b'}, which readily leads to a contradiction in~\eqref{eq:enexp-intro}. It is in the proof of Proposition~\ref{p:b'} where the need for refined asymptotics and refined regularity estimates for $\bs U(\mu, \s)$ arises -- indeed, one can observe from~\eqref{eq:b-intro} and~\eqref{eq:DE} that the equation for $b'$ will involve estimates on the second derivatives of $U(\mu, \s)$, given that $\bs g(t)$ can only be assumed to lie in $\HH$. The proof also requires weighted energy estimates. The list of estimates on $\bs U(\mu, \s)$ needed for the argument is given in Corollaries~\ref{c:U-est1} and ~\ref{c:U-est}, and Theorem~\ref{t:refined} is proved with these in mind.  Of course in~\cite{JL1} the same type of higher regularity and weighted estimates arise as well, but there we modulated around the $2$-bubble family $\bs Q_{\lam} - \bs Q_\mu$, rather than the \emph{constructed solution} $\bs U(\mu, \s)$,  and thus the analogous  estimates there followed trivially from the formula for $Q(r)$. 

\begin{rem} The basic outline above draws inspiration from the first author's work~\cite{JJ-Pisa} in a different context. There one uses a combination of the energy expansion with a modulation analysis to rule out two bubble configurations with opposing signs for the critical NLW, albeit without the virial correction to the modulation parameters, which is a crucial ingredient here.   
\end{rem} 
%Note that this argument does not use that this is a threshold solution in any crucial way. Thus the main arc of the proof may be adaptable to other settings. 

%We now turn to a brief description of the proof of Theorem~\ref{t:refined}. 
%of the error\dots related to the modulation parameters. can be studied using technique of modified modulation parameters as in~\cite{JL1}. 

%The goal of is to prove energy estimate, which means monotonicity estimates 

%\Red{ inspired by a related technique developed by Rapha\"el and Szeftel~\cite{RaSz11} in a different context. } 
%sdfs

\begin{rem} 
 Note that the argument does not use that $\bs u_c(t)$ is a threshold solution in a crucial way, and thus should be applicable in other settings. 
 \end{rem} 

\begin{rem} \label{r:kinks} 
Together with the proof of uniqueness of the strongly interacting kink-antikink pair in~\cite{JKL1} we have now introduced two quite different techniques to prove uniqueness (and existence) of solutions to dispersive PDEs exhibiting nontrivial dynamics under some qualitative assumption -- here the assumption is the solution has threshold energy but is non-scattering, and in~\cite{JKL1} we look for asymptotically stationary $2$-kink solutions to scalar field equations. 

The methods differ as follows. In~\cite{JKL1} we first establish a quantitative classification of the dynamics for any kink-antikink pair. Then we find a single, unique kink-antikink solution in a time-weighted function space via a contraction mapping argument -- in fact this is done in two-steps by way of a novel implementation of Liapunov-Schmidt reduction. The preliminary quantitative classification result is then used to show that any finite energy kink-antikink must also lie in this weighted function space, which proves uniqueness. 

 In contrast, here \emph{we do not make any use of the dynamical classification of non-scattering threshold solutions obtained in our previous paper~\cite{JL1}} to prove uniqueness. We resort instead to the novel modulation technique that we just outlined above. The steep cost however, is that this modulation method requires very refined information on the constructed solution $\bs u_c(t)$ (including $\HH^2$-estimates), which leads to the lengthy computations in the companion paper~\cite{JL2-regularity}. 
 
 In summary, one can say quite roughly that the method here is inspired by the general principle of weak-strong uniqueness whereas in~\cite{JKL1} the method uses the contraction mapping principle to deduce uniqueness. We note that the method from~\cite{JKL1} should be adaptable to the present setting and vice versa. Both methods should be applicable  in other settings as well. 
%To contrast with kinks paper We do not use the classification result. i.e., stress that this technique to prove uniqueness requires more information about a single constructed solution, but far less in terms of a preliminary classification of \emph{all} solutions. i.e., contrast the methods with that used in kinks paper, pros cons. etc. 
\end{rem}

%\subsection{Further discussion of the Theorem~\ref{t:main} }

%\begin{itemize} 
%\item Compare with Duyckaerts-Merle. 
%\item Compare with multi-bubble uniqueness, Martel gkdv, and then JKL1 -- kinks. Notion of \Red{strongly interacting} solitons. 
%\item Compare with uniqueness of psuedo-conformal blow up, Merle, and then Raphael-Szeftel. 
%\end{itemize} 
%
%\begin{itemize} 
%\item We do not use that this is a threshold solution in any essential way it seems. 
%\item We do not use the classification result. i.e., stress that this technique to prove uniqueness requires more information about a single constructed solution, but far less in terms of a preliminary classification of \emph{all} solutions. i.e., contrast the methods with that used in kinks paper, pros cons. etc. 
%%\item discuss how the forward in time $2$-bubble emerges from pure radiation. viewed in reverse, the bubbles annihilate into radiation after their collision. 
%%\item discuss $\HH^2$ regularity. for large time, solution gives an invariant $2d$ differentiable manifold. 
%%\item $k=4$ assumption is for simplicity. $k=3$ should follow easily from same argument (check if this is worth including briefly). $k=2$ introduces additional technical hurdles -- can be overcome probably. $k=1$ is different story. 
%\end{itemize} 
%

\section{Preliminaries} 
For radial functions $u, v$ on $L^2(\R^2)$, we write $u = u(r), v = v(r)$ and we use the notation, 
\EQ{
\ang{ u \mid v} := \frac{1}{2 \pi} \ang{ u\mid v}_{L^2( \R^2)} = \int_0^\infty u(r)  \ba{v(r)} \,r \, \ud r .
}
Let $\LL_0$ denote the operator 
\EQ{ \label{eq:LL0-def} 
\LL_0 w := -\De w + \frac{k^2}{r^2} w .
}
We define the function space $H$ as the completion of $C^\infty_0((0, \infty))$ functions $w$ under the norm 
\EQ{
\| w \|_{H}^2 := \ang{ \LL_0 w \mid w}  =  \int_0^\infty \Big(( \p_r w(r))^2  + k^2 \frac{w(r)^2}{r^2}  \Big) \,r \ud r
}
For the vector pair $\bs w =(w , \dot w)$ we define the norm $\HH$ by 
\EQ{
\| \bs w \|_{\HH}^2 := \| w \|_H^2 + \| \dot w \|_{L^2}^2
}
Next, we define the space $H^2$ via the norm, 
\EQ{
 \| w \|_{H^2}^2:= \ang{ \LL_0 w  \mid \LL_0 w} = \int_0^\infty\Big( ( \p_r^2 w(r))^2 + (2k^2 +1)\frac{( \p_r w(r))^2}{r^2}  + (k^4 - 4k^2) \frac{w(r)^2}{r^4} \Big) \, r \, \ud r 
}
And for the pair $ \bs w = (w,  \dot w)$ we define $\HH^2$ by 
  %\Red{Check what is good def of $\HH^2$ based on coercivity statement}
\EQ{
\| \bs w \|_{\HH^2}^2:= \| w \|_{H^2}^2 +   \| \dot w \|_H^2 
}
We also require the following weighted norm, 
\EQ{
\| \bs w \|_{\bs \Lam^{-1} \HH} := \| (\La w, \La_0 \dot w) \|_{\HH} 
}
where $\Lam, \Lam_0$ are defined in~\eqref{eq:LaLa0}. 
It is a standard fact that the regularity of a solution $\bs u(t)$ to~\eqref{eq:wmk} in the space $\HH \cap \HH^2 \cap \bs \Lam^{-1} \HH$ is propagated by the flow. 

%Given a radial function $w: \R^2  \to \R$ we denote the $H$ and $L^2$  re-scalings as follows 
%\EQ{ \label{eq:scale} 
%w_\la(r) = w(r/ \la), \quad
%w_{\ula}(r)  = \frac{1}{\la} w (r/ \la)
%}
%The corresponding infinitesimal generators  are given by 
%\EQ{ \label{eq:LaLa0} 
%&\La w:= -\frac{\partial}{\partial \lambda}\bigg|_{\lambda = 1} w_\la = r \p_r w  \quad (H  \,  \textrm{scaling}) \\
%& \La_0 w:= -\frac{\partial}{\partial \lambda}\bigg|_{\lambda = 1} w_{\ula} = (1 + r \p_r ) w  \quad (L^2  \textrm{scaling})
%}
%Recall that the infinitesimal generator of the scalings, $\Lam, \Lam_0$ are given by 
%\EQ{
%\Lam f (r) &:= r \p_r f(r) = - \p_{\lam} \rest_{\lam = 1}  f( r/ \lam)  \qquad  H-\textrm{scaling} \\
%\Lam_0 f(r) &:= (r \p_r +1)f(r)  := - \p_{\lam}  \rest_{\lam = 1}   \lam^{-1} f( r/ \lam)  \qquad L^2-\textrm{scaling} 
%}
The  infinitesimal generators $\Lam, \Lam_0$ defined in~\eqref{eq:LaLa0} satisfy the integration by parts identities,
\EQ{ \label{eq:ibpLa} 
\ang{ \Lam f \mid g} &= - \ang{ f \mid \La g} - 2 \ang{ f \mid g} , \quad 
\ang{ \Lam_0 f \mid g} = -\ang{ f \mid \Lam_0 g} 
}

The operator $\LL_U$ obtained by linearization of~\eqref{eq:wmk} about the first component of finite energy map $\bs U = (U,  \dot U)$ plays an important role in the analysis. Given $g \in H$ we have, 
\EQ{  \label{eq:LU-def} 
\LL_U g := - \De g + k^2 \frac{\cos 2U}{r^2}  g
}
In fact, given any $\bs g = (g,\dot g) \in \HH$ we have 
\EQ{
\ang{ \uD^2 \E(\bs U) \bs g \mid \bs g } = \ang{ \LL_U g \mid g}_{L^2} + \ang{ \dot g \mid \dot g}_{L^2} =  \int_0^\infty \Big(\dot g^2 + (\p_r g)^2 + k^2 \frac{\cos 2U}{r^2} g^2 \, \Big) r \ud r
}

%We record  the commutator identities, 
%\EQ{ \label{eq:comm-LL0} 
%[\LL_0, \La] w =   2  \LL_0 w  
%}
%\EQ{
%[\LL_U, \La_0] h = [\LL_U, \La]h = 2 \LL_Uh + \frac{2 k^2 \sin 2 U  \La U}{r^2} h 
%} 
%For a time dependent function $h$ we also compute, 
%\EQ{ \label{eq:comm-ptL} 
%[\p_t, \LL_U] h = -\frac{2k^2\sin 2 U  \p_t U}{r^2}  h 
%}

The most important instance of the operator $\LL_{U}$ is given by linearizing \eqref{eq:wmk} about $ U = Q_\la$. In this case we use the short-hand notation, 
\EQ{ \label{eq:LL-def} 
\LL_\la:= \LL_{Q_{\lam}} = (-\De + \frac{k^2}{r^2}) + \frac{1}{r^2} ( f'(Q) - k^2) 
}
We write $\LL := \LL_1$. %We also introduce the notation, 
%As usual, we denote by $\LL$ the operator 
%\EQ{
%\LL := (-\De + \frac{k^2}{r^2}) + \frac{1}{r^2} ( f'(Q) - k^2) 
%}
We often use the notation $\LL= \LL_0 + P$, where $\LL_0$ is as in~\eqref{eq:LL0-def} and 
\EQ{ \label{eq:P-def} 
P(r)&:= \frac{1}{r^2}( f'(Q) - k^2)  = -\frac{2k^2\sin^2 Q}{r^2}  = -4k^2\frac{ r^{2k-2}}{ (1+ r^{2k})^2} 
}
We recall that 
\EQ{
\Lam Q(r) = k \sin Q = \frac{2k r^k}{1 + r^{2k}}
}
 is a zero energy eigenfunction for  $\LL$, that is,  
\EQ{
\LL \La Q = 0, \mand \La Q  \in L^2_{\textrm{rad}}(\R^2).
}
for all $k \ge 2$. When $k=1$, $\LL \La Q = 0$ holds but  $\La Q \not \in L^2$ due to slow decay as $r \to \infty$  and $0$ is referred to as a threshold resonance.
In fact, $\La Q$ spans the kernel of $\LL$; see~\cite{JL1} for more. 

We require the following localized coercivity result for functions in the orthogonal complement to the kernel of $\LL$. This was proved in detail in~\cite{JJ-AJM}; see also~\cite{JL2-regularity}. 
\begin{lem}[Localized coercivity for $\LL$] \emph{\cite[Lemma 5.4]{JJ-AJM}}  \label{l:loc-coerce} 
There exists a uniform constant $c_1>0$  with the following property. Suppose $w \in H$ is such that 
\EQ{ \label{eq:w-ortho1} 
\ang{ w \mid  \Lam Q} = 0. 
}
Then, 
\EQ{ \label{eq:L-coerce}
\ang{ \LL w \mid w} \ge c_1  \| w \|_{H}^2 .
}
In addition, for any $c>0$, there exists $R_1>0$ large enough so that for all $w \in H$ as in~\eqref{eq:w-ortho1}, we have 
\EQ{ \label{eq:L-loc1} 
\int_0^{R_1} \Big((\p_r w(r))^2 + k^2 \frac{w(r)^2}{r^2} \Big) \, r \ud r + \ang{ P w \mid w} \ge - c \| w \|_{H}^2 
}
Lastly, for any $c>0$, there exists $r_1>0$ small enough so that for all $w \in H$ as in~\eqref{eq:w-ortho1}, we have 
\EQ{ \label{eq:L-loc2} 
\int_{r_1}^\infty \Big((\p_r w(r))^2 + k^2 \frac{w(r)^2}{r^2}\Big)  \, r \ud r +  \ang{ P w \mid w} \ge - c \| w \|_{H}^2 
}
\end{lem} 
%This can be seen using the following well known factorization of $\LL$, 
%\EQ{ \label{eq:LLfact} 
%\LL = \calB^* \calB  \mwhere   \calB^* = \partial_r + \frac{1+k\cos (Q)}{r}, \quad \calB = - \partial_r + \frac{k\cos(Q)}{r}
%}
%together with the observation that  $\calB(\La Q) = 0$;  we note that~\eqref{eq:LLfact} is a consequence of the Bogomol'nyi factorization of the nonlinear energy; see~\cite{JL1, RS, RR} for more.  
%The fact that $\LL_\la \Lambda Q_\la = 0$ will play an important role in the modulation estimates.

%We will also consider the operator $\LL^2$ which is given by the following formula, 
%\EQ{ \label{eq:LL2-def} 
%\LL^2 w := \LL_0^2 w + 2 P \LL_0 w  - 2 \p_r P \p_r w + (P^2 - \De P) w.
%}
%To ease notation we define 
%\EQ{ \label{eq:K-def} 
%\Ks w := 2 P \LL_0 w  - 2 \p_r P \p_r w + (P^2 - \De P) w
%}
%and write $\LL^2  = \LL_0^2 + \Ks$. 

%We record the following functional inequalities proved in~\cite{JL2-regularity}. 
%\begin{lem} \emph{\cite[Lemma 2.1]{JL2-regularity}} \label{l:wHH2} 
%Suppose that $w \in H$. Then, 
%\begin{align} 
% \| w \|_{L^{\infty}} &\lesssim \| w \|_H  \label{eq:winfty} 
% \end{align} 
% If $w \in H \cap H^2$, then, 
% \begin{align} 
% \| r^{-1} w  \|_{L^\infty}  + \| r^{-2} w \|_{L^2} &\lesssim  \| \p_r w \|_H \label{eq:w/rinfty} 
%\end{align} 
%as well as, 
%\EQ{ \label{eq:w3}
%\| r^{-2} w^3 \|_{L^2} \lesssim \| w \|_{H}^2 \| \p_r w \|_{H}  \\
%\| r^{-3} w^3 \|_{L^2} \lesssim  \| w \|_{H} \| \p_r w \|_{H}^2
%}
%\end{lem} 

\subsection{The truncated virial operators} \label{s:virial} 
We define truncated virial operators $\calA(\la)$ and $\calA_0(\la)$, and state related estimates. Nearly identical operators were introduced by the first author in~\cite{JJ-AJM} and used crucially by the authors in~\cite{JL1}. We require a slight modification, which was established in~\cite{JL2-regularity}.  
%We refer the reader to~\cite[Lemma 4.6]{JJ-AJM} and~\cite[Lemma 5.5]{JJ-AJM} for the proofs of the following statements. 

\begin{lem} \emph{\cite[Lemma 4.6]{JJ-AJM} \cite[Lemma 4.1]{JL2-regularity}}
  \label{l:pdef}
  For each $c, R > 0$ there exists a function $p(r) = p_{c, R}(r) \in C^{5,1}((0, +\infty))$ with the following properties:
  \begin{enumerate}%[label=(P\arabic*)]
    \item $p(r) = \frac{1}{2} r^2$ for $r \leq R$, %\label{enum:approx-q}
    \item there exists $\ti R = \ti R(R, c)> R$ such that $p(r) \equiv \tx{const}$ for $r \geq \ti R$, %\label{enum:support-q}
    \item $|p'(r)| \lesssim r$ and $|p''(r)| \lesssim 1$ for all $r > 0$, with constants independent of $c, R$, %\label{enum:gradlap-q}
    \item $p''(r) \geq -c$ and $\frac 1r p'(r) \geq -c$, for all $r > 0$, %\label{enum:convex-ym}
    \item $ \abs{r \p_r \De p} \le c$ for all $r > 0$,
    \item $\De^2 p(r) \leq c\cdot r^{-2}$, for all $r > 0$,% \label{enum:bilapl-ym} \\
    \item $\De^3 p(r)  \ge -c  \cdot r^{-4}$ for all $r >0$, 
    \item $\big|r\big(\frac{p'(r)}{r}\big)'\big| \leq c$, for all $r > 0$, %\label{enum:multip-petit-ym}
    \item $\abs{ r \bigg(r\big(\frac{p'(r)}{r}\big)'  \bigg)' } \leq c$ for all $r >0$.
 \end{enumerate}
\end{lem}

For each $\lambda > 0$ define $\A(\lambda)$ and $\A_0(\lambda)$ as follows, 
\begin{align}
  [\A(\lambda)w](r) &:= p'\big(\frac{r}{\lambda}\big)\cdot \p_r w(r), \label{eq:opA} \\
  [\A_0(\lambda)w](r) &:= \big(\frac{1}{2\lambda}p''\big(\frac{r}{\lambda}\big) + \frac{1}{2r}p'\big(\frac{r}{\lambda}\big)\big)w(r) + p'\big(\frac{r}{\lambda}\big)\cdot\p_r w(r). \label{eq:opA0}
\end{align}
Note the similarity between $\A$ and $\frac{1}{\la} \La$ and between $\A_0$ and $\frac{1}{\la} \La_0$. 
%For technical reasons we introduce the space 
%\EQ{
%X:= \{ w \in H \mid \frac{w}{r}, \p_r w \in H\}
%}
%Let 
%$
%f(\rho):= \frac{k^2}{2} \sin 2 \rho
%$
%denote the nonlinearity in~\eqref{eq:wmk} and
Recall the notation, 
$
\LL_0:= -\De + \frac{k^2}{r^2} . 
$
\begin{lem} \emph{\cite[Lemma 5.5]{JJ-AJM} \cite[Lemma 4.2]{JL2-regularity}}
  \label{l:opA}
  Let $c_0>0$ be arbitrary. There exists $c>0$ small enough and $R, \ti R>0$ large enough in Lemma~\ref{l:pdef} so that the operators $\A(\lambda)$ and $\A_0(\lambda)$ defined in~\eqref{eq:opA} and~\eqref{eq:opA0} have the following properties:
  \begin{itemize}[leftmargin=0.5cm]
    \item the families $\{\A(\lambda): \lambda > 0\}$, $\{\A_0(\lambda): \lambda > 0\}$, $\{\lambda\partial_\lambda \A(\lambda): \lambda > 0\}$
      and $\{\lambda\partial_\lambda \A_0(\lambda): \lambda > 0\}$ are bounded in $\mathscr{L}(H; L^2)$, with the bound depending only on the choice of the function $p(r)$,
      \item In addition, the operators $\A_0(\lam)$ and $\lam \p_\lam \A_0(\lam)$ satisfy the bounds
      \EQ{ \label{eq:prA0} 
       \| \p_r \A_0(\lam) w \|_{L^2}  +  \| r^{-1} \A_0(\lam) w \|_{L^2} &\lesssim  \| \p_r w \|_H +  \frac{1}{\lam} \| w \|_H \\
       \| \p_r \lam \p_\lam \A_0(\lam) w \|_{L^2}  +  \| r^{-1} \lam \p_\lam \A_0(\lam) w \|_{L^2} &\lesssim  \| \p_r w \|_H +  \frac{1}{\lam} \| w \|_H 
      }
      with a constant that depends only on the choice of the function $p(r)$, 
  
%    \item %Let $f(\rho):= \frac{k^2}{2}\sin 2 \rho$.  
%    For all $\lambda > 0$ and $g_1, g_2 \in X$  there holds
%      \begin{multline}  \label{eq:A-by-parts-wm}
%      \Big| \ang{ \A(\lambda)g_1\mid  \frac{1}{r^2}\big(f(g_1 + g_2) - f(g_1) - f'(g_1)g_2\big)}  \\ +\ang{ \A(\lambda)g_2\mid \frac{1}{r^2}\big(f(g_1+g_2) - f(g_1) -k^2 g_2\big)}\Big| 
%        \leq \frac{c_0}{\lambda} \|g_2\|_H^2, 
%      \end{multline}
     % with a constant $\eps_0$ arbitrarily small,
   
    \item For all $w \in H \cap H^2$ we have  %for any $\eps_0 > 0$, if the constants $c$ and $R$ in the definition of $q(r)$ are chosen appropriately, then
\EQ{
        \label{eq:A-pohozaev}
        \ang{\A_0(\lambda)w  \mid  \LL_0 w} \ge -\frac{c_0}{\lambda}\|w\|_{H}^2 + \frac{1}{\lambda}\int_0^{R\lambda}\Big((\partial_r w)^2 + \frac{k^2}{r^2}w^2\Big) \udr, 
        }
        
%        \item For all $w, \p_r w \in H \cap H^2$ we have 
%        \EQ{ \label{eq:A-pohozaev2} 
%        \ang{\A_0(\lambda)w  \mid \LL_0^2w} \ge - \frac{c_0}{\lam} \| w \|_{H^2}^2 + \frac{2}{\lam}  \int_0^{R \lam}  ( \LL_0 w)^2 \, \rdr 
%        }
        \item Moreover, for $\la, \mu >0$ with $\la/\mu \ll 1$, %Let $Q = Q_k(r) = 2 \arctan r^k$ and let $g \in H$.  There exists $\eps_0>0$ small enough so that if $\la^k  \le \eps_0 $ and $\mu \in [1, 2]$, then \Red{check hypothesis here. do we really want $\mu \in [1, 2]$ etc}
        \EQ{
         \label{eq:L0-A0-wm}
      \|\Lambda_0 \Lambda Q_\uln{\lambda} - \A_0(\lambda)\Lambda Q_{\lambda}\|_{L^2} \leq c_0,
        }
% \begin{gather}
%      \label{eq:L0-A0-wm}
%      \|\Lambda_0 \Lambda Q_\uln{\lambda} - \A_0(\lambda)\Lambda Q_{\lambda}\|_{L^2} \leq c_0, \\
%      \label{eq:L-A-wm}
%      \|\Lambda Q_\uln\lambda - \A(\lambda)Q_\lambda\|_{L^\infty} \leq \frac{c_0}{\lambda},  \\
%    \| \A(\la) Q_\mu \|_{L^\infty} + \| \A_0(\la) Q_\mu \|_{L^\infty}  \lesssim \frac{1}{\la} (\la/ \mu)^k,   \label{eq:Ainfty} 
%     \end{gather} 
%    and, for any $w \in H$, and $\Phi$ as in~\eqref{eq:Phidef} 
%    \begin{multline}   \label{eq:approx-potential-wm}
%        \bigg|\int_0^{+\infty}\frac 12 \Big(q''\big(\frac{r}{\lambda}\big) + \frac{\lambda}{r}q'\big(\frac{r}{\lambda}\big)\Big)\frac{1}{r^2}\big(f(\Phi + w) - f(\Phi)-k^2w\big)w \udr \\
%        - \int_0^{+\infty} \frac{1}{r^2}\big(f'(Q_\lambda)-k^2\big)w^2 \udr\bigg| \leq c_0(\|w\|_H^2 + (\lambda/\mu)^k).  
%  \end{multline} 
\item Finally, let $P_\lam(r)$ denote the potential, 
$P_{\lam}(r):= \frac{1}{r^2} ( f'(Q_\lam) - k^2) 
$.
%and let $\Ks_\lam$ denote the operator given by  
%\EQ{ 
%\Ks_\lam w := 2 P_\lam \LL_0 w  - 2 \p_r P_\lam \p_r w + (P_\lam^2 - \De P_\lam) w
%}
%For any $c_0>0$ we can choose $c, R>0$ in Lemma~\ref{l:pdef} so that, 
%\EQ{
%\abs{ b \ang{ \calA(\la)  w  \mid \frac{k^2}{r^2}( \cos 2 Q_\la - 1)   w} \Red{+} \frac{1}{2}\ang{ w \mid [\p_t, \LL_\la]  w}} \ll  
%}
%\EQ{
%\abs{\frac{1}{2}\ang{ w \mid [\p_t, \LL_\Phi]  w} + b  \ang{ \calA(\la)  w  \mid r^{-2}\Big( f( \Phi + w) - f( \Phi) - k^2 w \Big)} } \le c_0 \frac{b}{\la} \| w \|_{H}^2  
%}
We have, 
\EQ{ \label{eq:vir-new} 
 \abs{  \ang{ \calA_0(\la) w \mid P_\lam(r) w}  - \ang{  \frac{1}{\la} \La_0 w \mid P_\lam(r) w} } \le \frac{c_0}{\la} \| w \|_{H}^2 
}
%\Red{[Check sign above and it. needs to be rephrased to account for use of $\LL_\Phi$ rather than $\LL_\la$ in the actual estimates]} 
%as well as, 
%\EQ{ \label{eq:vir-new2} 
%\abs{ \ang{ \calA_0(\la) w \mid \Ks_\lam w}  - \ang{  \frac{1}{\la} \La_0 w \mid \Ks_\lam w} } \le \frac{c_0}{\la} \| w \|_{H^2}^2 
%}

  \end{itemize}
\end{lem}

\subsection{Properties of the ansatz $\bs \Phi(\mu, \lam, a, b)$}

The $2$-bubble ansatz $\bs \Phi(\mu, \lam, a, b)$ is defined in~\eqref{eq:Phi-def1}. The arguments in~\cite{JL2-regularity} required detailed information about $\bs \Phi$ and we recall several formulas and estimates proved there.  First recall that $A, B, \ti B$ are defined so that, 
\EQ{ \label{eq:LLb2T} 
\LL_{\la}( b^2 A_\la + \nu^k B_{\lam})&= - \frac{b^2}{\la} \La_0 \La Q_{\U \la} + \gamma_k \frac{\nu^k}{\la} \La Q_{\U \la}  - 4  \frac{(r/ \mu)^k (\La Q_{\la})^2}{r^2} \\
 \LL_\mu (a^2 A_{\mu} + \nu^k \ti B_{\mu}) &= - \frac{a^2}{\mu} \La_0  \La Q_{\U \mu}   - \ti \gamma_k \frac{\nu^k}{\mu}  \La Q_{\U \mu}  + 4 \frac{(r/ \la)^{-k} (\La Q_{\mu})^2}{r^2} 
}
The existence of such $A, B, \ti B$ is made precise in the following lemma. 
\begin{lem}  \emph{\cite[Lemma 3.3]{JL2-regularity}}\label{l:ABB} Let $k \ge 4$. 
There exist $C^\infty(0, \infty)$ functions $A, B, \ti B$ satisfying~\eqref{eq:ABB-def}. Moreover $A, B, \ti B$ satisfy the estimates 
\EQ{ \label{eq:Aest} 
A(r) &= O( r^{k}), \quad \p_r A(r) = O(r^{k-1}), \quad \p_r^2 A(r) = O(r^{k-2})   \mas r \to 0  \\
A(r) &= O(r^{-k+2}), \quad \p_r A(r) = O(r^{-k+1}), \quad \p_r^2 A(r) = O(r^{-k})   \mas r \to \infty
}
\EQ{ \label{eq:Best} 
B(r) &= O( r^{k}), \quad \p_r B(r) = O(r^{k-1}), \quad \p_r^2 B(r) = O(r^{k-2})   \mas r \to 0  \\
B(r) &= O(r^{-k+2}), \quad \p_r B(r) = O(r^{-k+1}), \quad \p_r^2 B(r) = O(r^{-k})   \mas r \to \infty
}
\EQ{ \label{eq:tiBest} 
\ti B(r) &= O( r^{k}\abs{\log r}), \, \,  \p_r \ti B(r) = O(r^{k-1}\abs{\log r}), \, \,  \p_r^2  \ti B(r) = O(r^{k-2}\abs{\log r})   \mas r \to 0  \\
 \ti B(r) &= O(r^{-k+2}), \quad \p_r  \ti B(r) = O(r^{-k+1}), \quad \p_r^2 \ti  B(r) = O(r^{-k})   \mas r \to \infty
}
Moreover we have $A, \La A, \La_0 \La A, B,  \La B, \La_0 \La B \in L^2(\R^2)$ and $\ti B, \La \ti B, \La_0 \La \ti B  \in L^2(\R^2)$. 
\end{lem} 
We have the following technical lemmas proved in~\cite{JL2-regularity}.  
\begin{lem} \emph{\cite[Lemma 3.8]{JL2-regularity}} \label{l:ABBQL21} 
Let $A, B, \ti B$ be as in Lemma~\ref{l:ABB}, and let $\nu= \lam/ \mu \ll 1$. Then, 
\EQ{ \label{eq:ABBQL21} 
 \| r^{-1} [\La Q_\la]^2 \La Q_\mu  \|_{L^2} +
    \| r^{-1} \La Q_\la [\La Q_\mu]^2 \|_{L^2}  +
    \| r^{-1} [\La Q_\la]^2 A_\mu \|_{L^2}  &\lesssim \nu^k  \\
  \|r^{-1}[ \La Q_\la]^2 \ti B_\mu \|_{L^2} &\lesssim  \nu^{k-o(1)}  \\
   \| r^{-1} [\La Q_\mu]^2 A_\la \|_{L^2}    +
 \| r^{-1} [\La Q_\mu]^2B_\la \|_{L^2} &\lesssim \nu^{k-2}  \\
  \|r^{-1} \La Q_\la [A_\mu]^2 \|_{L^2}+ 
   \| r^{-1} \La Q_\la [\ti B_\mu]^2 \|_{L^2} +%& \lesssim  \nu^k  \\
%   \| r^{-1} [\La Q_\mu]^2 \La Q_\la \|_{L^2} &\lesssim   \\
      %  \| r^{-1} \La Q_\mu [\La Q_\la]^2 \|_{L^2} &\lesssim  \\
 \| r^{-1} \La Q_\mu [A_{\la}]^2 \|_{L^2}+ 
 \| r^{-1} \La Q_\mu[ B_\la]^2 \|_{L^2}  &\lesssim  \nu^k 
}
where the $o(1)$ above can be replaced with any small constant. 
\end{lem}

With $\bs \Phi = (\Phi, \dot \Phi)$ defined as in~\eqref{eq:Phi-def1} we have,  
\EQ{ \label{eq:eqPhi} 
&-\De \Phi  + \frac{1}{r^2} f( \Phi) =  \gamma_k \frac{\nu^k}{\lam} \Lam Q_{\U \lam} - \frac{b^2}{\lam} \Lam_0 \Lam Q_{\U \lam}  + \gamma_k  \frac{\nu^k}{\mu}  \Lam Q_{\U \mu} + \frac{a^2}{\mu} \Lam_0 \Lam Q_{\U \mu}  \\
 &\quad -\frac{1}{r^2} \Big( f(Q_\la - Q_\mu)  - f(Q_\la) + f(Q_\mu)  - 4  \big(\frac{r}{ \mu}\big)^k (\La Q_{\la})^2  -  4 \big(\frac{r}{ \la}\big)^{-k} (\La Q_{\mu})^2 \Big)\\
&\quad -\frac{1}{r^2} \Big( f(\Phi) - f(Q_\la - Q_\mu) -  f'(Q_\la- Q_\mu) (( b^2 A_\la + \nu^k B_{\lam}) - (a^2 A_{\mu} + \nu^k \ti B_{\mu}) \Big)  \\
&\quad -\frac{1}{r^2}  \Big(f'(Q_\la - Q_\mu) (( b^2 A_\la + \nu^k B_{\lam}) - (a^2 A_{\mu} + \nu^k \ti B_{\mu}))   \\
&\qquad \qquad \quad - f'(Q_\la) ( b^2 A_\la + \nu^k B_{\lam}) + f'(Q_\mu) (a^2 A_{\mu} + \nu^k \ti B_{\mu})\Big)
}

Next we recall the estimates proved in~\cite{JL2-regularity}.

\begin{lem}  \emph{\cite[Lemma 3.11]{JL2-regularity}}  \label{l:fest1} 
Let $\bs \Phi = ( \Phi, \dot \Phi)$ be defined as in~\eqref{eq:Phi-def1} and let $(\mu, \lam, a, b)$ satisfy $ \nu:= \lam/ \mu \ll 1$ and $ \abs{a}, \abs{b} \ll 1$. 
%Let $\nu:= \la/ \mu$ as usual. Assume in addition that the parameters $a, b, \nu$ satisfy, 
%\EQ{
%a, b ,\nu \ll 1, \quad \mu \simeq 1
%}
Then, for $\al = 1, 2, 3$ we have 
\EQ{ \label{eq:ff'L2} 
 & \Big\| r^{-\al}  \Big( f(Q_\la - Q_\mu)  - f(Q_\la) + f(Q_\mu)  - 4  \big(\frac{r}{ \mu}\big)^k (\La Q_{\la})^2  -  4 \big(\frac{r}{ \la}\big)^{-k} (\La Q_{\mu})^2 \Big) \Big\|_{L^2}  \\
  &\qquad \qquad\lesssim \nu^{2k} \lam^{-\al +1}    \\
 &   \Big\| r^{-\al}  \Big( f(\Phi) - f(Q_\la - Q_\mu) -  f'(Q_\la- Q_\mu) (( b^2 A_\la + \nu^k B_{\lam} - (a^2 A_{\mu} + \nu^k \ti B_{\mu})  \Big) \Big \|_{L^2} \\
  &\qquad \qquad \lesssim b^4 \lam^{-\al +1}  + a^4\nu^{\al-1}\lam^{-\al +1}  +\nu^{2k}  \lam^{-\al +1}    \\
&   \Big\| r^{-\al}  \Big(f'(Q_\la - Q_\mu) (( b^2 A_\la + \nu^k B_{\lam} - (a^2 A_{\mu} + \nu^k \ti B_{\mu}) )   \\
& \quad \qquad \qquad - f'(Q_\la) ( b^2 A_\la + \nu^k B_{\lam} + f'(Q_\mu) (a^2 A_{\mu} + \nu^k \ti B_{\mu}) \Big) \Big\|_{L^2} \\
 &\qquad \qquad\lesssim b^2 \nu^{k-1} \lam^{-\al +1} +a^2 \nu^{k+ \al - 2} \lam^{-\al +1}   + \nu^{2k  -1} \lam^{-\al +1} 
} 
 \end{lem} 
 
  \begin{lem} \emph{\cite[Lemma 3.13]{JL2-regularity}} \label{l:w^2est} 
 Let $ \bs w \in \HH \cap \HH^2$, let $\bs \Phi = ( \Phi, \dot \Phi)$ be defined as in~\eqref{eq:Phi-def1}, and  let $(\mu, \lam, a, b)$ satisfy $ \nu:= \lam/ \mu \ll 1$ and $ \abs{a}, \abs{b} \ll 1$. 
Then, 
\begin{align} 
  \left\| \frac{1}{r} \Big( f( \Phi + w) - f( \Phi) - f'(\Phi) w \Big) \right\|_{L^2} &\lesssim   \|  w \|_{H}^2     \label{eq:w^2L21} \\
%  \left\| \frac{1}{r^2} \Big( f( \Phi + w) - f( \Phi) - f'(\Phi) w \Big) \right\|_{L^2} &\lesssim  \frac{1}{\la} \|  w \|_{H}^2  +  \| w \|_H^2 \| \p_r w \|_H    \label{eq:w^2L2} \\
%  \left\| \frac{1}{r^3} \Big( f( \Phi + w) - f( \Phi) - f'(\Phi) w \Big) \right\|_{L^2} &\lesssim  \frac{1}{\la} \| \p_r w \|_{H} \| w \|_{H}  + \| \p_r w \|_{H}^2 \| w \|_H \label{eq:w^2L23} \\ % \frac{1}{\la} \|  w \|_{H}^2  +  \| w \|_H^2 \| \p_r w \|_H    \label{eq:w^2L2} \\
%   \left\| \frac{1}{r^2} \p_r \Big( f( \Phi + w) - f( \Phi) - f'(\Phi) w \Big) \right\|_{L^2} &\lesssim  \frac{1}{\la} \| \p_r w \|_{H} \| w \|_{H}  + \| \p_r w \|_{H}^2 \| w \|_H \label{eq:w^2L22}  \\ 
%    \abs{ \ang{ \La Q_{\U \la} \mid \frac{1}{r^2} \Big( f( \Phi + w) - f( \Phi) - f'( Q_\la) w \Big) }} &\lesssim  \frac{1}{\la} \|w \|_{H}^2 + \frac{1}{\la} \| w \|_H \Big( \nu^k + b^2 + a^2 \nu^k\Big)\label{eq:w^2la} \\ 
% \abs{ \ang{ \La Q_{\U \mu} \mid \frac{1}{r^2} \Big( f( \Phi + w) - f( \Phi) - f'( Q_\mu) w \Big) }} &\lesssim \frac{1}{\mu} \| w \|_H^2  +  \frac{1}{\mu} \| w \|_H \Big( \nu^k + b^2 \nu^{k-2}   + a^2 \Big)  \label{eq:w^2mu} 
 \end{align} 
 \end{lem}

\section{Refined Modulation Analysis}  \label{s:unique} 
Let $\bs u_c(t)$ be the solution constructed  in Theorem~\ref{t:refined} that approaches a $2$-bubble in forward time.  Let $\bs u(t) \in \HH$ be \emph{any two-bubble in forward time} as in~\eqref{eq:2-bub-def}. The goal of this paper is to prove Theorem~\ref{t:main} by showing that $\bs u(t) = \bs u_c(t)$. 

%\begin{prop} \Red{[Don't think this is needed at the moment] \dots}
%We have 
%\EQ{
%\| \bs u(t) - \bs u_c(t) \|_{\HH_0} \to 0 \mas t \to \infty
%}
%\end{prop} 
%\begin{proof}
%The idea is to show that 
%\end{proof} 

Since $\bs u(t)$ is a 2-bubble in forward time we know from~\eqref{eq:2-bub-def} that 
\EQ{
\hbfd (\bs u(t)) :=  \inf_{\sigma, \mu>0} \| \bs u(t) - (\bs Q_{\sigma \mu} - \bs Q_{\mu}) \|_{\HH}^2 + \sigma^k, 
}
satisfies 
\EQ{ \label{eq:hatd-0} 
\hbfd( \bs u(t)) \to 0 \mas t \to \infty. 
}
Note that here we have fixed the sign $\iota = +1$ in~\eqref{eq:2-bub-def} without loss of generality. In fact, we know more than this (see \cite[Theorem 1.6 and Remark 1.7]{JL1}), but the above is all we require in the sequel. 
In~\cite{JL1} we used the smallness of $\bfd(\bs u)$ to modulate around the $2$-parameter family of pure two bubbles $\bs Q_{\s \mu} - \bs Q_\mu$, that, is we imposed orthogonality conditions on the difference 
\EQ{
 \bs u(t) -  \bs Q_{\mu \s} -   \bs Q_\mu
}
by modulating in $\s$ and $\mu$. In contrast, here we modulate around the $2$-parameter family of maps given by the rescaled  \emph{trajectory} of the solution $\bs u_c(t)$ from Theorem~\ref{t:refined}, the two parameters being time $t$ and the scale $\mu$.  This is natural  in  that this trajectory is \emph{invariant} in the sense of dynamical systems.

Let $\bs u_c(t)$ be the solution given by Theorem~\ref{t:refined}. Note that that $\la_c(t)$ is monotone decreasing on the interval $[T_0, \infty)$. By Theorem~\ref{t:refined}  we have that 
\EQ{
\hbfd( \bs u_c(t)) \le \eta_0(T_0), \quad \forall t \in [T_0, \infty)
}
where $\eta_0 = \eta_0(T_0) \to 0 \mas  T_0 \to \infty$ is a constant that we can fix later to be as small as we like. 

\begin{defn}  Let $\s_0 = \la_c(T_0)$ and define the inverse function 
\EQ{
\la^{-1}_c: (0, \s_0] \to [T_0, \infty)
}
i.e., we can express each $t \in [T_0, \infty)$ uniquely by $t = \la_c^{-1}(\s)  $ for some $\s \in (0, \s_0]$. 

Define 
\EQ{ \label{eq:Udef} 
\bs U( \mu, \s, \cdot) :=  \bs u_c( \la_c^{-1}( \s), \cdot)_{\mu}  = ( u_c( \la_c^{-1}(\s),  \cdot/ \mu), \frac{1}{\mu} \p_t u_{c}( \la^{-1}_c(\s), \cdot/ \mu))
}
Then $\bs U$ defines a mapping, 
\EQ{
(0, \infty) \times (0, \s_0]  \ni ( \mu, \s)  \mapsto \bs U( \mu, \s, \cdot)  \in \HH
}
\end{defn} 

In fact, since the constructed solution has the threshold energy $\E = 2 \E(\bs Q)$, we can can view  $\bs U( \mu, \s)$ as a $2$-dimensional invariant (under the wave map flow) sub-manifold of $\{ \E = 2 \E (\bs Q)\} \subset \HH$, i.e, 
\EQ{
( \mu, \s) \mapsto \bs U( \mu, \s) \subset \{ \E = 2 \E( \bs Q)\} \subset \HH
}
\subsection{Consequences of Theorem~\ref{t:refined}} 
In this section we establish a collection of estimates on $\bs U(\mu, \s)$ that will be needed in the proof of Theorem~\ref{t:main}. 
We  write 
\EQ{
\bs U( \mu, \s) =: ( U( \mu, \s) , \, \dot U(\mu, \s)).
}
Introducing the notation, 
\EQ{
 \xi( \s) := - \la'_c( \la_c^{-1}(\s))
}
we record formulas for the $\mu, \s$ derivatives of $\bs U(\mu, \s)$. 
\begin{lem} 
Let $\bs U: (0, \infty) \times (0, \eta) \to \HH$ be defined as above with $\eta \le \s_0$. Then, 
\begin{align}
 \label{eq:pmuU} &\p_\mu \bs U ( \mu, \s) = -\frac{1}{\mu} \bs \La \bs U( \mu, \s) := -\frac{1}{\mu} \left( [\La u_c( \la_c^{-1}(\s), \cdot)]_{\mu}, \, [\La_0 \p_t u_c( \la_c^{-1}(\s), \cdot)]_{\U \mu}\right) \\
 \label{eq:psU}& \p_\s \bs U(\mu, \s) = -\frac{\mu}{\xi(\s)} J  \circ \uD \E( \bs U( \mu, \s)) = -\frac{\mu}{ \xi(\s)}  \pmat{ \dot U(\mu, \s) \\ \De U(\mu, \s) - \frac{1}{r^2} f( U( \mu, \s)) }
\end{align} 
\end{lem} 
\begin{proof} The proof of~\eqref{eq:pmuU} is a direct computation using the definition~\eqref{eq:Udef} and the definition of $\bs \La$, i.e, for $\bs V = (V, \dot V)$ we have  $$\bs\La \bs V= ( \La V, \La_0 \dot V).$$ To prove~\eqref{eq:psU} we note that for any $\mu>0$ the mapping 
\EQ{
 \bs {\ti u}(t):=  \bs U( \mu, \la_c(t/\mu)) = [ \bs u_c(t/ \mu, \cdot)]_{\mu}
}
solves~\eqref{eq:uham} on the interval $[\mu T_0, \infty)$, i.e, 
\EQ{
 \p_t \bs {\ti u}(t) = J \circ \uD \E( \bs {\ti u}(t)) \, \quad \forall  \, \, t \in [\mu T_0, \infty).
}
 In particular for $t = \mu \la_c^{-1}(\s)$ we have 
\EQ{
J \circ \uD \E( \bs {\ti u}(  \mu \la_c^{-1}(\s))) =  J \circ \uD \E( \bs U( \mu, \s)), 
}
which is the right-hand-side of~\eqref{eq:psU} up to the factor $-\frac{\mu}{\xi(\s)}$. On the other hand, using the chain-rule 
\EQ{
\p_t \bs {\ti u}(t) \rest_{t =  \mu \la_c^{-1}(\s)}  &= \left[ \p_\s \bs U( \mu, \la_c(t/\mu) )   \la'_c(t/ \mu) \frac{1}{\mu}\right]\rest_{t =  \mu \la_c^{-1}(\s)} \\
& =  \p_\s \bs U( \mu, \s) \la'_c( \la^{-1}_c( \s)) \frac{1}{\mu}   =  - \frac{\xi( \s)}{ \mu} \p_\s \bs U( \mu, \s) 
}
which establishes the claim. 
\end{proof} 

Importantly, Theorem~\ref{t:refined} yields refined regularity and decay information about $\bs U(\mu, \s)$ that will be crucial in the proof of Theorem~\ref{t:main}. 
First we fix notation. We write,  
\EQ{ \label{eq:U-Phi-w} 
\bs U(\mu, \s)  &=  \bs \Phi( \mu, \s) + \bs w_c(\mu, \s)
} 
where we have defined, 
\EQ{ \label{eq:Phi-w-mus}
\bs \Phi( \mu, \s) &:= \bs \Phi( \mu_c( \la_c^{-1}(\s)), \s, a_c( \la_c^{-1}(\s)), b_c( \la_c^{-1}(\s)))_{\mu}  \\
  \bs w_c( \mu, \s)  &:= \bs w_c( \la_c^{-1}(\s))_{\mu} 
}
Here $\bs w_c(\mu, \s) := \bs w_c( \la_c^{-1}(\s))_{\mu} $ is as in  Theorem~\ref{t:refined} and $\bs \Phi( \mu, \la, a, b)$ is defined in~\eqref{eq:Phi-def1}. 
Next, define $\bs g_c( \mu, \s)$  %= (g_c( \mu, \s),   \dot g_c(\mu, \s))$ 
via, 
\EQ{
\bs g_c( \mu, \s) &:=  \bs \Phi( \mu, \s)  + \bs w_c(\mu, \s)  \\
& \quad - \bs Q_{\mu \s} +  \bs Q_{\mu_c( \la_c^{-1}(\s)) \mu} - (0, b_c(\lam_c^{-1}(\s)) \Lam Q_{\U{\mu \s}} + a_c(\lam_c^{-1}(\s)) \Lam Q_{\U{\mu_c(\lam_c^{-1}(\s)) \mu}})
}
so that we may also write, 
\EQ{ \label{eq:U-Q-g} 
\bs U(\mu, \s)  & = ( Q_{\mu \s}, b_c(\lam_c^{-1}(\s)) \Lam Q_{\U{\mu \s}}) -   ( Q_{\mu_c( \la_c^{-1}(\s)) \mu}, -a_c(\lam_c^{-1}(\s)) \Lam Q_{\U{\mu_c(\lam_c^{-1}(\s)) \mu}})   \\
&\quad + \bs g_c( \mu, \s)
}
First we translate the main estimates from Theorem~\ref{t:refined} to estimates for the error $\bs w_c(\mu, \s)$ and for the parameters $\xi(\s), b_c( \lam_c^{-1}(\s)), \mu_c(\lam_c^{-1}(\s))$ and $a_c( \lam_c^{-1}(\s))$.  
%Importantly, Theorem~\ref{t:refined} yields refined regularity and decay information about $\bs w_c(\mu, \s)$ and thus also for $\bs g_c(\mu, \s)$.  
\begin{cor} \label{c:U-est1} Let $\bs U(\mu, \s)$, $\bs w_c(\mu, \s)$, and $\bs g_c(\mu, \s)$ be defined as above. Then, $\bs U(\mu, \s) \in \HH \cap \HH^2 \cap \Lam \HH^2$ and we have the estimates,  
\begin{align} 
 \| \bs w_c( \mu, \s) \|_{\HH} &\lesssim \s^{\frac{3}{2}k -1}  \label{eq:wc-mus-H}\\
  \| \bs w_c( \mu, \s) \|_{\HH^2} &\lesssim  \mu^{-1} \s^{\frac{3}{2} k -2}\label{eq:wc-mus-H2} \\
   \| \bs \Lam \bs w_c( \mu, \s) \|_{  \HH} &\lesssim \s^{k-1} \label{eq:wc-mus-LamH} 
\end{align} 
uniformly in $\s \in (0, \s_0]$ and  $\mu \in (0, \infty)$. We have the asymptotics, 
\EQ{ \label{eq:xis}
\lim_{ \s  \to 0^+} \frac{\xi(\s)}{ \rho_k \s^{\frac{k}{2}}} &=  \lim_{ \s  \to 0^+} \frac{b_c( \la_c^{-1}(\s))}{ \rho_k \s^{\frac{k}{2}}}  = 1, 
}
\EQ{ \label{eq:muc-lim} 
\abs{\mu_c(\la_c^{-1}(\s)) -1} = o(1) \s \mas \s \to 0^+ , 
}
and 
\EQ{ \label{eq:ac-lim} 
\abs{a_c( \la_c^{-1}(\s))} = o(1) \s^{\frac{k}{2}} \mas  \s \to 0^+. 
}
In particular, the above additionally yield the less-refined estimates, 
\begin{align} 
\| \bs g_c(\mu, \s)  \|_{\HH} &\lesssim \s^{k}  \label{eq:gc-H}  \\
\| \bs \Lam \bs g_c(\mu, \s)  \|_{\HH} &\lesssim \s^{k}  \label{eq:gc-LamH}
\end{align} 
%\Red{maybe redefine $g_c$ by extracting $b_c \Lam Q$ and $a_c \Lam Q$ so that estimates below can be more succinctly stated?} 
where the $o(1)$ in all of the above denotes a constant that can be made as small as we like by taking $T_0$ large enough (and hence $\s_0$ small enough). 
\end{cor} 

\begin{cor}  \label{c:U-est} 
Let $\bs U(\mu, \s)$ be as above. Then, 
\begin{align} 
 \| \frac{1}{\rho_k \s^{\frac{k}{2}}} \dot U(\mu, \s) - \Lam Q_{\U{\mu \s}} \|_{L^2}   &= o(1) \mas \s \to 0  \label{eq:dotU1} \\
 \| \frac{1}{\rho_k \s^{\frac{k}{2}}} \dot U(\mu, \s) - \Lam Q_{\U{\mu \s}} \|_{H}  & = \frac{o(1)}{\mu \s} \mas \s \to 0 \label{eq:dotU-H}  \\ 
 \| r(-\De U(\mu, \s) + r^{-2} f(U(\mu, \s))) \|_{L^2} & \lesssim  \s^{k-1} \mas \s \to 0  \label{eq:eqU}   \\
  \| -\De U(\mu, \s) + r^{-2} f(U(\mu, \s)) \|_{L^2} & \lesssim  \frac{\s^{k}}{ \mu \s}  \mas \s \to 0  \label{eq:eqU-L2}
\end{align} 
Moreover, for any $\bs h \in \HH$ we have, 
\begin{align} 
 \abs{\ang{ D^2 \E(\bs U(\mu, \s))\p_\mu \bs U(\mu, \s) \mid \bs h} } &\le o(1)   \s^{\frac{k}{2} } \frac{\| \bs h \|_{\HH}}{\mu \s}  \label{eq:D2mu}  \\ 
\abs{ \big\langle D^2 \E( \bs U(\mu, \s)) \p_\s \bs U(\mu, \s) - ( 0,  \frac{\s^{\frac{k}{2}}}{\ \rho_k \s} ( \gamma_k \Lam Q_{\U{\mu \s}} - \rho_k^2  \Lam_0 \Lam Q_{\U{\mu \s}})) \mid  \bs h \big\rangle }& \le o(1) \s^{\frac{k}{2}}  \frac{\| \bs h \|_{\HH}}{\s}   \label{eq:D2sig} 
\end{align} 
where $o(1)$ can be replaced with any small constant by taking $\s$ small enough. 
\end{cor} 

Before proving Corollary~\ref{c:U-est} it will be convenient to first  translate estimates for the ansatz $\bs \Phi(\mu, \lam, a, b)$ into estimates for $\bs \Phi(\mu, \s)$. 
\begin{lem}  \label{l:Phi-mus} Let $\bs \Phi(\mu, \s)$ be defined as above. Then, 
\EQ{ \label{eq:Lam0-dotPhi} 
 \| \dot \Phi(\mu, \s)  \|_{L^2} +  \|  \Lam_0  \dot \Phi(\mu, \s) \|_{L^2}  \lesssim \s^{\frac{k}{2}}
}
\EQ{\label{eq:dotPhi-bLamQ} 
\| \dot \Phi(\mu, \s) - b_c( \la_c^{-1}(\s))  \Lam Q_{\U{\mu \s}} \|_{L^2} & \lesssim o(1) \s^{\frac{k}{2}}  \\
\abs{\ang{ \Lam Q_{\U{\mu \s}} \mid  \dot \Phi(\mu, \s) - b_c( \la_c^{-1}(\s))  \Lam Q_{\U{\mu \s}}} }&\lesssim o(1) \s^{\frac{3}{2}k -1}  
}
\EQ{ \label{eq:dotPhi-bLamQ-H} 
\| \dot \Phi(\mu, \s) - b_c( \la_c^{-1}(\s))  \Lam Q_{\U{\mu \s}} \|_{H} & \lesssim o(1) \frac{\s^{\frac{k}{2}} }{\mu \s} 
}
\EQ{ \label{eq:dotPhi-aLamQ} 
\| \dot \Phi(\mu, \s) - a_c( \la_c^{-1}(\s))  \Lam Q_{\U{\mu  \mu_c( \la_c^{-1}(\s))}}  \|_{L^2} & \lesssim \s^{\frac{k}{2}}  \\
\abs{\ang{ \Lam Q_{\U \mu} \mid \dot \Phi(\mu, \s) - a_c( \la_c^{-1}(\s))  \Lam Q_{\U{\mu  \mu_c( \la_c^{-1}(\s))}}} }& \lesssim  \s^{\frac{3}{2} k - 1} 
}
\EQ{ \label{eq:eqPhi-est} 
 \| r ( \De \Phi(\mu, \s) - r^{-2} f( \Phi(\mu, \s))) \|_{L^2}  \lesssim  \s^k 
}
\EQ{ \label{eq:eqPhi-est1} 
 \|  \De \Phi(\mu, \s) - r^{-2} f( \Phi(\mu, \s)) \|_{L^2}  \lesssim  \frac{\s^k}{\mu \s}  
}
where the $o(1)$ in all of the above denotes a constant that can be made as small as we like by taking $T_0$ large enough (and hence $\s_0$ small enough). 
\end{lem} 
\begin{proof}[Proof of Lemma~\ref{l:Phi-mus}] 
The estimate~\eqref{eq:Lam0-dotPhi} the estimate~\eqref{eq:dotPhi-bLamQ-H}  and the first estimates in~\eqref{eq:dotPhi-bLamQ} and~\eqref{eq:dotPhi-aLamQ} follow directly from the definition of $\dot \Phi(\mu, \s)$ along with the estimates~\eqref{eq:xis},~\eqref{eq:muc-lim}, and~\eqref{eq:ac-lim}. The second estimates in~\eqref{eq:dotPhi-bLamQ} and~\eqref{eq:dotPhi-aLamQ} follow from the same considerations along with~\cite[Lemma 3.5]{JL2-regularity}, which contains the standard estimates regarding the pairings of $A, B, \ti B, \Lam Q$ in $L^2$ at different scales. 

To prove~\eqref{eq:eqPhi-est} and~\eqref{eq:eqPhi-est1} we record the formula, 
\EQ{ \label{eq:eqPhi-mus} 
-\De \Phi(\mu, \s)  &+ \frac{1}{r^2} f( \Phi(\mu, \s))  \\
&=  \gamma_k \frac{\s^k}{\mu_c(\lam_c^{-1}(\s))^k \mu \s} \Lam Q_{\U{\mu \s}} - \frac{b_c(\la_c^{-1}(\s))^2}{\mu \s} \Lam_0 \Lam Q_{\U{\mu \s}}   \\
& + \gamma_k  \frac{\s^k}{\mu_c(\la_c^{-1}(\s))^{k+1} \mu}  \Lam Q_{\U{\mu_c(\lam_c^{-1}(\s)) \mu}} + \frac{a_c(\la_c^{-1}(\s))^2}{\mu_c(\lam_c^{-1}(\s)) \mu} \Lam_0 \Lam Q_{\U{\mu_c(\lam_c^{-1}(\s)) \mu}}  \\
 &-\frac{1}{r^2} \Big( f(Q_{\mu \s}  - Q_{\mu_c(\lam_c^{-1}(\s)) \mu})  - f(Q_{\mu \s} ) + f(Q_{\mu_c(\lam_c^{-1}(\s)) \mu})  \\
 &\qquad \quad   - 4  \big(\frac{r}{ \mu_c(\lam_c^{-1}(\s)) \mu}\big)^k (\La Q_{\mu \s})^2  -  4 \big(\frac{r}{ \mu \s}\big)^{-k} (\La Q_{\mu_c(\lam_c^{-1}(\s)) \mu})^2 \Big)\\
& -\frac{1}{r^2} \Big( f(\Phi) - f(Q_{\mu \s} - Q_{\mu_c(\lam_c^{-1}(\s)) \mu})  \\
&\qquad \quad -  f'(Q_{\mu \s}- Q_{\mu_c(\lam_c^{-1}(\s)) \mu}) \big(b_c(\lam_c^{-1}(\s))^2 T_{\mu \s} - a_c(\lam_c^{-1}(\s))^2 \ti T_{\mu_c(\lam_c^{-1}(\s)) \mu} \big)\Big)  \\
& -\frac{1}{r^2}  \Big(f'(Q_{\mu \s} - Q_{\mu_c(\lam_c^{-1}(\s)) \mu}) \big(b_c(\lam_c^{-1}(\s))^2 T_{\mu \s} - a_c(\lam_c^{-1}(\s))^2 \ti T_{\mu_c(\lam_c^{-1}(\s)) \mu}\big)  \\
&\qquad \quad  - f'(Q_{\mu \s}) b_c(\lam_c^{-1}(\s))^2 T_{\mu \s}  + f'(Q_{\mu_c(\lam_c^{-1}(\s)) \mu}) a_c(\lam_c^{-1}(\s))^2 \ti T_{\mu_c(\lam_c^{-1}(\s)) \mu}\Big)
}
which is a rescaling of~\eqref{eq:eqPhi}. 
The estimate~\eqref{eq:eqPhi-est} now follows from~\eqref{eq:LLb2T} and ~\eqref{eq:ff'L2} with $\al =1$ along with~\eqref{eq:xis},~\eqref{eq:muc-lim}, and~\eqref{eq:ac-lim}. The estimate~\eqref{eq:eqPhi-est1} follows from the same considerations using~\eqref{eq:LLb2T} and ~\eqref{eq:ff'L2} with $\al =2$. 
\end{proof}

\begin{proof} 
First we note that~\eqref{eq:dotU1} is a direct consequence of~\eqref{eq:dotPhi-bLamQ} along with~\eqref{eq:xis} and~\eqref{eq:wc-mus-H}. The estimate~\eqref{eq:dotU-H} is a direct consequence of~\eqref{eq:dotPhi-bLamQ-H} along with~\eqref{eq:xis} and~\eqref{eq:wc-mus-H2}. 

To prove~\eqref{eq:eqU} and~\eqref{eq:eqU-L2} we use the decomposition~\eqref{eq:U-Phi-w} to write, 
\EQ{
-\De U(\mu, \s) &+ r^{-2}f(U(\mu, \s))   = -\De \Phi(\mu, \s)  + r^{-2} f( \Phi(\mu, \s)) \\
 &  + r^{-2} ( f(\Phi(\mu, \s) + w_c(\mu, \s)) - f( \Phi(\mu, \s)) - f'( \Phi(\mu, \s)) w_c(\mu, \s) )  \label{eq:eqU1}  \\
& - \De w_c(\mu, \s) + r^{-2} f'( \Phi(\mu, \s)) w_c(\mu, \s)  
}
First we prove~\eqref{eq:eqU}. For the first line in~\eqref{eq:eqU1} we apply~\eqref{eq:eqPhi-est}.  
%the estimate, 
%\EQ{
% \| r(-\De \Phi(\mu, \s)  + r^{-2} f( \Phi(\mu, \s)) ) \|_{L^2} \lesssim \s^k % \s^{2k-1}  
%}
%which follows from~\eqref{eq:LLb2T} and ~\eqref{eq:ff'L2} with $\al =1$ along with~\eqref{eq:xis},~\eqref{eq:muc-lim}, and~\eqref{eq:ac-lim}. 
For the second line we use~\eqref{eq:w^2L21} along with~\eqref{eq:wc-mus-H} to obtain, 
\EQ{
\| r^{-1} ( f(\Phi(\mu, \s) + w_c(\mu, \s)) - f( \Phi(\mu, \s)) - f'( \Phi(\mu, \s)) w_c(\mu, \s)) \|_{L^2} \lesssim \| w_c(\mu, \s) \|_{H}^2 \lesssim \s^{3k-2}
}
Finally, for the last line in~\eqref{eq:eqU1} we have the estimates, 
\EQ{
\| r \De w_c(\mu, \s) \|_{L^2} \lesssim \| \Lam w_c(\mu, \s) \|_{H} \lesssim \s^{k-1} 
}
which follows from~\eqref{eq:wc-mus-LamH} and 
\EQ{
\| r^{-1} f'( \Phi(\mu, \s)) w_c(\mu, \s) \|_{L^2} \lesssim \| w_c (\mu, \s) \|_{H} \lesssim \s^{\frac{3}{2} k - 1}
}
which follows from~\eqref{eq:wc-mus-H}. This proves~\eqref{eq:eqU}. The proof of~\eqref{eq:eqU-L2} is similar. 

Next, we prove~\eqref{eq:D2mu}. First, using~\eqref{eq:pmuU} we have  
\EQ{ \label{eq:D2-mu1} 
\ang{ D^2\E(\bs U(\mu, \s) \p_\mu \bs U(\mu, \s)  \mid \bs h} %&= D^2\E(\bs U(\mu, \s) ( -\frac{1}{\mu} \Lam U(\mu, \s)  , \, -\frac{1}{\mu} \Lam_0 \dot U(\mu, \s)) \\
& = -\frac{1}{\mu} \ang{ \LL_{U(\mu, \s)}   \Lam U(\mu, s) \mid h} - \frac{1}{\mu} \ang{ \Lam_0 \dot U(\mu, \s) \mid \dot h} 
}
To treat the second term on the right above we use~\eqref{eq:U-Phi-w} to estimate, 
\EQ{
\abs{\frac{1}{\mu} \ang{ \Lam_0 \dot U(\mu, \s) \mid \dot h}} &\lesssim \frac{1}{\mu} ( \| \Lam_0 \dot  \Phi(\mu, \s) \|_{L^2} + \| \Lam_0 \dot w_c(\mu, \s) \|_{L^2}) \| \dot h \|_{L^2} 
 \lesssim  \frac{ \s^{\frac{k}{2}}}{\mu} \| \dot h \|_{L^2}
}
where we have used the estimates~\eqref{eq:Lam0-dotPhi} and~\eqref{eq:wc-mus-LamH} in the second inequality above. 
To handle the first term on the right of~\eqref{eq:D2-mu1} we make use of the fact that $\LL \Lam Q =0$ along with the decomposition~\eqref{eq:U-Q-g} to write, 
\EQ{
\LL_{U(\mu, \s)} \Lam U(\mu, \s) &= (\LL_{U(\mu, \s)} - \LL_{\mu\s})\Lam Q_{\mu \s}  + (\LL_{U(\mu, \s)} - \LL_{\mu_c(\lam_c^{-1}(\s)) \mu}) \Lam Q_{\mu_c(\lam_c^{-1}(\s)) \mu}  \\
&\quad +  \LL_{U(\mu, \s)} \Lam g_c(\mu, \s)
}
We use the weighted estimate~\eqref{eq:gc-LamH} to treat the contribution of last term above, 
\EQ{
\abs{ \ang{ \frac{1}{\mu}  \LL_{U(\mu, \s)} \Lam g_c(\mu, \s) \mid h}} \lesssim \frac{1}{\mu}  \| \Lam g_c(\mu, \s) \|_H \| h \|_H \lesssim \frac{\s^{k}}{\mu} \| h \|_H
}
To treat the first term note that, 
\EQ{ \label{eq:LLU-LLmus} 
(\LL_{U(\mu, \s)} - \LL_{\mu\s})\Lam Q_{\mu \s}& = r^{-2} \Big( f'( U(\mu, \s)) -  f'( Q_{\mu \s}) \Big)   \Lam Q_{\mu \s} \\ 
& = r^{-2}  \Big(f'(Q_{\mu \s}- Q_{\mu_c( \lam_c^{-1}(\s)) \mu}) - f'( Q_{\mu \s}) \Big)  \Lam Q_{\mu \s} + r^{-2} O( g_c( \mu, \s) \Lam Q_{\mu \s})
%\Big( f'( Q_{\mu \s}- Q_{\mu_c( \lam_c^{-1}(\s)) \mu} + g_c(\mu, \s)) - f'(Q_{\mu \s}- Q_{\mu_c( \lam_c^{-1}(\s)) \mu})  - f''(
}
Using the pointwise estimate, 
\EQ{
\abs{f'(Q_{\mu \s}- Q_{\mu_c( \lam_c^{-1}(\s)) \mu}) - f'( Q_{\mu \s})}  \Lam Q_{\mu \s} \lesssim  \Lam Q_{\mu_c( \lam_c^{-1}(\s)) \mu}^2 \Lam Q_{\mu \s} + \Lam Q_{\mu_c( \lam_c^{-1}(\s)) \mu} \Lam Q_{\mu \s}^2
}
we deduce that, 
\begin{multline} 
\abs{\frac{1}{\mu} \ang{  r^{-2}  \Big(f'(Q_{\mu \s}- Q_{\mu_c( \lam_c^{-1}(\s)) \mu}) - f'( Q_{\mu \s}) \Big)  \Lam Q_{\mu \s} \mid h }} \\
 \lesssim \frac{1}{\mu}  \Big( \| r^{-1} \Lam Q_{\mu_c( \lam_c^{-1}(\s)) }^2 \Lam Q_{ \s} \|_{L^2} + \| \Lam Q_{\mu_c( \lam_c^{-1}(\s)) } \Lam Q_{ \s}^2 \|_{L^2} ) \| h \|_{H}  
  \lesssim \frac{\s^{k}}{\mu} \| h \|_{H} 
\end{multline} 
We can also use~\eqref{eq:gc-H} to deduce that, 
\EQ{
\abs{\frac{1}{\mu}  \ang{ r^{-2}  g_c( \mu, \s) \Lam Q_{\mu \s} \mid h}} \lesssim \frac{\s^{k}}{\mu} \| h \|_{H} 
}
where in the last line we used~\eqref{eq:muc-lim}. 
And thus, 
\EQ{
\abs{ \frac{1}{\mu}  \ang{ (\LL_{U(\mu, \s)} - \LL_{\mu\s})\Lam Q_{\mu \s} \mid h}} \lesssim  \frac{\s^{k}}{\mu}  \| h \|_H 
}
The second term in~\eqref{eq:D2-mu1} is treated in the same way. This completes the proof of~\eqref{eq:D2mu}. 

Lastly, we prove~\eqref{eq:D2sig}. Using~\eqref{eq:psU} we have 
\EQ{ \label{eq:D2-sig1} 
\ang{ D^2 \E( \bs U(\mu, \s)) \p_\s \bs U(\mu, \s)  \mid  \bs h} &= -\frac{\mu}{\xi(\s)} \ang{ \De U(\mu, \s) - r^{-2} f(U(\mu, \s)) \mid  \dot h} \\
& \quad  - \frac{\mu}{\xi(\s)} \ang{  \LL_{U(\mu, \s)} \dot U(\mu,  \s) \mid h} 
}
For the first term on the right above we recall the decomposition~\eqref{eq:eqU1} to obtain, 
\EQ{
 -\frac{\mu}{\xi(\s)}& \ang{ \De U(\mu, \s) - r^{-2} f(U(\mu, \s)) \mid  \dot h} = -\frac{\mu}{\xi(\s)} \ang{ -\De \Phi(\mu, \s)  + r^{-2} f( \Phi(\mu, \s))  \mid  \dot h}\\
 &  -\frac{\mu}{\xi(\s)} \ang{r^{-2} ( f(\Phi(\mu, \s) + w_c(\mu, \s)) - f( \Phi(\mu, \s)) - f'( \Phi(\mu, \s)) w_c(\mu, \s) )   \mid \dot h}   \\
&  -\frac{\mu}{\xi(\s)}\ang{-\De w_c(\mu, \s) \mid \dot h}   -\frac{\mu}{\xi(\s)}\ang{ r^{-2} f'( \Phi(\mu, \s)) w_c(\mu, \s)  \mid \dot h}  
}
For the first line above we use the expansion~\eqref{eq:eqU} along with the estimates~\eqref{eq:ff'L2} with $\al =2$ and the asymptotics~\eqref{eq:xis},\eqref{eq:muc-lim}, and~\eqref{eq:ac-lim}    to obtain, 
\EQ{ \label{eq:D2-sig-main} 
-\frac{\mu}{\xi(\s)} &\ang{ -\De \Phi(\mu, \s)  + r^{-2} f( \Phi(\mu, \s))  \mid  \dot h} \\ 
&=  \gamma_k \frac{\s^{\frac{k}{2}}}{  \rho_k \s} \ang{\Lam Q_{\U{\mu \s}}  \mid \dot h} - \frac{ \rho_k \s^{\frac{k}{2}}}{ \s} \ang{ \Lam_0 \Lam Q_{\U{\mu \s}} \mid \dot h}  + o(1) \frac{\s^{\frac{k}{2}}}{\s} \|  \dot h \|_{L^2}   \\
%& + \gamma_k  \frac{\s^k}{\mu_c(\la_c^{-1}(\s))^{k+1} \mu}  \Lam Q_{\U{\mu_c(\lam_c^{-1}(\s)) \mu}} + \frac{a_c(\la_c^{-1}(\s))^2}{\mu_c(\lam_c^{-1}(\s)) \mu} \Lam_0 \Lam Q_{\U{\mu_c(\lam_c^{-1}(\s)) \mu}}  \\
% &-\frac{1}{r^2} \Big( f(Q_{\mu \s}  - Q_{\mu_c(\lam_c^{-1}(\s)) \mu})  - f(Q_{\mu \s} ) + f(Q_{\mu_c(\lam_c^{-1}(\s)) \mu})  \\
% &\qquad \quad   - 4  \big(\frac{r}{ \mu_c(\lam_c^{-1}(\s)) \mu}\big)^k (\La Q_{\mu \s})^2  -  4 \big(\frac{r}{ \mu \s}\big)^{-k} (\La Q_{\mu_c(\lam_c^{-1}(\s)) \mu})^2 \Big)\\
%& -\frac{1}{r^2} \Big( f(\Phi) - f(Q_{\mu \s} - Q_{\mu_c(\lam_c^{-1}(\s)) \mu})  \\
%&\qquad \quad -  f'(Q_{\mu \s}- Q_{\mu_c(\lam_c^{-1}(\s)) \mu}) \big(b_c(\lam_c^{-1}(\s))^2 T_{\mu \s} - a_c(\lam_c^{-1}(\s))^2 \ti T_{\mu_c(\lam_c^{-1}(\s)) \mu} \big)\Big)  \\
%& -\frac{1}{r^2}  \Big(f'(Q_{\mu \s} - Q_{\mu_c(\lam_c^{-1}(\s)) \mu}) \big(b_c(\lam_c^{-1}(\s))^2 T_{\mu \s} - a_c(\lam_c^{-1}(\s))^2 \ti T_{\mu_c(\lam_c^{-1}(\s)) \mu}\big)  \\
%&\qquad \quad  - f'(Q_{\mu \s}) b_c(\lam_c^{-1}(\s))^2 T_{\mu \s}  + f'(Q_{\mu_c(\lam_c^{-1}(\s)) \mu}) a_c(\lam_c^{-1}(\s))^2 \ti T_{\mu_c(\lam_c^{-1}(\s)) \mu}\Big)
}
which reveals the leading order terms that appear on the left-hand side of~\eqref{eq:D2sig}.  
For the next line we have, 
\EQ{
 |\frac{\mu}{\xi(\s)} &\ang{r^{-2} ( f(\Phi(\mu, \s) + w_c(\mu, \s)) - f( \Phi(\mu, \s)) - f'( \Phi(\mu, \s)) w_c(\mu, \s) )   \mid \dot h} |   \\
 & \lesssim  \frac{\mu}{\xi(\s)} \| r^{-2} w_c( \mu, \s)^2 \|_{L^2} \| \dot h \|_{L^2} \lesssim \s^{-\frac{k}{2}} \mu \| w_c(\mu, \s) \|_{H} \| w_c(\mu, \s)\|_{H^2}   \| \dot h \|_{L^2} \lesssim  \s^{ \frac{5}{2}k -3} \| \dot h \|_{L^2} 
}
where in the last line we used~\eqref{eq:wc-mus-H} and~\eqref{eq:wc-mus-H2}. Finally, for the last line we again use~\eqref{eq:wc-mus-H2}  to obtain,  
\EQ{
&\abs{  \frac{\mu}{\xi(\s)}\ang{-\De w_c(\mu, \s) \mid \dot h} } \lesssim \mu \s^{-\frac{k}{2}} \| w_c(\mu, \s) \|_{H^2}  \| \dot h \|_{L^2}  \lesssim \s^{k-1} \frac{\| \dot h \|_{L^2} }{\s}  \\
& \abs{\frac{\mu}{\xi(\s)}\ang{ r^{-2} f'( \Phi(\mu, \s)) w_c(\mu, \s)  \mid \dot h}} \lesssim  \s^{k-1} \frac{\| \dot h \|_{L^2} }{\s}
}
which completes the estimates for the first term in~\eqref{eq:D2-sig1}. To treat the second term in~\eqref{eq:D2-sig1} we expand using the fact that $\LL \Lam Q = 0$ as follows, 
\EQ{ \label{eq:LU-dotU} 
\LL_{U(\mu, \s)}  \dot U(\mu, \s) &=  b_c(\lam_c^{-1}(\s))(\LL_{U(\mu, \s)} - \LL_{\mu \s}) \Lam Q_{\U{\mu \s}}  \\
& \quad + a_c(\lam_c^{-1}(\s))(\LL_{U(\mu, \s)} - \LL_{\mu \mu_c(\lam_c^{-1}(\s))}) \Lam Q_{\U{\mu_c(\lam_c^{-1}(\s)) \mu}}  \\
& \quad + \LL_{U(\mu \s)}  \Big( \dot \Phi(\mu, \s)-  b_c(\lam_c^{-1}(\s)) \Lam Q_{\U{\mu \s}}- a_c(\lam_c^{-1}(\s))\Lam Q_{\U{\mu_c(\lam_c^{-1}(\s)) \mu}}\Big)  \\
& \quad +  \LL_{U(\mu \s)}  \dot w_c(\mu, \s) 
}
The contribution of the last term above to~\eqref{eq:D2sig} is controlled by the estimate~\eqref{eq:wc-mus-H2}, 
\EQ{
\frac{\mu}{\xi(\s)} \abs{ \ang{ \LL_{U(\mu, s)}   \dot w_c(\mu, \s)  \mid h}} \lesssim \mu \s^{-\frac{k}{2}} \|  \dot w_c(\mu, \s) \|_{H} \| h \|_H & \lesssim \s^{k-1} \frac{ \| h \|_H}{\s}
}
Next, consider the first line in~\eqref{eq:LU-dotU}. Using~\eqref{eq:LLU-LLmus},~\eqref{eq:xis} we have, 
\EQ{
\frac{\mu}{\xi(\s)} \Big| b_c(\lam_c^{-1}(\s)) \langle(\LL_{U(\mu, \s)} &- \LL_{\mu \s}) \Lam Q_{\U{\mu \s}}   \mid h \rangle \Big| \lesssim  \frac{\mu}{\mu \s} \| r^{-1}  \Lam Q_{\mu_c( \lam_c^{-1}(\s)) \mu}^2 \Lam Q_{\mu \s} \|_{L^2} \| h \|_H \\
&  + \frac{\mu}{\mu \s}  \| r^{-1}  \Lam Q_{\mu_c( \lam_c^{-1}(\s)) \mu} \Lam Q_{\mu \s}^2 \|_{L^2} \| h \|_H  +  \frac{\mu}{\mu \s}  \| g_c( \mu, \s) \|_H  \| h \|_H \\
& \lesssim  \s^k \frac{ \| h \|_H}{\s}  
}
where the last inequality follows from the estimates~\eqref{eq:ABBQL21} from Lemma~\ref{l:ABBQL21} along with the estimate~\eqref{eq:gc-H}. The contribution of the second line in~\eqref{eq:LU-dotU} to~\eqref{eq:D2sig} is handled similarly. Finally, the estimate, 
\EQ{
\frac{\mu}{\xi(\s)}  \Big| \ang{  \LL_{U(\mu \s)}  \Big( \dot \Phi(\mu, \s)-  b_c(\lam_c^{-1}(\s)) \Lam Q_{\U{\mu \s}}- a_c(\lam_c^{-1}(\s))\Lam Q_{\U{\mu_c(\lam_c^{-1}(\s)) \mu}}\Big) \mid h} \Big| \lesssim \s^k \frac{ \| h \|_H}{\s}
}
follows directly from the definition of $\dot \Phi(\mu, \s)$ along with~\eqref{eq:xis},~\eqref{eq:muc-lim}, and~\eqref{eq:ac-lim}. Plugging the preceding estimates back into~\eqref{eq:D2-sig1} we obtains, 
\EQ{
\ang{ D^2 \E( \bs U(\mu, \s)) \p_\s \bs U(\mu, \s)  \mid  \bs h} &= \gamma_k \frac{\s^{\frac{k}{2}}}{  \rho_k \s} \ang{\Lam Q_{\U{\mu \s}}  \mid \dot h} - \frac{ \rho_k \s^{\frac{k}{2}}}{ \s} \ang{ \Lam_0 \Lam Q_{\U{\mu \s}} \mid \dot h}  + o(1)O( \frac{\s^{\frac{k}{2}}}{\s} \|  \bs h \|_{\HH})
}
as claimed.
\end{proof} 
\subsection{Modulation around $\bs U(\mu, \s)$} 

Next we modulate around $\bs U( \mu, \s)$. 

\begin{lem}[Modulation Lemma] \label{l:mod}  There exists an $\eta_0>0$ small enough so that the following statement holds true.  Let $J$ be a time interval, and let $\bs u: J \to \HH$ be a solution to~\eqref{eq:wmk} such that 
\EQ{
 \hbfd ( \bs u(t)) < \eta  \le \eta_0 \quad \forall t \in J
}
Then there exist $C^{1}(J; (0, \infty))$ functions $\mu(t), \s(t)$ such that defining $\bs g(t) \in \HH$ by 
\EQ{ \label{eq:gdef}
\bs g(t) = \bs u(t) - \bs U( \mu(t), \s(t))
}
we have, for each $t\in J$, 
\begin{align}  
\label{eq:omu} &\ang{\La Q_{\U\mu} \mid g} = 0 \\ 
 \label{eq:ola}  & \ang{\La Q_{\U{\s \mu}} \mid g} = 0
\end{align} 
In addition, there exists a uniform constant $c_1>0$ such that
\EQ{\label{eq:coer} 
\ang{ D^2 \E( \bs U( \mu(t), \s(t)) \bs g(t) \mid \bs g(t)} \ge c_1 \| \bs g(t) \|_{\HH_0}^2 
}
Finally, we have the estimates, 
\begin{align}
\label{eq:dtmu}  \abs{\mu'(t)} &\lesssim    \| \dot g \|_{L^2} + \s^{\frac{k}{2}} \| g \|_{H}  \lesssim  \| \bs g \|_{\HH} \\ % \|  \dot g \|_{L^2} + \s^{k+1} \| g \|_{H}   \lesssim \| \bs g \|_{\HH_0} \\
\label{eq:dtsig}  \Big|\mu(t) \s'(t) + \xi( \s(t)) + \frac{\ang{\La Q_{\U{\la(t)}} \mid \dot g}}{ \|\La Q \|_{L^2}^2} \Big|  &\lesssim     o(1)  \| \dot g \|_{L^2}  +  \s^{\frac{k}{2}} \|g \|_{H} \lesssim o(1)  \| \bs g \|_{\HH}
\end{align} 
where the $o(1)$ term above can be taken as small as we like by taking $\eta>0$ small. %where $c_0(\eta) \to 0 \mas \eta \to 0$. 
\end{lem} 

\begin{rem} 
By~\eqref{eq:hatd-0} we  may apply Lemma~\ref{l:mod} to the arbitrary 2-bubble solution~$\bs u(t)$ on the time interval $J=[T_0, \infty)$ for large enough $T_0>0$, and we obtain a decomposition
\EQ{
\bs u(t) = \bs U( \mu(t), \s(t)) + \bs g(t) , \quad \forall \, \, t \in J
} 
 Note that by~\eqref{eq:hatd-0} and the proof of Lemma~\ref{l:mod}  we obtain the qualitative behavior,  %for any $\eta_0>0$ there exists $T_0>0$ large enough so that $\bs g(t), \s(t)$ satisfy  
%\Red{To conclude will need certain qualitative properties of the parameter $\s(t)$ and the function $\bs g(t)$. For example, we need to be able to say that for any $\eps >0$ there exists $T_0>0$ such that 
\EQ{ \label{eq:g-sig-small} 
 \| \bs g(t) \|_{\HH}^2 + \s(t)^k  \to 0 \mas t \to \infty.
} 
%This should follow directly from the modulation lemma and the fact that $\bfd_+( \bs u(t)) \to 0 $ as $ t \to \infty$.}
\end{rem} 

Before proving Lemma~\ref{l:mod} we record several identities that will be used throughout the rest of the paper. We use the notation 
\EQ{
\bs g(t) =: (g(t), \dot g(t)), \quad \bs U( \mu, \s) := (U( \mu, \s), \dot U (\mu, \s))
}
The equation satisfied by $\bs g(t)$ as defined in~\eqref{eq:gdef} is 
\EQ{
\p_t \bs g(t) &=  \p_t \bs u(t) - \p_t \bs U( \mu(t), \s(t)) \\
& = J \circ \uD \E(\bs u(t))  - \mu'(t) \p_\mu \bs U( \mu, \s) - \s' (t) \p_\s \bs U( \mu, \s) \\
& = J \circ \Big(\uD \E(\bs U(\mu, \s) + \bs g) -   \uD \E(\bs U(\mu, \s))\Big)  \\
& \quad - \mu' \p_\mu \bs U( \mu, \s) - \big( \s' \p_\s \bs U( \mu, \s) - J \circ \uD \E( \bs U(\mu, \s)) \big) 
}
We use~\eqref{eq:psU} to rewrite the last line above, obtaining,  
\EQ{ \label{eq:eq-g-ham} 
\p_t \bs g(t) &=J \circ \Big(\uD \E(\bs U(\mu, \s) + \bs g) -   \uD \E(\bs U(\mu, \s))\Big) \\
& \quad  - \mu' \p_\mu \bs U( \mu, \s) - \big( \s'   + \frac{\xi(\s)}{\mu} \Big)  \p_\s \bs U( \mu, \s)   \\ 
& = J \circ \Big(\uD \E(\bs U(\mu, \s) + \bs g) -   \uD \E(\bs U(\mu, \s))\big) \\
& \quad  - \mu' \p_\mu \bs U( \mu, \s) - \Big( \s'   + \frac{\rho_k \s^{\frac{k}{2}}}{\mu} \Big)  \p_\s \bs U( \mu, \s) +  \frac{1}{\mu}\big( \rho_k \s^{\frac{k}{2}} - \xi(\s)  \Big)  \p_\s \bs U( \mu, \s)
}
In components this reads, 
\EQ{ \label{eq:eq-g} 
 \p_t \pmat{ g \\ \dot g} =  \pmat{ \dot g - \mu' \p_\mu U( \mu, \s) -  ( \s' + \frac{\xi(\s)}{\mu}) \p_\s U(\mu, \s) \\  \De g - \frac{1}{r^2} \Big( f( U(\mu, \s) + g) - f( U(\mu, \s)) \Big)  - \mu' \dot{\p_\mu  U}(\mu, \s)   - ( \s' + \frac{\xi(\s)}{\mu})  \dot{\p_\s   U}(\mu, \s)} 
}

%We prove Lemma~\ref{l:mod}.  

\begin{proof}[Proof of Lemma~\ref{l:mod}]
The proof of the existence of $(\bs g(t), \mu(t), \s(t))$ as in the statement of the lemma is nearly identical to~\cite[Proof of Lemma $3.1$]{JL1} so we give only a brief sketch here, highlighting the differences. First note that since $\hbfd(\bs u(t))$ is small, we can find $\mu_1(t), \s_1(t)$ with $\s_1^k(t) \le \eta^2$ so that for $g_1(t)$ defined by 
\EQ{ \label{eq:g_1} 
g_1(t)  := u(t) - (Q_{\mu_1 \s_1} - Q_{\mu_1})
}
we have 
\EQ{ \label{eq:dg1-small} 
\| g_1(t) \|_H^2 + \s_1(t)^k  \le  2\eta^2
}
To simplify notation we will suppress the time-dependency in the expressions below. Define a mapping $F: H \times (0, \infty) \times (0, \infty) \to H$ by 
\EQ{
F(g, \mu, \s) := g - U(\mu, \s) + (Q_{\mu_1 \s_1} - Q_{\mu_1}) 
}
and recall that by definition of $\bs g_c(\mu, \s)$ and the estimate~\eqref{eq:gc-H}, 
\EQ{ \label{eq:Umsg} 
U(\mu, \s) &= Q_{\mu \s} - Q_{\mu \mu_c(\la_c^{-1}(\s))} + g_c(\mu, \s) , \quad 
 \|g_c(\mu, \s) \|_H \lesssim   \s^k 
}
It follows that 
\EQ{ \label{eq:F=0} 
F( g_c(\la_c^{-1}(\s_1))_{\mu_1} + Q_{\mu_1 \mu_c(\la_c^{-1}(\s_1))} - Q_{\mu_1}, \mu_1, \s_1) = 0 
}
and we have 
\EQ{
\| F(g, \mu, \s) \|_H &\le \|g \|_H + \|U(\mu, \s) - (Q_{\mu \s} - Q_{\mu \mu_c(\la_c^{-1}(\s))}) \|_H  + \| Q_{\mu \s} - Q_{\mu_1 \s_1} \|_H\\
& \quad +  \| Q_{\mu \mu_c(\la_c^{-1}(\s))} - Q_{\mu_1\mu_c(\la_c^{-1}(\s))} \|_{H} +  \| Q_{\mu_1\mu_c(\la_c^{-1}(\s))} - Q_{\mu_1} \|_H  \\
&   \lesssim  \|g \|_H + \| g_c(\la_c^{-1}(\s)) \|_H + \abs{\frac{\mu \s}{\mu_1 \s_1} - 1}^{\frac{1}{2}}   + \abs{ \frac{\mu}{\mu_1} - 1 }^{\frac{1}{2}} + \abs{ \mu_{c}( \la_c^{-1}(\s)) - 1 }^{\frac{1}{2}}
}
Next, define a mapping $G: H \times (0, \infty) \times (0, \infty) \to \R^2$ by 
\EQ{
G(g, \mu, \s) = \pmat{    \frac{1}{\mu} \ang{ \La Q_{\U\mu} \mid F(g, \mu, \s)} \\  \frac{1}{\mu \s} \ang{  \La Q_{\U{\mu \s}}  \mid F(g, \mu, \s)}}
}
Using~\eqref{eq:F=0} we have 
\EQ{
G(g_c(\la_c^{-1}(\s_1))_{\mu_1} + Q_{\mu_1 \mu_c(\la_c^{-1}(\s_1))} - Q_{\mu_1}, \mu_1, \s_1) = 0
}
Moreover, for any $h \in H$ we have the estimates  
\EQ{
\frac{1}{\mu}\abs{ \ang{ \La Q_{\U\mu} \mid h} } &\lesssim  \| r/ \mu \La Q_{\U \mu} \|_{L^2}  \| r^{-1} h \|_{L^2} \lesssim \| h \|_H \\
\frac{1}{\mu\s}\abs{ \ang{ \La Q_{\U{\mu\s}} \mid h}}  &\lesssim  \| r( \mu \s)^{-1}  \La Q_{\U {\mu \s}} \|_{L^2}  \| r^{-1} h \|_{L^2} \lesssim \| h \|_H 
}
which ensures that $G$ is well defined and continuous. As in~\cite[Proof of Lemma $3.1$]{JL1} one can now readily check that the implicit function theorem can applied to $G$, meaning that for each $g_0$ in a small enough neighborhood (of size $\simeq \eta_0$) of $g_c(\la_c^{-1}(\s_1))_{\mu_1} + Q_{\mu_1 \mu_c(\la_c^{-1}(\s_1))} - Q_{\mu_1}$, we can find unique $(\mu_0, \s_0) = \si(g_0)$ (for the function $\si$ given by the implicit function theorem) in a neighborhood of $(\mu_1, \s_1)$ (we note that it is convenient here to work in the variables, $s=\log \s, m := \log \mu$) and so that 
\EQ{
G( g_0, \mu_0, \s_0) = 0
}
We refer the reader to~\cite[Lemma 3.1 and Remark 3.2]{JL1} for precise details on the version and implementation of the implicit function in this setting.   
The desired triple $(g, \mu, \s)$ as in the lemma is then given by 
\EQ{
(\mu, \s) := \si(g_1), \quad g:= F(g_1, \mu, \s)
}
where $g_1$ is as in~\eqref{eq:g_1}, as long as $g_1$ is close enough in $H$ to $g_c(\la_c^{-1}(\s_1))_{\mu_1} + Q_{\mu_1 \mu_c(\la_c^{-1}(\s_1))} - Q_{\mu_1}$. To see this we measure, 
\EQ{
\| g_1 - g_c(\la_c^{-1}(\s_1))_{\mu_1} + Q_{\mu_1 \mu_c(\la_c^{-1}(\s_1))} - Q_{\mu_1} \|_H &\le \| g_1 \|_H + \| g_c(\la_c^{-1}(\s_1))\|_H  \\
& \quad + \| Q_{\mu_1 \mu_c(\la_c^{-1}(\s_1))} - Q_{\mu_1} \|_H  \\
& \lesssim \s_1^k  + \abs{ \mu_c(\la_c^{-1}(\s_1)) - 1}^{\frac{1}{2}}  \lesssim \eta
}
where the last line above follows from~\eqref{eq:Umsg}, ~\eqref{eq:dg1-small}, and~\eqref{eq:muc-lim}. Thus $\mu, \s$ and $g$ are well-defined. To conclude, we note that it follows from the definition of $F$ that 
\EQ{
g = F(g_1, \si(g_1)) = F(g_1, \mu, \s) =  u - U(\mu, \s) 
}
and from the definition of $G$ that  
\EQ{
\ang{\La Q_{\U{\mu \s}} \mid g} = 0 \mand \ang{\La Q_{\U{\mu }} \mid g} = 0
}
as desired. 
%Hence, it remains to show 
%\EQ{
%\|  g_1 - g_c(\la_c^{-1}(\s_1))_{\mu_1} + Q_{\mu_1 \mu_c(\la_c^{-1}(\s_1))} - Q_{\mu_1} \|_H \ll 1
%}

The coercivity estimate follows from a standard argument using the orthogonality conditions~\eqref{eq:omu} and~\eqref{eq:ola} together with the localized coercivity Lemma~\ref{l:loc-coerce}. Indeed,  the smallness of $\s$  yields a uniform constant $c_1>0$ for which 
\EQ{
\int_0^\infty (\p_r g)^2 + k^2 \frac{ \cos(2Q_{\mu \s} - \cos 2 Q_\mu)}{r^2}  g^2 \,  r \, \ud r \ge c_1 \| g \|_H^2
}
For a detailed proof of the above see~\cite[Lemma 5.4]{JJ-AJM}. Next, we argue perturbatively.  Note that 
\EQ{
\ang{ D^2 \E( \bs U( \mu, \s) \bs g \mid \bs g} &= \int_0^\infty (\p_r g)^2 + k^2 \frac{ \cos(2U(\mu, \s))}{r^2}  g^2 \,  r \, \ud r 
}
and, 
\EQ{
\cos 2 U(\mu, \s) &= \cos \big( 2 Q_{\mu \s} - 2 Q_\mu  + ( 2 Q_\mu - 2Q_{\mu_c(\la_c^{-1}(\s)) \mu}) + 2g_c( \la_c^{-1}(\s))\big) \\
& =  \cos \big( 2 Q_{\mu \s} - 2 Q_\mu \big) + O( \abs{Q_\mu - Q_{\mu_c(\la_c^{-1}(\s)) \mu}}) + O( \abs{g_c( \la_c^{-1}(\s))})
}
Thus, 
\EQ{
\ang{ D^2 \E( \bs U( \mu, \s) \bs g \mid \bs g}  &\ge c_1 \| g \|_H^2 - O\Big( \abs{\mu_c(\la_c^{-1}(\s)) - 1}^{\frac{1}{2}}) + \| g_c( \la_c^{-1}(\s))\|_H \Big) \|g \|_{H}^2  \\
& \ge c_0 \| g \|_H^2
}
where the last line follows by taking $\s>0$ small enough. This completes the proof of~\eqref{eq:coer}.

Next we prove the estimates~\eqref{eq:dtmu} and~\eqref{eq:dtsig}. We differentiate the modulation equations, beginning with~\eqref{eq:omu}, 
\EQ{
0 = \frac{\ud }{\ud t} \ang{ \Lam Q_{\U \mu}  \mid g}  &= -\frac{\mu'}{\mu}\ang{ [\La_0 \Lam Q]_{\U \mu} \mid  g} + \ang{ \Lam Q_{\U \mu} \mid \p_t g} \\
& = -\frac{\mu'}{\mu}\ang{ [\La_0 \Lam Q]_{\U \mu} \mid  g}   + \ang{ \Lam Q_{\U \mu} \mid \dot g } - \mu' \ang{ \Lam Q_{\U \mu} \mid \p_\mu U( \mu, \s) }   \\ 
& \quad - \Big( \s' + \frac{\xi(\s)}{\mu}\Big) \ang{ \Lam Q_{\U \mu} \mid \p_\s U(\mu, \s)} 
}
Rearranging the above gives
\EQ{
 \ang{ \Lam Q_{\U \mu} \mid \dot g }  &=  \mu' \Big(\ang{ \Lam Q_{\U \mu} \mid \p_\mu U( \mu, \s) }   + \frac{1}{\mu} \ang{ [\La_0 \Lam Q]_{\U \mu} \mid  g} \Big) \\ 
 &\quad + \Big( \mu \s' + \xi(\s)\Big) \frac{1}{\mu} \ang{ \Lam Q_{\U \mu} \mid \p_\s U(\mu, \s)} 
}
Next write $\la := \s \mu$, and note that by the chain rule we have 
\EQ{ \label{eq:la'}
\frac{\la'}{\la}  = \frac{\s'}{\s} + \frac{\mu'}{\mu} 
}
Differentiating~\eqref{eq:ola} gives 
\EQ{
0 = \frac{\ud }{\ud t} \ang{ \La Q_{\U \la}  \mid g}  &= -\frac{\la'}{\la} \ang{[\La_0 \La Q]_{\U \la}  \mid g} + \ang{ \La Q_{\U \la} \mid \p_t g } \\
& = -\frac{\la'}{\la} \ang{[\La_0 \La Q]_{\U \la}  \mid g} + \ang{ \La Q_{\U \la} \mid \dot g } - \mu' \ang{ \La Q_{\U \la} \mid \p_\mu U( \mu, \s) }   \\ 
& \quad - \Big( \s' + \frac{\xi(\s)}{\mu}\Big) \ang{\La Q_{\U \la} \mid \p_\s U(\mu, \s)} 
}
which, using~\eqref{eq:la'} yields, 
\EQ{
  \ang{ \La Q_{\U \la} \mid \dot g } + \frac{\xi(\s)}{\mu \s}\langle [\La_0 \La Q]_{\U \la}  \mid & g \rangle = \mu' \Big( \ang{ \La Q_{\U \la} \mid \p_\mu U( \mu, \s) } + \frac{1}{\mu}  \ang{[\La_0 \La Q]_{\U \la}  \mid g} \Big)  \\
  &  + \big( \mu \s' + \xi(\s)\big) \Big(  \frac{1}{\mu} \ang{\La Q_{\U \la} \mid \p_\s U(\mu, \s)} +  \la^{-1} \ang{[\La_0 \La Q]_{\U \la}  \mid g} \Big)    
}
 We obtain the following system of equations, 
 \EQ{ \label{eq:Msys} 
 \pmat{M_{11} & M_{12} \\ M_{21} & M_{22}} \pmat{ \mu' \\ \mu \s' + \xi(\s)} =  \pmat{  \ang{ \Lam Q_{\U \mu} \mid \dot g }  \\ \ang{ \La Q_{\U \la} \mid \dot g } + \frac{\xi(\s)}{\mu \s}\ang{[\La_0 \La Q]_{\U \la}  \mid g}  } =:  \pmat{ B_1 \\ B_2}
 }
 where 
 \EQ{
 &M_{11} := \ang{ \Lam Q_{\U \mu} \mid \p_\mu U( \mu, \s) }   + \mu^{-1} \ang{ \La_0 \Lam Q_{\U \mu} \mid  g}  \\ 
 &M_{22}:=  \mu^{-1} \ang{\La Q_{\U \la} \mid \p_\s U(\mu, \s)} +   \lam^{-1} \ang{\La_0 \La Q_{\U \la}  \mid g} \\
 &M_{12}:=   \mu^{-1} \ang{ \Lam Q_{\U \mu} \mid \p_\s U(\mu, \s)} \\ 
 &M_{21}:= \ang{ \La Q_{\U \la} \mid \p_\mu U( \mu, \s) } + \mu^{-1}  \ang{[\La_0 \La Q]_{\U \la}  \mid g}
 }
% For the right-hand-side of the system~\eqref{eq:Msys} we define 
% \EQ{
% \pmat{ B_1 \\ B_2}:= \pmat{  \ang{ \cZ_{\U \mu} \mid \dot g }  \\ \ang{ \La Q_{\U \la} \mid \dot g } + \frac{\xi(\s)}{\mu \s}\ang{[\La_0 \La Q]_{\U \la}  \mid g}  } 
% }
% and note that 
Note the estimates, 
 \EQ{ \label{eq:B-est} 
& \abs{B_1} \lesssim  \| \dot g \|_{L^2} \\
 & \abs{ B_2} \lesssim  \| \dot g \|_{L^2}  +  \s^{\frac{k}{2}}  \|g \|_{H}  \mand  \abs{B_2 - \ang{  \La Q_{\U{\la}} \mid  \dot g} } \lesssim  \s^{\frac{k}{2}}  \|g \|_{H} 
 }
We claim the bounds 
\begin{claim}  \label{c:Mest} The following estimates hold true. 
 \begin{align}
 \label{eq:M11} \abs{ M_{11} -  \|\La Q \|_{L^2}^2} & \lesssim  \s^{\frac{1}{2}} +  \|  g \|_{H}\\
  \label{eq:M22}  \abs{M_{22} + (1- o(1))\|\La Q \|_{L^2}^2 } &\lesssim \s^k + \s \|g \|_{H}\\ 
\label{eq:M12} \abs{M_{12}} &\lesssim o(1) \\%\s^{\frac{k}{2}}\\
 \label{eq:M21} \abs{M_{21}} &\lesssim \s\\
 \label{eq:detM} \det{M} &= M_{11}M_{22} + O( \s)  %\\ %=  -(1- o(1)) \| \Lam Q \|_{L^2}^4  + O( \s^{\frac{1}{2}}) \\
 %\label{eq:detMinv} \frac{1}{\det{M}} &= \frac{1}{M_{11}M_{22}} + O( \s )
\end{align}
where $o(1)$ can be replaced by a constant that can be made as small as we like by taking $\eta$ small enough. 
 \end{claim} 
 
 \begin{proof}[Proof of Claim~\ref{c:Mest}] 
First we prove~\eqref{eq:M11}. Recall that 
$
\p_\mu U( \mu, \s) = -\frac{1}{\mu} \La U ( \mu, \s)
$.
%and thus changing variables we have 
%\EQ{
%\ang{ \Lam Q_{\U \mu} \mid \p_\mu U( \mu, \s) } = - \ang{ \Lam Q  \mid \La U(\mu, \s)_{\mu^{-1}}}
%}
Hence, 
\EQ{
 \abs{ M_{11} -  \| \Lam Q \|_{L^2}^2} \le \abs{\ang{ \Lam Q_{\U \mu} \mid  \frac{1}{\mu} (\La U(\mu, \s) + \La Q_{\mu})}} + \abs{ \mu^{-1} \ang{ [\La_0 \Lam Q]_{\U \mu} \mid  g}}
}
The second term on the right above can be bounded as follows:
\EQ{
 \abs{ \mu^{-1} \ang{ [\La_0 \Lam Q]_{\U \mu} \mid  g}} \lesssim  \| (r/ \mu) [\La_0 \Lam Q]_{\U \mu} \|_{L^2} \| r^{-1} g \|_{L^2} \lesssim  \|g \|_{H}
}
To control the first term on the right, we first write 
\EQ{
\La U( \mu, \s) = \La Q_{\s \mu} - \La Q_{\mu \mu_c( \la^{-1}(\s))} + \La g_c( \mu, \s)
}
so after rescaling we have, 
\EQ{
\ang{ \Lam Q \mid ( \La U(\mu, \s)_{\mu^{-1}} + \La Q} &\lesssim \abs{\ang{ \Lam Q \mid \La Q_{\s}}}+ \abs{\ang{ \Lam Q \mid \La Q - \La Q_{\mu_c( \la^{-1}(\s))}}}  \\
& \quad +  \abs{\ang{ \Lam Q\mid\La g_c( \la^{-1}(\s))}}
}
%Using that $\cZ \in C^{\infty}_0$ can be chosen to satisfy $r^{-\frac{k-1}{2}} \cZ(r) \lesssim 1$ we have
For the first term we have  
$
\abs{\ang{ \Lam Q \mid \La Q_{\s}}} \ll \s^{\frac{1}{2}} 
$.
Next, observe that by~\eqref{eq:gc-H}, 
\EQ{ \label{eq:Zgcpair} 
\abs{\ang{  \Lam Q \mid\La g_c( \la_c^{-1}(\s))}} \lesssim  \| r \Lam Q \|_{L^2} \| r^{-1}\La g_c( \la_c^{-1}(\s) \|_{L^2} \lesssim \|g_c( \la_c^{-1}(\s)\|_{H} \lesssim \s^{k}
}
Lastly, we use~\eqref{eq:muc-lim} to deduce that 
\EQ{
\abs{\ang{ \Lam Q \mid \La Q - \La Q_{\mu_c( \la_c^{-1}(\s))}}} \lesssim \abs{ \mu_c(\la_c^{-1}(\s)) -1 }^{\frac{1}{2}}  \lesssim \s^{\frac{1}{2}} 
}
Combining these estimates proves~\eqref{eq:M11}. Next we treat the term $M_{22}$.  
%Here it seems we have the estimate 
% \EQ{
%  \abs{M_{22} + \|\La Q \|_{L^2}^2 } \lesssim \s^k + \s \|g \|_{H}
% }
We have 
 \EQ{
 M_{22}   =  \ang{\La Q_{\U \la} \mid \frac{1}{\mu} \p_\s U(\mu, \s) } +  \lam^{-1}  \ang{[\La_0 \La Q]_{\U \la}  \mid g}
 }
For the second term above, we have 
 \EQ{
\abs{ \la^{-1} \ang{[\La_0 \La Q]_{\U \la}  \mid g}}  \lesssim  \| (r/\la) [\La_0 \La Q]_{\U \la} \|_{L^2} \| r^{-1} g \|_{L^2} \lesssim  \|g \|_{H}
 }
 By~\eqref{eq:psU} and the definition of $\dot U(\mu, \s)$ we have 
 \EQ{
\frac{1}{\mu}  \p_\s U(\mu, \s) = - \frac{1}{ \xi(\s)} \dot U(\mu, \s)  & = - \Lam Q_{\U{\mu \s} }+ \Big( 1 - \frac{b_{c}( \la_c^{-1}(\s))}{ \xi( \s)} \Big) \Lam Q_{\U {\mu \s}} \\
& \quad - \frac{1}{\xi(\s)}  \Big(  \dot \Phi(\mu, \s) - b_c( \la_c^{-1}(\s))  \Lam Q_{\U{\mu \s}} \Big) - \frac{1}{\xi(\s)} \dot w_c(\mu, \s)
 }
 It then follows from~\eqref{eq:xis},~\eqref{eq:dotPhi-bLamQ}, and~\eqref{eq:wc-mus-H} that, 
 \EQ{
 \ang{\La Q_{\U \la} \mid \frac{1}{\mu} \p_\s U(\mu, \s) } = -( 1- o(1)) \| \Lam Q\|_{L^2} + O(\s^{\frac{k}{2}})  
 }
 % \EQ{
% \abs{\frac{1}{\mu} \ang{\La Q_{\U \la} \mid \p_\s U(\mu, \s) + \mu \La Q_{\U \la}}} \le  \frac{1}{\mu} \| \La Q \|_{L^2}  \| \p_\s U(\mu, \s) + \mu \La Q_{\U \la} \|_{L^2}  \le c_0  \, (\lesssim  \s^k)
% }
 This proves~\eqref{eq:M22}. 
 To prove~\eqref{eq:M12} we  write, 
 \EQ{
 \frac{1}{\mu}  \p_\s U(\mu, \s) &= -\frac{a_c( \la_c^{-1}(\s))}{\xi(\s)}  \Lam Q_{\U{\mu \mu_c( \la_c^{-1}(\s))}}  - \frac{1}{\xi(\s)} \Big( \dot \Phi(\mu, \s) - a_c( \la_c^{-1}(\s))  \Lam Q_{\U{\mu  \mu_c( \la_c^{-1}(\s))}}  \Big) \\
 & \quad  - \frac{1}{\xi(\s)} \dot w_c(\mu, \s) 
 }
 and thus, using~\eqref{eq:xis}~\eqref{eq:muc-lim},~\eqref{eq:ac-lim},~\eqref{eq:wc-mus-H} and~\eqref{eq:dotPhi-aLamQ} we arrive at the estimate, 
 \EQ{
  \abs{M_{12}} =  \abs{ \ang{ \Lam Q_{\U \mu} \mid \frac{1}{\mu} \p_{\s} U(\mu, \s)}} \lesssim o(1) 
 }
%  \EQ{
% \abs{M_{12}}  = \abs{ \ang{ \Lam Q_{\U \mu} \mid \frac{1}{\mu} \p_{\s} U(\mu, \s)}} \le \abs{ \ang{ \Lam Q_{\U \mu} \mid \frac{1}{\mu} \p_{\s} U(\mu, \s)+ \La Q_{\U{\s \mu}}}} + \abs{\frac{1}{\mu} \ang{ \Lam Q_{\U \mu} \mid  \mu \La Q_{\U{\s \mu}}}}  
% }
% The first term on the right above is estimated using~\eqref{eq:psUL2} 
% \EQ{
% \abs{\frac{1}{\mu} \ang{ \cZ_{\U \mu} \mid  \p_{\s} U(\mu, \s)+ \mu \La Q_{\U{\s \mu}}}} \lesssim \frac{1}{\mu} \| \cZ \|_{L^2} \|  \p_{\s} U(\mu, \s)+ \mu \La Q_{\U{\s \mu}} \|_{L^2} \lesssim \s^k 
% }
% Te estimate the second term we note that $\cZ$ can be chosen so that $\abs{r^{-\frac{k}{2}} \cZ(r) }\lesssim 1$ so that \marginpar{\tiny{modify if $k=2$}}
% \EQ{
%  \abs{\frac{1}{\mu} \ang{ \cZ_{\U \mu} \mid  \mu \La Q_{\U{\s \mu}}}}  \lesssim \| r^{-\frac{k}{2}} \cZ \|_{L^2} \| r^{\frac{k}{2}}  \La Q_{\U \s} \|_{L^2} \lesssim \s^{\frac{k}{2}} 
% }
 Finally, we estimate~\eqref{eq:M21}. 
 \EQ{
 \abs{M_{21}}  \lesssim \abs{\ang{ \La Q_{\U \la} \mid \p_\mu U( \mu, \s) }} + \abs{\frac{1}{\mu}  \ang{[\La_0 \La Q]_{\U \la}  \mid g}}
 }
 The second term above is controlled as follows, 
 \EQ{
  \abs{\frac{1}{\mu}  \ang{[\La_0 \La Q]_{\U \la}  \mid g}} \lesssim \s \| r/ \la \La_0 \La Q_{\U \la}\|_{L^2} \| r^{-1}g \|_{L^2} \lesssim \s \|g \|_{H}
 }
 To estimate the first recall  that 
 \EQ{
 \p_\mu U(\mu, \s) = -\frac{1}{\mu} \La U(\mu, \s) = -\s  \La Q_{\U{\s \mu}} +  \mu_c(\la_c^{-1}(\s)) \La Q_{\U{\mu_c( \la^{-1}_c(\s)) \mu}} -  (\La g_c)( \la_c^{-1}(\s))_{\U \mu}  
 }
 and hence 
 \EQ{ %\label{eq:laQpmuU}
 \abs{\ang{ \La Q_{\U \la} \mid \p_\mu U( \mu, \s) }} &\lesssim  \s \| \La Q \|_{L^2}^2 +  \abs{ \ang{ \La Q_{\U \s} \mid  \La Q_{\U{\mu_c( \la^{-1}_c(\s)) }}}} + \abs{ \ang{ \Lam Q_{\U \s} \mid  (\La g_c)( \la_c^{-1}(\s))}} \\
 & \lesssim  \s 
 }
 as claimed. 
\end{proof} 
With the estimates in Claim~\ref{c:Mest} in hand, we see that we can invert $M$  as  long as $\| \bs g \|_{\HH}$ and $\s$ are  small enough and solve for $(\mu', \mu \s' + \xi(\s))$ in~\eqref{eq:Msys}. This yields, 
\EQ{
 \mu' = \left[ \frac{1}{M_{11}M_{22}} + O( \s) \right]\left( M_{22}  B_1 - M_{12} B_2 \right)
}
From Claim~\ref{c:Mest} and~\eqref{eq:B-est} we conclude that 
\EQ{
\abs{\mu'} \lesssim \|   \dot  g \|_{L^2}   + \s^{\frac{k}{2}} \| g \|_H
}
which proves~\eqref{eq:dtmu}. Similarly, 
\EQ{\label{eq:mus'xi} 
 \mu \s' +  \xi(\s) = \left[ \frac{1}{M_{11} M_{22}} + O(  \s)  \right] \left( M_{11} B_2 - M_{21} B_1 \right)
}
Therefore, on the one hand we can conclude from Claim~\ref{c:Mest} and~\eqref{eq:B-est} that 
\EQ{
\abs{\mu \s' +  \xi(\s) } \lesssim \| \dot g \|_{L^2} + \s^{\frac{k}{2}} \|g \|_{H}
}
In fact, extracting the leading order from the right-hand-side of~\eqref{eq:mus'xi} we deduce that 
\EQ{
\abs{\mu \s' +  \xi(\s) + \frac{1}{\| \La Q \|_{L^2}^2} \ang{ \La Q_{\U \la} \mid \dot g} }  \lesssim  o(1)  \| \dot g \|_{L^2}  +  \s^{\frac{k}{2}} \|g \|_{H}
}
proving~\eqref{eq:dtsig}. 
\end{proof}

\section{The poof of uniqueness} 

In this section we complete the proof of Theorem~\ref{t:main}. 

\subsection{An outline of the proof of Theorem~\ref{t:main} } \label{s:outline} 

We begin with  a short outline of end of the proof of Theorem~\ref{t:main}. The purpose is to motivate the computations performed in the next subsection.

 Let $\bs u(t) \in \HH$ be any $2$-bubble in forward time as in~\eqref{eq:2-bub-def} on the time interval $[T_0, \infty)$. By taking $T_0>0$ large enough we may apply Lemma~\ref{l:mod} on the time interval $[T_0, \infty)$, obtaining a decomposition
\EQ{
\bs u(t) = \bs U(\mu(t), \s(t)) + \bs g(t)
}
as in Lemma~\ref{l:mod}. 
By the local Cauchy theory, it will suffice to find a single time $t \ge T_0$ for which we have  $\| \bs g(t)\|_{\HH} =0$. 
The starting point is the following Taylor expansion of the conserved energy about the constructed  trajectory $\bs U(\mu, \s)$.  For each time $t \ge T_0$ we have 
\EQ{
2 \E(\bs Q) &= \E(  \bs u(t)) = \E( \bs U(\mu(t), \s(t)) +\bs g(t))  \\
& =  \E( \bs U(\mu(t), \s(t))) + \ang{ D \E(\bs U(\mu(t), \s(t))) \mid \bs g(t)}   \\
& \quad +  \ang{ D^2 \E(\bs U(\mu(t), \s(t)))  \bs g(t) \mid \bs g(t)}  + o( \| \bs g \|_{\HH}^2) \\
& = 2 \E(\bs Q)  +  \ang{ D \E(\bs U(\mu(t), \s(t))) \mid \bs g(t)}   \\
& \quad +  \ang{ D^2 \E(\bs U(\mu(t), \s(t)))  \bs g(t) \mid \bs g(t)}  + o( \| \bs g \|_{\HH}^2)
}
Subtracting $2 \E( \bs Q)$ from both sides,  recalling the coercivity estimate from Lemma~\ref{l:mod}, i.e.,~\eqref{eq:coer}, and making the ``little oh" term above smaller than half the coercivity constant $c_1>0$ (which is possible by taking $T_0>0$ large enough)  we arrive at the inequality 
\EQ{ \label{eq:enexp1} 
 0 \ge  \ang{ D \E(\bs U(\mu(t), \s(t))) \mid \bs g(t)}  + \frac{1}{2} c_1 \| \bs g(t) \|_{\HH}^2
}
We will show there is necessarily a time $T_1  \ge T_0$ such that 
\EQ{\label{eq:enexp2} 
\ang{ D \E(\bs U(\mu(T_1), \s(T_1))) \mid \bs g(T_1)} \ge -\frac{1}{4} c_1 \| \bs g(T_1) \|_{\HH}^2
} 
which together with~\eqref{eq:enexp1} would imply that $\| \bs g(T_1)\|_{\HH}  =  0$ and thus 
\EQ{ \label{eq:T1-property} 
\bs u(T_1) = \bs U(\mu(T_1), \s(T_1)) = (u_c( \lam_c^{-1}(\s(T_1)), \cdot/ \mu(T_1)), \mu(T_1)^{-1} \p_t u_c( \lam_c^{-1}(\s(T_1)), \cdot/ \mu(T_1))), 
}
which would prove Theorem~\ref{t:main}. In the next section we analyze   the dynamics of 
\EQ{ \label{eq:DEUg} 
\ang{ D \E(\bs U(\mu(t), \s(t))) \mid \bs g(t)}. 
}
with the goal of proving~\eqref{eq:enexp2}.

\subsection{Analysis of the refined instability component} 
We now come to the heart of the argument. For each $\mu, \s >0$ define  
\EQ{ \label{eq:be-def} 
\bs \be( \mu, \s) :=  \frac{1}{ \rho_k  \s^{\frac{k}{2}}} \uD \E( \bs U(\mu, \s)) = \frac{1}{ \rho_k  \s^{\frac{k}{2}}} \pmat{   -\De U(\mu, \s) + r^{-2} f(U(\mu, \s)) \\  \dot U( \mu, \s) } 
}
We make a few comments on how to think of $\bs \be(\mu, \s)$. Recall that 
\EQ{
\p_\s \bs U(\mu, \s) = - \frac{\mu}{\xi(\s)} J \circ \uD \E(\bs U(\mu, \s) ) 
}
Since $\rho_k  \s^{\frac{k}{2}}   \simeq \xi( \s)$, we see that  $\bs \be (\mu, \s)$ is essentially a $90$-degree rotation of $\p_\s \bs U(\mu, \s)$, rescaled by $\mu^{-1}$ i.e., 
\EQ{
\bs \be(\mu, \s)  = \frac{1}{ \rho_k  \s^{\frac{k}{2}}} \uD \E( \bs U(\mu, \s)) \simeq  \frac{1}{\mu} J \circ  \p_\s  \bs U(\mu, \s)
}
Next, consider the coefficient of the projection of $\bs g(t)$ onto $\bs \be (\mu, \s)$, modified by a small ``virial'' type correction term.  
\EQ{ \label{eq:bdef} 
b(t): &=  \ang{ \bs\be(\mu(t), \s(t)) \mid  \bs g(t)}   + \ang{  \calA_0(\mu(t) \s(t)) g(t) \mid \dot g(t)}  \\ 
& =  \frac{1}{\rho_k \s^{\frac{k}{2}}} \ang{ D \E(\bs U(\mu(t), \s(t))) \mid \bs g(t)} + \ang{  \calA_0(\mu(t) \s(t)) g(t) \mid \dot g(t)} 
}
The correction is intended to produce cancellations of terms of critical size, but indeterminate sign, when we compute $b'(t)$ below.

The basic lemma for the family $\bs \be(\mu, \s)$ is the following. 
\begin{lem}  \label{l:be-est} 
The family of functionals $\be(\mu, \s)$ is uniformly bounded in $\HH^*$. In fact, we have the estimates, 
\EQ{ \label{eq:beest} 
\abs{\ang{ \bs\be(\mu, \s)  - ( 0, \La Q_{\U{\mu \s}})  \mid  \bs h}}  = o(1)  \| \bs h \|_{\HH} \mas \s \to 0 .
}
for all $ \bs h \in \HH$. 
%uniformly in $\bs h \in \HH$. 
In particular,
\EQ{ \label{eq:best} 
\abs{ b(t) -  \ang{ \La Q_{\U{\s \mu}}, \mid \dot g}}  = o( 1) \| \bs g\|_{\HH} \mas  \s \to 0 .
}
We also have the estimate, 
\EQ{ \label{eq:beest2} 
 \| \bs\be(\mu, \s)  - ( 0, \La Q_{\U{\mu \s}})  \|_{L^2 \times H} = o(1) \frac{1}{\mu \s} \mas  \s \to 0
}
%\Red{check signs above \dots }
\end{lem} 

\begin{rem}From~\eqref{eq:beest} we see that to leading order 
\EQ{
\bs \be(\mu, \s) \simeq (0,   \La Q_{\U{\mu \s}}) 
}
and is thus $b(t)$ is closely related to the quantity that is also called $b(t)$ in~\cite{JL1}. It is also related to the \emph{refined unstable component} from~\cite{JJ-Pisa}.  
\end{rem} 

\begin{proof}[Proof of Lemma~\ref{l:be-est}]  
The proof of~\eqref{eq:beest} and hence also of~\eqref{eq:best} are direct consequences of the definition of $\bs \be(\mu, \s)$ in~\eqref{eq:be-def} and the estimates~\eqref{eq:dotU1} and~\eqref{eq:eqU} from Corollary~\ref{c:U-est}. The estimate~\eqref{eq:beest2} follows from \eqref{eq:dotU-H} and~\eqref{eq:eqU-L2}. 
\end{proof} 

\begin{prop}\label{p:b'}    Let $\bs u(t) \in \HH$ be a two-bubble in forward time and define $b(t)$ as in~\eqref{eq:bdef}. For any $c_0>0$ there exists $T_0>0$ such that 
\EQ{ \label{eq:b'} 
b'(t) \le  \frac{k \rho_k}{2 \mu(t) \s(t)} \s(t)^{\frac{k}{2}} b(t)  + c_0 \frac{1}{\mu(t) \s(t)}  \left( \abs{b(t)} \s(t)^{\frac{k}{2}} +  \| \bs g(t) \|_{\HH}^2\right)
} 
holds uniformly on the time interval $[T_0, \infty)$. 
\end{prop} 

The main application of Proposition~\ref{p:b'} is the following corollary. 

\begin{cor} \label{c:T1} Suppose that $\bs u(t) \in \HH$ is a two-bubble in forward time. There exists $T_0>0$ with the following property. For every $\eps_1>0$ there exists $ T_1>0$ with $ T_1 \in [T_0, \infty)$ such that 
\EQ{
 \s(T_1)^{\frac{k}{2}} b(T_1) \ge - \eps_1 \| \bs g(T_1) \|_{\HH}^2
}
\end{cor} 
\begin{proof}[Proof of Corollary~\ref{c:T1} assuming Proposition~\ref{p:b'}]
Note that if $ \|\bs g(t)  \|_{\HH} = 0$ for any $t$, then we have $\bs u(t)  \equiv \bs U( \mu(t), \s(t))$ as claimed by Theorem~\ref{t:main} and there is nothing to do. So we may assume that $ \|\bs g(t)  \|_{\HH}  > 0$ for all $t$ for which Lemma~\ref{l:mod} applies (i.e., all sufficiently large $t>0$). 
Suppose Corollary~\ref{c:T1} fails. Fixing a sufficiently large $T_0'$ as in Lemma~\ref{l:mod} there  there exists $c_2 >0$ so that for all $t \ge T_0'$ we have  
\EQ{ \label{eq:bcont} 
 \s(t)^{\frac{k}{2}} b(t) \le  -c_2 \,   \| \bs g(t) \|_{\HH}^2
}
By Proposition~\ref{p:b'} we can choose $T_0 \ge T_0'$ sufficiently large in order to find a uniform constant $c_3 >0$ for which 
\EQ{
b'(t) \le - c_3 \frac{\| \bs g(t) \|_{\HH}^2}{\mu(t) \s(t)} \quad \forall \, \, t \in [T_0, \infty)
}
But this implies that $b'(t) <0$ on the entire interval $[T_0, \infty)$. By~\eqref{eq:bcont} we also have $b(t) <0$ for all $t \in [T_0, \infty)$. But these two conditions are impossible since we know that $b(t) \to 0 $ as $t \to \infty$. 
%Now since $\| \bs g(t) \|_{\HH_0} \to 0$ as $t \to \infty$ we can find $T_2 \ge T_0$ such that 
%\EQ{
%\| \bs g(T_2) \|_{\HH_0}  = \sup_{t \in [T_2, \infty)} \|  \bs g(t) \|_{\HH_0}
%}
%Then we have 
\end{proof} 

\begin{proof}[Proof of Proposition~\ref{p:b'}]
We compute 
\EQ{
b'(t)  &= \frac{\ud}{\ud t} \ang{   \bs \be(\mu(t), \s(t)) \mid \bs g(t)}  + \frac{\ud}{\ud t}\ang{  \calA_0(\mu(t) \s(t)) g(t) \mid \dot g(t)}  \\
& =   \left( \frac{\ud}{\ud t}\frac{1}{ \rho_k \s^{\frac{k}{2}}} \right) \ang{ D\E( \bs U(\mu, \s)) \mid \bs g(t) }   + \frac{1}{ \rho_k \s^{\frac{k}{2}}} \frac{\ud}{\ud t}\ang{ D\E( \bs U(\mu, \s)) \mid \bs g(t) }  \\
&\quad  + \frac{\ud}{\ud t}\ang{  \calA_0(\mu(t) \s(t)) g(t) \mid \dot g(t)}
}
%\EQ{
%b'(t)  &= \ang{ \frac{\ud}{\ud t}  \bs \be(\mu(t), \s(t)) \mid \bs g(t)}  + \ang{   \bs \be(\mu(t), \s(t)) \mid \frac{\ud}{\ud t}\bs g(t)}  \\
%& \quad  -  \ang{ \frac{\ud}{\ud t} [\calA_0(\mu(t) \s(t)) g(t)] \mid \dot g(t)}  - \ang{  \calA_0(\mu(t) \s(t)) g(t) \mid  \frac{\ud}{\ud t} \dot g(t)}  
%}
%We begin by computing the time derivative of $\be(\mu, \s)$. From the definition of $\bs \be$ we have 
%\EQ{
%\frac{\ud}{\ud t}  \bs \be(\mu(t), \s(t)) = \frac{\ud}{\ud t}   \left( \frac{1}{ \rho_k \s^{\frac{k}{2}}} \right) D\E( \bs U(\mu, \s))  + \frac{1}{ \rho_k \s^{\frac{k}{2}}} D^2 \E(\be U(\mu, \s)) \p_t [\bs U(\mu, \s) ]
%}
The first leading term on the right-hand-side of~\eqref{eq:b'} comes from differentiating $\s^{-\frac{k}{2}}$ above. Indeed by~\eqref{eq:dtsig} we have 
\EQ{
  \left(  \frac{\ud}{\ud t} \frac{1}{ \rho_k \s^{\frac{k}{2}}} \right)\ang{ D\E( \bs U(\mu, \s)) \mid \bs g(t) } &= -\frac{k}{2} \frac{\s'}{\s}  \frac{1}{\rho_k \s^{\frac{k}{2}}}\ang{ D\E( \bs U(\mu, \s)) \mid \bs g(t) } = -\frac{k}{2}  \frac{\s'}{\s}  b(t) \\
% & = -\frac{k}{2} \frac{\s'}{\s} \frac{1}{\rho_k \s^{\frac{k}{2}}}\ang{ D\E( \bs U(\mu, \s)) \mid \bs g(t) }  =  \\
 & = \frac{k}{2} \left( \frac{\xi( \s)}{\mu \s} + \frac{\ang{ \La Q_{\U{\mu \s}} \mid \dot g}}{ \mu \s \| \La Q \|_{L^2}^2 } \right)b(t) + o(1) O\left(  \frac{\| \bs g(t)\|_{\HH}^2}{ \mu \s}\right) 
 }
Using~\eqref{eq:xis} and~\eqref{eq:best}, and the fact that $\s(t) \to 0 \mas t \to \infty$ we conclude that 
\EQ{
 \left(  \frac{\ud}{\ud t} \frac{1}{ \rho_k \s^{\frac{k}{2}}} \right)\ang{ D\E( \bs U(\mu, \s)) \mid \bs g(t) } & =  \frac{k \rho_k}{2 \mu(t) \s(t)} \s(t)^{\frac{k}{2}} b(t)  +  \frac{k}{2} \frac{\ang{ \La Q_{\U{\mu \s}} \mid \dot g}^2}{ \mu(t) \s(t) \| \La Q \|_{L^2}^2 }   \\
 & \quad + \frac{1}{\mu(t) \s(t)} o \left( \abs{b(t)} \s(t)^{\frac{k}{2}} +  \| \bs g(t) \|_{\HH}^2\right)
}
Hence Proposition~\ref{p:b'} follows from the estimate 
\EQ{  \label{eq:rest} 
 \frac{k\ang{ \La Q_{\U{\mu \s}} \mid \dot g}^2}{2 \mu \s \| \La Q \|_{L^2}^2 }  +  \frac{1}{ \rho_k \s^{\frac{k}{2}}} \frac{\ud}{\ud t}\ang{ D\E( \bs U(\mu, \s)) \mid \bs g}   + \frac{\ud}{\ud t}\ang{  \calA_0(\mu \s) g \mid \dot g}    \le o(1) O\Big( \frac{\| \bs g \|_{\HH}^2}{\mu\s} \Big)
}
We begin the proof of~\eqref{eq:rest} by expanding the second term on the left above using~\eqref{eq:eq-g-ham}. 
\EQ{
 \frac{\ud}{\ud t}\ang{ D\E( \bs U(\mu, \s)) \mid \bs g} &=   \ang{ D^2 \E(\bs U(\mu, \s))  \p_t [\bs U( \mu, \s)] \mid \bs g}   +  \ang{ D\E( \bs U(\mu, \s)) \mid  \p_t \bs g} \\
& =   \ang{ D^2 \E(\bs U(\mu, \s)) [ \p_t \bs U( \mu, \s)] \mid \bs g} \\
& \quad  + \ang{ D\E( \bs U(\mu, \s)) \mid  \p_t \bs u}  - \ang{ D\E( \bs U(\mu, \s)) \mid  \p_t [\bs U(\mu, \s)]}  \\
&  =  \ang{ D^2 \E(\bs U(\mu, \s)) [ \p_t \bs U( \mu, \s)] \mid \bs g}  \\
&\quad   + \ang{ D\E( \bs U(\mu, \s)) \mid  J \circ D\E(\bs U(\mu, \s) + \bs g) }  \\
& \quad  - \mu' \ang{ D\E( \bs U(\mu, \s)) \mid  \p_\mu \bs U(\mu, \s)}    -  \s' \ang{ D\E( \bs U(\mu, \s)) \mid  \p_\s \bs U(\mu, \s)}  
}
The fact that $\E( \bs U(\mu, \s))$ is constant in $\mu, \s$ implies the last two lines above $\equiv 0$ since 
\EQ{
 0 = \frac{\ud}{\ud \mu} \E( \bs U(\mu, \s)) =\ang{ D\E (\bs U(\mu, \s) \mid \p_\mu \bs U(\mu, \s)} \\
  0 = \frac{\ud}{\ud \s} \E( \bs U(\mu, \s)) = \ang{ D\E (\bs U(\mu, \s) \mid \p_\mu \bs U(\mu, \s)} 
}
Then, subtracting $ 0 = \ang{ D\E( \bs U(\mu, \s)) \mid J \circ D\E \bs U(\mu, \s)} $ we obtain 
\EQ{ \label{eq:dtEUg1} 
 \frac{\ud}{\ud t}\ang{ D\E( \bs U(\mu, \s)) \mid \bs g} &=\ang{ D^2 \E(\bs U(\mu, \s)) [ \p_t \bs U( \mu, \s)] \mid \bs g}  \\
& \quad  + \ang{ D\E( \bs U(\mu, \s)) \mid  J \circ \big[D\E(\bs U(\mu, \s) + \bs g)  - D \E(\bs U(\mu, \s)) \big]}
}
Next, we re-write the first term above as follows. Recall that by~\eqref{eq:psU} we have 
\EQ{
\p_t \bs U(\mu, \s) &= \mu' \p_\mu \bs U(\mu, \s) + \s' \p_\s \bs U (\mu, \s)  \\
%& = \mu' \p_\mu \bs U(\mu, \s) + \s' \p_\s \bs U (\mu, \s)  \\
%& \quad + \frac{\xi(\s)}{\mu} \p_\s \bs U(\mu, \s) - \frac{\xi(\s)}{\mu} \p_\s \bs U(\mu, \s) \\
&  = \mu' \p_\mu \bs U(\mu, \s) + \left(\s' + \frac{\xi(\s)}{\mu}  \right) \p_\s \bs U (\mu, \s)   + J \circ D \E(\bs U(\mu, \s))
}
Hence, using also the self-adjointness of $D^2 \E(\bs U(\mu, \s))$  and the skew-symmetry of $J$  we have 
\EQ{
\ang{ D^2 \E(\bs U(\mu, \s)) [ \p_t \bs U( \mu, \s)] \mid \bs g} & = \ang{ D^2 \E(\bs U(\mu, \s)) [ \p_t \bs U( \mu, \s) - J \circ D\E(\bs U(\mu, \s))] \mid \bs g} \\
& \quad + \ang{ D^2 \E(\bs U(\mu, \s)) J \circ D\E(\bs U(\mu, \s))\mid \bs g} \\
& =  \mu' \ang{ D^2 \E(\bs U(\mu, \s))\p_\mu \bs U(\mu, \s) \mid \bs g}  \\
&\quad + \left(\s' + \frac{\xi(\s)}{\mu}  \right) \ang{ D^2 \E(\bs U(\mu, \s) \p_\s \bs U (\mu, \s)\mid \bs g} \\
& \quad - \ang{  D\E(\bs U(\mu, \s))\mid J \circ D^2 \E(\bs U(\mu, \s))\bs g}
}
Inserting this back into~\eqref{eq:dtEUg1} we obtain
\begin{multline} 
 \frac{\ud}{\ud t}\ang{ D\E( \bs U(\mu, \s)) \mid \bs g}  \\
= \ang{ D\E( \bs U(\mu, \s)) \mid  J \circ \big[D\E(\bs U(\mu, \s) + \bs g)  - D \E(\bs U(\mu, \s)) - D^2\E(\bs U(\mu, \s) \bs g \big]} \\
  +  \mu' \ang{ D^2 \E(\bs U(\mu, \s))\p_\mu \bs U(\mu, \s) \mid \bs g}  
+ \left(\s' + \frac{\xi(\s)}{\mu}  \right) \ang{ D^2 \E(\bs U(\mu, \s) \p_\s \bs U (\mu, \s)\mid \bs g}  
%& =  \ang{ (0, - \La Q_{\U \mu \s}) \mid  J \circ \big[D\E(\bs U(\mu, \s) + \bs g)  - D \E(\bs U(\mu, \s)) - D^2\E(\bs U(\mu, \s) \bs g \big]} \\
%& \quad +  \ang{ \big[D\E( \bs U(\mu, \s))  - (0, -\La Q_{\U \mu \s})\big]\mid  J \circ \big[D\E(\bs U(\mu, \s) + \bs g)  - D \E(\bs U(\mu, \s)) - D^2\E(\bs U(\mu, \s) \bs g \big]} \\
% & \quad  + \mu' \ang{ D^2 \E(\bs U(\mu, \s))\p_\mu \bs U(\mu, \s) \mid \bs g}  \\
%&\quad + \left(\s' + \frac{\xi(\s)}{\mu}  \right) \ang{ D^2 \E(\bs U(\mu, \s) \p_\s \bs U (\mu, \s)\mid \bs g} 
\end{multline} 
Finally multiplying by~$ \s^{-\frac{k}{2}} \rho_k^{-1}$ and preparing for an application of~\eqref{eq:beest2} we obtain
\EQ{ \label{eq:recap1} 
&\frac{1}{ \rho_k \s^{\frac{k}{2}}}  \frac{\ud}{\ud t}\ang{ D\E( \bs U(\mu, \s)) \mid \bs g}  \\
&  =  \ang{ \bs \be(\mu, \s)  - (0, \La Q_{\U{\mu \s}}) \mid  J \circ \big[D\E(\bs U(\mu, \s) + \bs g)  - D \E(\bs U(\mu, \s)) - D^2\E(\bs U(\mu, \s)) \bs g \big]} \\ 
&\quad + \ang{ (0,  \La Q_{\U{\mu \s}}) \mid  J \circ \big[D\E(\bs U(\mu, \s) + \bs g)  - D \E(\bs U(\mu, \s)) - D^2\E(\bs U(\mu, \s) \bs g \big]} \\
 &\quad + \frac{1}{ \rho_k \s^{\frac{k}{2}}}\mu' \ang{ D^2 \E(\bs U(\mu, \s))\p_\mu \bs U(\mu, \s) \mid \bs g}   \\
 &\quad  + \frac{1}{ \rho_k \s^{\frac{k}{2}}}\left(\s' + \frac{\xi(\s)}{\mu}  \right) \ang{ D^2 \E(\bs U(\mu, \s) \p_\s \bs U (\mu, \s)\mid \bs g} 
} 
Consider the first line in~\eqref{eq:recap1}. Using~\eqref{eq:beest2} we have, 
\begin{multline} 
\Big| \ang{ \bs \be(\mu, \s)  - (0, \La Q_{\U{\mu \s}}) \mid  J \circ \big[D\E(\bs U(\mu, \s) + \bs g)  - D \E(\bs U(\mu, \s)) - D^2\E(\bs U(\mu, \s)) \bs g \big]}\Big|  \\
\lesssim o(1) \frac{ \| \bs  g \|_{\HH}^2}{ \mu \s}
\end{multline} 
Next, consider the second term on the right in~\eqref{eq:recap1}. Note that, 
\EQ{ \label{eq:crit-term} 
 \ang{ (0,  \La Q_{\U{\mu \s}}) \mid  J \circ \big[D\E(\bs U(\mu, \s) + \bs g)  - D \E(\bs U(\mu, \s)) - D^2\E(\bs U(\mu, \s) \bs g \big]} \\
  = - \frac{1}{\mu \s} \ang{ \Lam Q_{\mu \s} \mid r^{-2} \big( f( U(\mu, \s) + g) - f( U(\mu, \s)) - f'( U( \mu, \s)) g \big)}
}
We write, 
\EQ{
f( U(\mu, \s) + g) - f( U(\mu, \s)) &- f'( U( \mu, \s)) g  \\
&\quad = \frac{1}{2} f''( Q_{\mu \s}) g^2 + \frac{1}{2} \big(f''(U(\mu, \s) ) - f''(Q_{\mu \s}) \big)g^2  + O( \abs{g}^3) 
}
One can readily show using $U(\mu, \s) = \Phi(\mu, \s) + w_c( \mu, \s)$, the definition of $\Phi(\mu, \s)$ and the estimates~\eqref{eq:xis},~\eqref{eq:muc-lim},~\eqref{eq:ac-lim} and ~\eqref{eq:wc-mus-H} that the last two terms above contribute negligible errors, i.e., errors of size $o(1) (\mu \s)^{-1} \| g \|_{H}^2$. Hence, 
\begin{multline}
- \frac{1}{\mu \s} \ang{ \Lam Q_{\mu \s} \mid r^{-2} \big( f( U(\mu, \s) + g) - f( U(\mu, \s)) - f'( U( \mu, \s)) g \big)} \\
= - \frac{1}{\mu \s} \frac{1}{2} \ang{ \Lam Q_{\mu \s}  \mid r^{-2} f''( Q_{\mu \s}) g^2} + o(1) O\Big( \frac{ \| g \|_H^2}{\mu \s} \Big)
\end{multline} 
We integrate by parts in the first term as follows, 
\EQ{
-\frac{1}{2} \ang{ \Lam Q_{\mu \s}  \mid r^{-2} f''( Q_{\mu \s}) g^2} &=- \frac{1}{2} \int_0^\infty  \p_r( f'(Q_{\mu \s}) - k^2) g^2 \, \ud r  \\
&\quad  =  \int_0^\infty  r^{-2} ( f'(Q_{\mu \s}) - k^2) g \Lam g  \,r \,  \ud r =  \ang{ \Lam g \mid P_{\mu \s} g}   \\
&\quad =  \ang{ \Lam_0 g \mid P_{\mu \s} g} - \ang{ g \mid P_{\mu \s} g} 
}
where $P_{\mu \s}$ is as in~\eqref{eq:P-def}. Plugging all of this back into~\eqref{eq:crit-term} we have show that  
\EQ{
\ang{ (0,  \La Q_{\U{\mu \s}}) \mid  J \circ \big[D\E(\bs U(\mu, \s) + \bs g)  - D \E(\bs U(\mu, \s)) - D^2\E(\bs U(\mu, \s) \bs g \big]}  \\
 = \frac{1}{\mu \s} \ang{ \Lam_0 g \mid P_{\mu \s} g} - \frac{1}{\mu \s} \ang{ g \mid P_{\mu \s} g} +  o(1) O\Big( \frac{ \| g \|_H^2}{\mu \s} \Big)
}
Next, applying the estimates~\eqref{eq:dtmu} and~\eqref{eq:D2mu} we have, 
\EQ{
\Big|  \frac{\mu'}{ \rho_k \s^{\frac{k}{2}}} \ang{ D^2 \E(\bs U(\mu, \s))\p_\mu \bs U(\mu, \s) \mid \bs g} \Big| \lesssim o(1)  \frac{\| \bs g \|_{\HH}^2}{\mu \s} 
}
which takes care of the third term on the right-hand side of~\eqref{eq:recap1}. Next, consider the last line of~\eqref{eq:recap1}. Using the estimate~\eqref{eq:D2sig} followed by~\eqref{eq:dtsig} we have, 
\EQ{
\frac{1}{ \rho_k \s^{\frac{k}{2}}}&\left(\s' + \frac{\xi(\s)}{\mu}  \right) \ang{ D^2 \E(\bs U(\mu, \s) \p_\s \bs U (\mu, \s)\mid \bs g} = \frac{1}{ \rho_k \s^{\frac{k}{2}}}\left(\s' + \frac{\xi(\s)}{\mu}  \right)\gamma_k \frac{\s^{\frac{k}{2}}}{  \rho_k \s} \ang{\Lam Q_{\U{\mu \s}}  \mid \dot g} \\
&\quad -\frac{1}{ \rho_k \s^{\frac{k}{2}}}\left(\s' + \frac{\xi(\s)}{\mu}  \right) \frac{ \rho_k \s^{\frac{k}{2}}}{ \s} \ang{ \Lam_0 \Lam Q_{\U{\mu \s}} \mid \dot g}   + o(1)O( \frac{\s^{\frac{k}{2}}}{\s} \|  \bs g \|_{\HH})\frac{1}{ \rho_k \s^{\frac{k}{2}}}\left(\s' + \frac{\xi(\s)}{\mu}  \right) \\
%& =  - \frac{1}{\mu \s} \frac{ \gamma_k}{\rho_k^2} \frac{ \ang{ \Lam Q_{\U{\mu \s}} \mid \dot g }^2 }{\| \La Q \|_{L^2}^2}  -\left(\frac{\s'}{\s} + \frac{\xi(\s)}{\mu \s}  \right) \ang{ \Lam_0 \Lam Q_{\U{\mu \s}} \mid \dot g}  + o(1) O \Big( \frac{ \| \bs g \|_{\HH}^2}{\mu \s} \Big)  \\
& =  - \frac{1}{\mu \s} \frac{ k}{2}  \frac{ \ang{ \Lam Q_{\U{\mu \s}} \mid \dot g }^2 }{\| \La Q \|_{L^2}^2}   -\left(\frac{\s'}{\s} + \frac{\xi(\s)}{\mu \s}  \right) \ang{ \Lam_0 \Lam Q_{\U{\mu \s}} \mid \dot g}  + o(1) O \Big( \frac{ \| \bs g \|_{\HH}^2}{\mu \s} \Big)
}
where in the last line we used~\eqref{eq:q-rho-ga}, i.e., $\gamma_k =  \rho_k^2 \frac{k}{2}$. To recap, by inserting the previous three estimates into~\eqref{eq:recap1}  we have now shown that, 
\EQ{ \label{eq:recap2} 
\frac{1}{ \rho_k \s^{\frac{k}{2}}} & \frac{\ud}{\ud t}\ang{ D\E( \bs U(\mu, \s)) \mid \bs g}  = - \frac{1}{\mu \s} \frac{ k}{2}   \frac{ \ang{ \Lam Q_{\U{\mu \s}} \mid \dot g }^2 }{\| \La Q \|_{L^2}^2}  \\
&   -\left(\frac{\s'}{\s} + \frac{\xi(\s)}{\mu \s}  \right) \ang{ \Lam_0 \Lam Q_{\U{\mu \s}} \mid \dot g}   +  \frac{1}{\mu \s} \ang{ \Lam_0 g \mid P_{\mu \s} g} - \frac{1}{\mu \s} \ang{ g \mid P_{\mu \s} g}\\
& + o(1) O \Big( \frac{\| \bs g \|_{\HH}^2}{\mu \s} \Big) 
}
%To summarize, we have now proved that
%\EQ{ \label{eq:recap3} 
%\frac{1}{ \rho_k \s^{\frac{k}{2}}} & \frac{\ud}{\ud t}\ang{ D\E( \bs U(\mu, \s)) \mid \bs g}  = - \frac{1}{\mu \s} \frac{ k}{2}  \ang{ \Lam Q_{\U{\mu \s}} \mid \dot g }^2     -\left(\frac{\s'}{\s} + \frac{\xi(\s)}{\mu \s}  \right) \ang{ \Lam_0 \Lam Q_{\U{\mu \s}} \mid \dot g}  \\
%& +  \Red{fill in} \\
%& + o(1) O \Big( \frac{\| \bs g \|_{\HH}^2}{\mu \s} \Big) 
%}
Note that the first term on the right above exactly cancels the first term on the left of~\eqref{eq:rest}. The terms on the second line are of critical size, and we now show that the differentiated virial correction will cancel these terms up to admissible errors and a coercive term. Indeed, we claim the estimate, 
\EQ{ \label{eq:virial-mus} 
\frac{\ud}{\ud t}\ang{  \calA_0(\mu \s) g \mid \dot g}  &+ \frac{1}{\mu \s} \ang{ \Lam_0 g \mid P_{\mu \s} g}- \Big( \frac{\s'}{\s} + \frac{\xi(\s)}{\mu \s} \Big)  \ang{ \calA_0(\mu \s)\Lam Q_{\mu \s}  \mid \dot g} \\
& \le  c_0 \frac{ \| g \|_H^2}{ \mu \s}  - \frac{1}{\mu \s} \int_0^{R \mu \s} \Big(  ( \p_r g)^2 + k^2 \frac{g^2}{r^2} \Big) \, r \, \ud r 
}
where, $c_0>0$, $R>0$ are as in Lemma~\ref{l:opA} and $c_0>0$ can be taken as small as we like independent of $\mu \s$, and $P_{\mu \s}$ is as in~\eqref{eq:P-def}. To see this, we expand the derivative of the virial correction as follows. Using the notation, $\la = \mu \s$ we have, 
\EQ{ \label{eq:virial-mus1} 
\frac{\ud}{\ud t}\ang{  \calA_0(\mu \s) g \mid \dot g} = (\frac{ \mu'}{\mu} + \frac{ \s'}{\s}) \ang{ [\lam\p_\la \calA_0](\lam)  g \mid  \dot g} + \ang{  \calA_0(\mu \s) \p_t g \mid \dot g} + \ang{  \calA_0(\mu \s) g \mid \p_t  \dot g}
} 
Note that the first term on the right above contributes an admissible error. Indeed, using the estimates~\eqref{eq:dtmu},~\eqref{eq:dtsig},~\eqref{eq:xis}, and the first bullet point in Lemma~\ref{l:opA} we have, 
\EQ{
\abs{(\frac{ \mu'}{\mu} + \frac{ \s'}{\s}) \ang{ [\lam\p_\la \calA_0](\lam)  g \mid  \dot g}}  & \lesssim \abs{\frac{ \mu'}{\mu} + \frac{ \s'}{\s}} \| \bs g \|_{\HH}^2  \lesssim ( \s^{\frac{k}{2}} +  \| \bs g \|_{\HH}) \frac{ \| \bs g \|_{\HH}^2}{\mu \s}
}
Next, we expand the second two terms on the right of~\eqref{eq:virial-mus1} using the equation satisfied by $\bs g$ in~\eqref{eq:eq-g}. For the second term we have, 
\EQ{ \label{eq:dt-g1} 
 \ang{  \calA_0(\mu \s) \p_t g \mid \dot g} &=  \ang{  \calA_0(\mu \s)  \dot g  \mid \dot g}  - \mu'   \ang{  \calA_0(\mu \s)\p_\mu U( \mu, \s) \mid \dot g}   \\
 &\quad -  ( \s' + \frac{\xi(\s)}{\mu})  \ang{  \calA_0(\mu \s) \p_\s U(\mu, \s) \mid \dot g}
}
Since $\calA_0(\mu \s) $ is antisymmetric, we have, 
$
 \ang{  \calA_0(\mu \s)  \dot g  \mid \dot g} = 0. 
$
For the second term on the right above we use~\eqref{eq:pmuU}  and the first bullet point in Lemma~\ref{l:opA} to deduce that 
\EQ{
\abs{ \mu'   \ang{  \calA_0(\mu \s)\p_\mu U( \mu, \s) \mid \dot g}  }  % \frac{\abs{\mu'}}{\mu}  \| \Lam U(\mu ,\s) \|_{H}  \| \dot g \|_{L^2} 
 \lesssim \frac{\abs{\mu'}}{\mu} ( \| \Lam \Phi(\mu ,\s) \|_{H} + \| \Lam w_c (\mu, \s) \|_{H})  \| \dot g \|_{L^2} \lesssim \frac{ \| \bs g \|_{\HH}^2}{\mu}  \lesssim o(1) \frac{ \| \bs g \|_{\HH}^2}{\mu \s}
}
where the last inequality follows from~\eqref{eq:wc-mus-LamH} and~\eqref{eq:dtmu}. Next, we treat the last term in~\eqref{eq:dt-g1}. Using~\eqref{eq:psU} we write, 
\EQ{
-  ( \s' &+ \frac{\xi(\s)}{\mu})  \ang{  \calA_0(\mu \s) \p_\s U(\mu, \s) \mid \dot g} = \Big( \frac{\s'}{\s} + \frac{\xi(\s)}{\mu \s} \Big)  \ang{ \calA_0(\mu \s) \frac{\mu \s}{\xi(\s)} \dot U(\mu, \s) \mid \dot g} \\
&  = \Big( \frac{\s'}{\s} + \frac{\xi(\s)}{\mu \s} \Big)  \ang{ \calA_0(\mu \s)\Lam Q_{\mu \s}  \mid \dot g}    + ( \mu \s'  + \xi(\s) )  \ang{ \calA_0(\mu \s) \Big(\frac{b_c(\la_c^{-1}(\s))}{\xi(\s)} - 1 \Big) \Lam Q_{\U{\mu \s}} \mid \dot g} \\
&\quad + ( \mu \s'  + \xi(\s) )  \ang{ \calA_0(\mu \s) \frac{1}{\xi(\s)}  \Big( \dot \Phi(\mu, \s) - b_c(\lam_c^{-1}(\s) )  \Lam Q_{\U{\mu \s}} \Big) \mid \dot g} \\
&\quad +( \mu \s'  + \xi(\s) )   \ang{ \calA_0(\mu \s) \frac{1}{\xi(\s)} \dot w_c(\mu, \s) \mid \dot g}  \\
& =   \Big( \frac{\s'}{\s} + \frac{\xi(\s)}{\mu \s} \Big)  \ang{ \calA_0(\mu \s)\Lam Q_{\mu \s}  \mid \dot g}  + o(1) O\Big( \frac{\| \bs g \|_{\HH}^2}{\mu \s} \Big) 
}
where the last line follows from the first bullet point in~Lemma~\ref{l:opA} with ~\eqref{eq:xis},~\eqref{eq:muc-lim},~\eqref{eq:ac-lim}, and~\eqref{eq:dotPhi-bLamQ-H}, and  finally~\eqref{eq:wc-mus-H2}. Plugging the previous three estimates back into~\eqref{eq:dt-g1} we obtain, 
\EQ{
 \ang{  \calA_0(\mu \s) \p_t g \mid \dot g} &= \Big( \frac{\s'}{\s} + \frac{\xi(\s)}{\mu \s} \Big)  \ang{ \calA_0(\mu \s)\Lam Q_{\mu \s}  \mid \dot g}  + o(1) O\Big( \frac{\| \bs g \|_{\HH}^2}{\mu \s} \Big) 
}
Lastly, we expand the term in~\eqref{eq:virial-mus1} involving $\p_t \dot g$ using~\eqref{eq:eq-g}. Preparing for a near identical argument to the one used to treat the virial correction in the companion paper~\cite[Proof of Lemma~4.6]{JL2-regularity} we  write, 
\EQ{ \label{eq:dt-dotg1} 
\ang{  \calA_0(\mu \s) g \mid \p_t  \dot g} & = \ang{  \calA_0(\mu \s) g \mid (- \LL_{\mu \s}) g}  \\
&\quad -  \ang{  \calA_0(\mu \s) g \mid r^{-2} \Big(  f( U(\mu, \s) + g) - f( U(\mu, \s)) - f'(Q_{\mu \s}) g \Big)}  \\
& \quad - \mu'  \ang{  \calA_0(\mu \s) g \mid \dot{ \p_\mu U}(\mu, \s)}  - ( \s' + \frac{\xi(\s)}{\mu}) \ang{  \calA_0(\mu \s) g \mid \dot{\p_\s U}(\mu, \s)}  
}
For the first term on the right of~\eqref{eq:dt-dotg1} we recall the notation $\LL_{\mu \s} = \LL_0 + P_{\mu \s}$ and write, 
\EQ{
 \ang{  \calA_0(\mu \s) g \mid (- \LL_{\mu \s}) g}  = -  \ang{  \calA_0(\mu \s) g \mid  \LL_0 g}  -  \ang{  \calA_0(\mu \s) g \mid P_{\mu \s} g} 
}
It then follows from~\eqref{eq:A-pohozaev} from Lemma~\ref{eq:opA}  along with the estimate~\eqref{eq:vir-new} (to treat the second term above) that, 
\EQ{
 \ang{  \calA_0(\mu \s) g \mid (- \LL_{\mu \s}) g}   \le c_0 \frac{ \| g \|_H^2}{ \mu \s}  - \frac{1}{\mu \s} \int_0^{R \mu \s} \Big(  ( \p_r g)^2 + k^2 \frac{g^2}{r^2} \Big) \, r \, \ud r   - \frac{1}{\mu \s} \ang{  \Lam_0 g \mid P_{\mu \s} g} 
}
Note that  $c_0>0$, $R>0$ are as in Lemma~\ref{l:opA} and $c_0>0$ can be taken as small as we like independent of $\mu \s$. Next, we estimate the second term on the right of~\eqref{eq:dt-dotg1} via an analysis nearly identical to the one used to estimate the second term in~\cite[Eqn. (4.30)]{JL2-regularity}. The difference is that here we can only make use of the $\HH$ regularity of $g$. First, note that by Lemma~\ref{l:opA} we have, 
\EQ{
\Big| &\ang{  \calA_0(\mu \s) g \mid r^{-2} \Big(  f( U(\mu, \s) + g) - f( U(\mu, \s)) - f'(Q_{\mu \s}) g \Big)}  \Big|  \\
& \qquad \qquad \lesssim  \| g \|_H  \| r^{-2} \big(  f( U(\mu, \s) + g) - f( U(\mu, \s)) - f'(Q_{\mu \s}) g \big) \|_{L^2}
}
Hence it suffices to establish the estimate, 
\EQ{ \label{eq:r2f''} 
\| r^{-2} \big(  f( U(\mu, \s) + g) - f( U(\mu, \s)) - f'(Q_{\mu \s}) g \big) \|_{L^2} & \lesssim o(1)  \frac{ \| g \|_{H}}{\mu \s} 
}
To see this, we write,  
\EQ{
f( U(\mu, \s) &+ g) - f( U(\mu, \s)) - f'(Q_{\mu \s}) g  \\
&= f( \Phi(\mu, \s)+ w_c(\mu, \s) + g) - f(\Phi(\mu, \s)+ w_c(\mu, \s)) - f'( \Phi(\mu, \s) + w_c(\mu, s)) g \\
& \quad+  \Big( f'( \Phi(\mu, \s) + w_c(\mu, s))  - f'( \Phi(\mu, \s) )\Big) g  + \Big( f'( \Phi) - f'(Q_{\mu \s}) \Big) g
}
The contribution of the first line is handled using the pointwise estimate, 
\EQ{
\frac{1}{r^2} \Big|f( \Phi(\mu, \s)+ w_c(\mu, \s) + g) - f(\Phi(\mu, \s)+ w_c(\mu, \s)) - f'( \Phi(\mu, \s) + w_c(\mu, s)) g \Big| \lesssim \frac{1}{\s\mu} r^{-1} g^2 
}
which follows from the definition of $\Phi(\mu, \s)$,~\eqref{eq:xis}~\eqref{eq:muc-lim},~\eqref{eq:ac-lim} and~\eqref{eq:wc-mus-H}. For the second term we use the pointwise estimate, 
\EQ{
r^{-2} \abs{ \Big( f'( \Phi(\mu, \s) + w_c(\mu, s))  - f'( \Phi(\mu, \s) )\Big) g } \lesssim \frac{1}{\mu \s}  r^{-1}  \abs{ w_c( \mu, \s)}\abs{g} 
}
together with~\eqref{eq:wc-mus-H}. Finally, to treat the last term we note the pointwise estimate, 
\EQ{
r^{-2} \abs{f'( \Phi) - f'(Q_{\mu \s})} \lesssim \frac{1}{\mu} r^{-1}  \lesssim o(1) \frac{1}{ \mu \s} r^{-1} 
}
%which is analogous to~\eqref{eq:f'Phi-f'Q-1} and~\eqref{eq:f'Phi-f'Q-2}. 
This is sufficient to prove~\eqref{eq:r2f''}. Next, we use~\eqref{eq:dtmu} and~\eqref{eq:pmuU} to estimate, 
\EQ{
\abs{ \mu'  \ang{  \calA_0(\mu \s) g \mid  \dot{\p_\mu  U}(\mu, \s)} } \lesssim  \frac{ \| \bs g \|_{H}^2}{\mu}  \Big( \| \Lam_0 \dot \Phi(\mu, \s) \|_{L^2} + \| \Lam_0 w_c(\mu, \s) \|_{L^2} \Big)  \lesssim o(1)  \frac{ \| \bs g \|_{\HH}^2}{\mu \s} 
}
where in the last inequality we used~\eqref{eq:wc-mus-H2}. Lastly, using the first bullet point in Lemma~\ref{l:opA}~\eqref{eq:dtsig} and~\eqref{eq:psU} we have, 
\EQ{
\abs{( \s' + \frac{\xi(\s)}{\mu}) \ang{  \calA_0(\mu \s) g \mid  \dot{\p_\s  U}(\mu, \s)}} &\lesssim  \frac{1}{\mu}  \|  \bs g \|_{\HH}^2  \frac{\mu}{\xi(\s)}  \|   \De U(\mu, \s) -\frac{1}{r^2} f(U(\mu, \s)) \|_{L^2}  \\
& \lesssim  \s^{\frac{k}{2}} \frac{ \| \bs g \|_{\HH}^2}{\mu \s}
}
where the last inequality is by~\eqref{eq:eqU-L2} and~\eqref{eq:xis}. This completes the proof of~\eqref{eq:virial-mus}. 

We can now complete the proof of~\eqref{eq:rest}. Combining the estimates~\eqref{eq:recap2} and~\eqref{eq:virial-mus} we have, 
\EQ{
 \frac{k\ang{ \La Q_{\U{\mu \s}} \mid \dot g}^2}{2 \mu \s \| \La Q \|_{L^2}^2 } & +  \frac{1}{ \rho_k \s^{\frac{k}{2}}} \frac{\ud}{\ud t}\ang{ D\E( \bs U(\mu, \s)) \mid \bs g}   + \frac{\ud}{\ud t}\ang{  \calA_0(\mu \s) g \mid \dot g}  \\
 & \le   c_0 \frac{ \| \bs g \|_{\HH}^2}{ \mu \s}  +  \left(\frac{\s'}{\s} + \frac{\xi(\s)}{\mu \s}  \right)  \Big( \ang{ \calA_0(\mu \s)\Lam Q_{\mu \s}  \mid \dot g}-  \ang{ \Lam_0 \Lam Q_{\U{\mu \s}} \mid \dot g}  \Big)   \\
&  - \frac{1}{\mu \s} \int_0^{R \mu \s} \Big(  ( \p_r g)^2 + k^2 \frac{g^2}{r^2} \Big) \, r \, \ud r - \frac{1}{\mu \s} \ang{ g \mid P_{\mu \s} g} \\
}
Using the estimate~\eqref{eq:L0-A0-wm} along with~\eqref{eq:dtsig} we have, 
\EQ{
\Big| \left(\frac{\s'}{\s} + \frac{\xi(\s)}{\mu \s}  \right)  \ang{ (\calA_0(\mu \s)\Lam Q_{\mu \s}- \Lam_0 \Lam Q_{\U{\mu \s}} )  \mid \dot g} \Big| & \le  \frac{1}{\mu \s} \| \calA_0(\mu \s)\Lam Q_{\mu \s} - \Lam_0 \Lam Q_{\U{\mu \s}} \|_{L^2}  \| \bs g \|_{\HH}^2  \\
& \le c_0 \frac{\| \bs g \|_{\HH}^2}{\mu \s} 
}
Finally, the localized coercivity estimate~\eqref{eq:L-loc1} from Lemma~\ref{l:loc-coerce} yields, 
\EQ{
- \frac{1}{\mu \s} \int_0^{R \mu \s} \Big(  ( \p_r g)^2 + k^2 \frac{g^2}{r^2} \Big) \, r \, \ud r - \frac{1}{\mu \s} \ang{ g \mid P_{\mu \s} g}  \le  c_0 \frac{\| \bs g \|_{\HH}^2}{\mu \s} 
}
by taking $R>0$ large enough. We conclude that, 
\EQ{
 \frac{k\ang{ \La Q_{\U{\mu \s}} \mid \dot g}^2}{2 \mu \s \| \La Q \|_{L^2}^2 }  +  \frac{1}{ \rho_k \s^{\frac{k}{2}}} \frac{\ud}{\ud t}\ang{ D\E( \bs U(\mu, \s)) \mid \bs g}   + \frac{\ud}{\ud t}\ang{  \calA_0(\mu \s) g \mid \dot g} 
  \le   c_0 \frac{ \| \bs g \|_{\HH}^2}{ \mu \s}  
}
where $c_0$ is a constant that can be taken arbitrarily small, independently of $\mu, \s$. This proves~\eqref{eq:rest} and completes the proof of the proposition. 
\end{proof} 

\subsection{The proof of Theorem~\ref{t:main}} 
We put the finishing touches on the the proof of Theorem~\ref{t:main}. 
\begin{proof} 
We pick up where we left off in Section~\ref{s:outline}. Let $\bs u(t) \in \HH$ be any forward-in-time $2$-bubble solution to~\eqref{eq:wmk}  on the time interval $[T_0, \infty)$ where  $T_0>0$ is chosen sufficiently large so that Corollary~\ref{c:T1} holds, as well as~\eqref{eq:enexp1}. Assume for contradiction  that $  \| \bs g(t)  \|_{\HH} >0$ for all $t \ge T_0$.  So on the one hand, by~\eqref{eq:enexp1} we have, 
\EQ{ \label{eq:enexp3} 
0 \ge \ang{ D\E( \bs U(\mu(t), \s(t))) \mid \bs g(t)} + \frac{1}{2} c_1 \| \bs g(t) \|_{\HH}^2. 
}
for all $t  \in [T_0, \infty)$ for a uniform constant $c_1>0$. One the other hand, by Corollary~\ref{c:T1}, the definition of $b(t)$ in~\eqref{eq:bdef}, and by possibly taking $T_0$ larger so that $\s(t)$ is sufficiently small, we can find $T_1 \ge T_0$ so that, 
\EQ{
 \ang{ D\E( \bs U(\mu(T_1), \s(T_1))) \mid \bs g(T_1)} \ge -\frac{c_1}{4}  \| \bs g(T_1) \|_{\HH}^2 
 }
 which yields a contradiction in~\eqref{eq:enexp3} at time $T_1$. Thus, there exists $T \ge T_0$ for which $\|\bs g(T) \|_{\HH}  = 0$. But this means that, 
 \EQ{
 \bs u(T)  = \bs U( \mu(T), \s(T)) = ( u_c( \lam_c^{-1}(\s(T)),  \cdot/ \mu(T)), \mu(T)^{-1} \p_t u_c( \lam_c^{-1}(\s(T)),  \cdot/ \mu(T))), 
 }
i.e. $\bs u(t)$ agrees with $\bs u_c(t)$ up to a fixed time translation and rescaling. This completes the proof. 
\end{proof} 

\bibliographystyle{plain}
\bibliography{researchbib}

\bigskip
\centerline{\scshape Jacek Jendrej}
\smallskip
{\footnotesize
 \centerline{CNRS and LAGA, Université Sorbonne Paris Nord}
\centerline{99 av Jean-Baptiste Cl\'ement, 93430 Villetaneuse, France}
\centerline{\email{jendrej@math.univ-paris13.fr}}
} 
\medskip 
\centerline{\scshape Andrew Lawrie}
\smallskip
{\footnotesize
 \centerline{Department of Mathematics, Massachusetts Institute of Technology}
\centerline{77 Massachusetts Ave, 2-267, Cambridge, MA 02139, U.S.A.}
\centerline{\email{alawrie@mit.edu}}
}

\end{document}